\documentclass[16pt]{amsart}
\usepackage{amssymb, latexsym, amsmath, amsfonts}

\newtheorem{thm}{Theorem}[section]
\newtheorem{cor}[thm]{Corollary}
\newtheorem{lem}[thm]{Lemma}
\newtheorem{prop}[thm]{Proposition}
\theoremstyle{definition}

\theoremstyle{remark}
\newtheorem{rem}[thm]{Remark}
\numberwithin{equation}{section}

\newcommand{\mbf}{\mathbf}
\newcommand{\ra}{\rightarrow}

\newcommand{\pa}{\partial}
\newcommand{\ov}{\overline}

\newcommand{\ep}{\epsilon}
\newcommand{\no}{\noindent}
\newcommand{\Om}{\Omega}
\newcommand{\cal}{\mathcal}

\begin{document}
\title{Some aspects of the Kobayashi and Carath\'{e}odory metrics on pseudoconvex domains}
\thanks{The second author was supported in part by a grant from UGC under DSA-SAP, Phase IV}
\author{Prachi Mittal and Kaushal Verma}
\address{Department of Mathematics,
Indian Institute of Science, Bangalore 560 012, India}
\email{mittal@math.iisc.ernet.in, kverma@math.iisc.ernet.in}

\begin{abstract}
The purpose of this article is to consider two themes both of
which emanate from and involve the Kobayashi and the
Carath\'{e}odory metric. First we study the biholomorphic
invariant introduced by B. Fridman on strongly pseudoconvex
domains, on weakly pseudoconvex domains of finite type in $\mbf
C^2$ and on convex finite type domains in $\mbf C^n$ using the
scaling method. Applications include an alternate proof of the
Wong-Rosay theorem, a characterization of analytic polyhedra with
noncompact automorphism group when the orbit accumulates at a
singular boundary point and a description of the Kobayashi balls
on weakly pseudoconvex domains of finite type in $ \mathbf{C}^2$
and convex finite type domains in $ \mathbf{C}^n $ in terms of
Euclidean parameters. Second a version of Vitushkin's theorem
about the uniform extendability of a compact subgroup of
automorphisms of a real analytic strongly pseudoconvex domain is
proved for $C^1$-isometries of the Kobayashi and
Carath\'{e}odory metrics on a smoothly bounded strongly
pseudoconvex domain.
\end{abstract}

\maketitle

\section{Introduction}

\no Let $X$ be a Kobayashi hyperbolic complex manifold of
dimension $n$ and $\cal H \subset \mbf C^n$ a bounded homogeneous
domain. B. Fridman in \cite{Fridman-1983} constructed an
interesting non-negative continuous function on $X$ whose value at
a given $p \in X$ essentially measures the largest Kobayashi ball
around $p$ that is comparable with $\cal H$. To be more specific
take $\cal H = \mbf B^n$ the unit ball in $\mbf C^n$ and let
$B_X(p, r)$ denote the ball in the Kobayashi metric with radius $r
> 0$ around $p \in X$. Since $X$ is hyperbolic the topology
induced by the Kobayashi metric is equivalent to the intrinsic
topology on it. Thus for small $r > 0$, $B_X(p, r)$ is contained
in a coordinate chart around $p$ and hence there is a
biholomorphic imbedding $f : \mbf B^n \ra X$ such that $f(\mbf
B^n) \supset B_X(p, r)$. Let $\cal R$ denote the set of all $r >
0$ such that there exists a biholomorphic imbedding $f : \mbf B^n
\ra X$ with $f(\mbf B^n) \supset B_X(p, r)$. $\cal R$ is evidently
non-empty as explained before. Define
\begin{equation}
h_X(p, \mbf B^n) = \inf_{r \in \cal R} \frac{1}{r}
\end{equation}
which is a non-negative, real valued function on $X$. If $Y$ is
another such hyperbolic manifold and $g : X \ra Y$ a
biholomorphism then $g$ preserves the balls in the Kobayashi
metrics on $X, Y$ and hence $h_X(p, \mbf B^n) = h_Y(g(p), \mbf
B^n)$ which says that $h_X(p, \mbf B^n)$ is a biholomorphic
invariant. The function $h_X(p, \cal H)$ will henceforth be
referred to as Fridman's invariant and while the same construction
holds for any invariant metric (provided of course when the
intrinsic topology on $X$ coincides with that induced by the
invariant metric), the Kobayashi and the Carath\'{e}odory metric
shall be exclusively dealt with here. In fact, Fridman's invariant
defined using the Carath\'{e}odory metric appears only in two
places one at the end of the section $4$ which considers strongly
pseudoconvex domains where this metric is most tractable and
second in section $6$ where convex domains are considered on which
the Carath\'{e}odory metric equals the Kobayashi metric by
Lempert's work. Barring these exceptions, the Kobayashi metric is
the one that is considered. The choice of metric will be clarified
often enough to avoid ambiguities for the same notation will be
used for this invariant; the choice of $\cal H$ will be made
explicit though. Having defined this let us recall some of its
basic properties proved in \cite{Fridman-1983}. Among other
things, it was shown that when $X$ is Kobayashi hyperbolic then $p
\mapsto h_X(p, \cal H)$ is continuous and if there exists $p^0 \in
X$ such that $h_X(p^0, \cal H) = 0$ then $h_X(p, \cal H) \equiv 0$
for all $p \in X$ and moreover $X$ is biholomorphic to $\cal H$.
This is reminiscent of the $ C/K $ invariant, i.e., the ratio of
the Carath\'{e}odory and the Eisenman-Kobayashi volume form, used
in \cite{Wong-1977}. Indeed, this ratio is at most one for a
bounded domain and equals one at some point if and only if the
domain is equivalent to $ \mathbf{B}^n$. It was also shown that if
$D \subset C^n$ is a $C^3$-smooth strongly pseudoconvex domain
equipped with the Kobayashi metric such that $D$ is not
biholomorphic to $\mbf B^n$, then $h_D(p, \mbf B^n) \ra 0$ as $p
\ra \pa D$. The proof of this crucially used the fact that the
model domain at strongly pseudoconvex boundary points is the ball.
As a consequence if $p^0 \in \pa D$ and $U$ is an open
neighbourhood of $p^0$ then using the automorphism group of $\mbf
B^n$ it is possible to find a sequence of holomorphic mappings
$F^j : U \cap D \ra \mbf B^n$ that are biholomorphic onto their
images which exhaust $\mbf B^n$, i.e., for every compact $K
\subset \mbf B^n$ there is some $m=m(K)$ such that $K \subset
F_m(U \cap D)$.

\medskip

The first goal of this work is to understand Fridman's invariant on a broader class of domains, more specifically the weakly pseudoconvex finite type
domains in $\mbf C^2$ and convex domains of finite type in $\mbf C^n$. The problem is of course to determine the behaviour of $h_D(p, \cal H)$ as $p \ra
\pa D$ and it is natural to try and apply the scaling method to understand this question on more general domains. Any attempt
to do this must take into account the fact that model domains at boundary points are not unique unlike the strongly pseudoconvex case and hence the
boundary limits of $h_D(p, \cal H)$ will depend on the nature of approach to the boundary point. Indeed, in general the model domains for a smooth
weakly pseudoconvex finite type domain in $\mbf C^2$ and convex finite type domains in $\mbf C^n$ are given by
\[
\rho(z) = 2 \Re z_2 + P_{2m}(z_1, \ov z_1)
\]
where $P_{2m}(z_1, \ov z_1)$ is a homogeneous subharmonic polynomial of degree $2m$, and
\[
\rho(z) = 2 \Re z_n + Q_{2m}('z, '\ov z)
\]
where $'z = (z_1, z_2, \ldots, z_{n - 1})$ and $Q_{2m}('z, '\ov
z)$ is a convex polynomial of degree at most $2m$ respectively.
More can be said about the polynomials $P_{2m}(z_1, \ov z_1)$ and
$Q_{2m}('z, '\ov z)$ in case the approach is non-tangential. Since
the definition of $h_D(p, \cal H)$ involves the
Kobayashi/Carath\'{e}odory balls (which are global objects
vis-a-vis the infinitesimal metric), this approach quickly leads
to considerations that involve the Kobayashi/Carath\'{e}odory
balls on the scaled domains and their convergence to the
corresponding balls in the limit domain. Viewed differently, this
is a stability problem for the integrated distance and not its
infinitesimal version about which much is known -- see
\cite{Greene&Krantz-1984} and \cite{Yu-1995} for instance.

\begin{thm}
Let $D \subset \mbf C^n$ be a bounded domain, $p^0 \in \pa D$ and
let $\{p^j\} \subset D$ be a sequence that converges to $p^0$.
Then the boundary behaviour of $h_D(p, \mbf B^n)$ along $\{p^j\}$
can be described in the following cases:

\begin{enumerate}

\item[(i)] If $ D$ is $C^2$-smooth strongly pseudoconvex
equipped with either the Kobayashi or the Carath\'{e}odory metric,
then $h_D(p^j, \mbf B^n) \ra 0$.

\item[(ii)] If $ D$ is $C^{\infty}$-smooth weakly pseudoconvex
of finite type in $ \mathbf{C}^2$ equipped with the Kobayashi
metric, then $h_D(p^j, \mbf B^2) \ra h_{D_{\infty}}((0, -1),
\mathbf B^2)$ where $D_{\infty}$ is the model domain defined by
\[
D_{\infty} = \{ z = (z_1, z_2) \in \mbf C^2 : 2 \Re z_2 +
P_{2m}(z_1, \ov{ z}_1) < 0\},
\]
and $P_{2m}(z_1, \ov {z}_1)$ is a subharmonic polynomial of degree
at most $2m$ ($m \ge 1$) without harmonic terms, $2m$ being the
$1$-type of $\pa D$ near $p^0$.

\item[(iii)] If $D$ is $C^{\infty}$-smooth convex of finite
type equipped with the Kobayashi metric (which equals the
Carath\'{e}odory metric), then $h_D(p^j, \mbf B^n) \ra
h_{D_{\infty}}(('0, -1), \mbf B^n)$ where $D_{\infty}$ is the
model domain defined by
\[
D_{\infty} = \{ z = ('z, z_n) \in \mbf C^n : 2 \Re z_n +
Q_{2m}('z, '\ov z) < 0\},
\]
and $Q_{2m}('z, '\ov z)$ is a non-degenerate convex polynomial of
degree at most $2m$ ($m \ge 1$), $2m$ being the $1$-type of $\pa
D$ near $p^0$.

\end{enumerate}

\end{thm}

\noindent The proof of this theorem is contained in sections 4, 5
and 6. While this is a global result, it turns out that $
h_D(\cdot, \mathbf{B}^n) $ can be localised much like the
Kobayashi or the Carath\'{e}odory metric near peak points and this
is done in proposition \ref{A1}. Therefore there is a completely
natural formulation of this result that is of a local nature and
this will be evident from the proof of this theorem.

\medskip

\no Studying this invariant using the scaling procedure unifies
several disparate questions and applications include an alternate
proof of the Wong-Rosay theorem -- it must be mentioned that this
was also done by B. Fridman in \cite{Fridman-1983} using the fact
that $h_D(p, \mbf B^n) \ra 0$ as $p \ra \pa D$ for $D$ a strongly
pseudoconvex domain, the emphasis here being a different approach,
a characterization of normal analytic polyhedra with a noncompact
automorphism group with one orbit accumulating at a singular
boundary point (and this recovers part of the main theorem in
\cite{Kim&Pagano-2001}), and theorems of Coupet-Pinchuk-Sukhov
(\cite{CoupetPinchuk&Sukhov-1996}) and Pinchuk
(\cite{Pinchuk-1980}) about the inequivalence of two given domains
with different Levi geometry.

\medskip

\noindent As another byproduct, it is also possible to describe
the Kobayashi balls on a weakly pseudoconvex finite type domains
in $ \mathbf{C}^2$ and on convex finite type domains in $
\mathbf{C}^n$ in Euclidean terms. More precisely, it is known that
on a strongly pseudoconvex domain $D$ in $ \mathbf{C}^n$, $ B_D(p,
R) $ contains and is contained in ellipsoids $ E^{\pm}_p $ each of
whose major and minor axis are of the order of $ C(R) \big(
\mbox{dist} (p, \partial D ) \big)^{1/2} $ and $ c(R) \mbox{dist}
(p, \partial D) $ respectively where the positive constants $ C(R)
$ and $c(R) $ are independent of $p$. It is possible to obtain
analogues of this result for weakly pseudoconvex and convex finite
type domains without integrating the infinitesimal metric. Such
estimates were obtained by Aladro in \cite{Aladro-1989} for weakly
pseudoconvex finite type domains in $ \mathbf{C}^2$ using a
suitable metric on the horizontal subbundle on $ \partial D$ due
to Nagel-Stein-Wainger \cite{NagelStein&Wainger-1985}. The
arguments used here avoid any reference to such considerations and
instead rely only on scaling. These estimates, as is known (see
for example \cite{Krantz-1991}), are useful in verifying the
generalized sub-mean value property for plurisubharmonic functions
- a property that is needed in proving analogues of Fatou's
theorem on the boundary behaviour of $ H^p$ functions.

\medskip

\no The second theme explored in this article is the rigidity of
continuous isometries of these metrics. More precisely if $D, D'$
are $C^2$-smooth strongly pseudoconvex domains in $\mbf C^n$
and $f : D \ra D'$ is a continuous isometry of the Kobayashi
metrics on $D, D'$, it is not known whether $f$ must necessarily
be holomorphic or conjugate holomorphic. The same question can be
asked for the Carath\'{e}odory metric or for that matter any
invariant metric as well. An affirmative answer for the Bergman
metric was given in \cite{Greene&Krantz-1982} and this required
knowledge of the limiting behaviour of the holomorphic sectional
curvatures of the Bergman metric near strongly pseudoconvex
points. In general, the Kobayashi metric is just upper
semicontinuous and therefore a different approach will be needed
for this question. The case of continuous isometries when $D$ is
smooth strongly convex and $D'$ is the unit ball was dealt with in
\cite{Seshadri&Verma-2006} and this was improved upon in
\cite{Kim&Krantz-2008} to handle the case when $D$ is a $C^{2,
\ep}$-smooth strongly pseudoconvex domain $ (\epsilon > 0) $,
and a common ingredient in both proofs was the use of Lempert
discs. Motivated by such considerations it seemed natural to
determine the extent to which isometries behave like holomorphic
mappings and one example is provided by the following:

\begin{thm}
Let $D_1, D_2$ be two bounded strongly pseudoconvex domains in
$\mbf C^n$ with $C^2$-smooth boundaries. Let $D_1^k, D_2^k $
for $ k \geq 1 $ be two sequences of domains that converge to
$D_1, D_2$ respectively in the $C^2$ topology and let $d_{D^k_1},
d_{D^k_2}$ denote the Kobayashi metrics on these domains.
Similarly let $d_{D_1}, d_{D_2}$ denote the Kobayashi metrics on
$D_1, D_2$ respectively. Suppose that $f^k : (D^k_1, d_{D^k_1})
\ra (D^k_2, d_{D^k_2})$ is a $C^1$-smooth isometry for each $k
\ge 1$ and that there is a point $p^1 \in D_1$ such that some
subsequence $\{f^{k_j}(p^1)\}$ converges to a point $p^2 \in D_2$.
Then there is a uniform constant $C > 0$ with the property that:
\[
\vert f^{k_j}(p) - f^{k_j}(q) \vert \le C \vert p - q \vert^{1/2}
\]
for all $p, q \in D_1$.
\end{thm}

\no Several remarks are in order here. First, by a $ {C}^0$-isometry 
we mean a distance preserving bijection between the
metric spaces $ ( D_1, d_{D_1} )$ and $ ( D_2, d_{D_2} )$. For $ k
\geq 1$, a $ {C}^k$-isometry is a $ {C}^k$-diffeomorphism
$f$ from $ D_1$ onto $ D_2$ with $ f^*( F_{D_1}) = F_{D_2}$.
Second, a given pair of points $p, q \in D_1$ are evidently
contained in $D^k_1$ for all large $k$ and hence $f^{k_j}(p),
f^{k_j}(q)$ are well defined. Thirdly, the estimate shows that the
family $\{f^{k_j}\}$ is `uniformly equicontinuous' on $\ov D_1$,
and this may be regarded as a version of Vitushkin's theorem about
the uniform extendability of a compact subgroup of holomorphic
automorphisms of a smooth real analytic strongly pseudoconvex
domain for isometries. The hypothesis about the existence of $p^1
\in D_1$ such that $f^{k_j}(p^1) \ra p^2 \in D_2$ is to be
understood as saying that $\{f^{k_j}\}$ is a compact family.
Without this the theorem is false even in the holomorphic category
as the example of the unit ball shows. It must be mentioned that
other relevant theorems of this nature for holomorphic
automorphisms were proved by Coupet (\cite{Coupet-1992}) and
Coupet-Sukhov (\cite{Coupet&Sukhov-1996}). When no such $p^1, p^2$
exist for any subsequence of $f^k$, that corresponds to the
non-compact situation and this has been studied by Kim-Krantz in
\cite{Kim&Krantz-2008} where they prove that $D_1, D_2$ are both
biholomorphic to the ball provided their boundaries are at least
$C^{2, \ep}$-smooth where $\ep > 0$. This statement can be
deduced from their arguments in \cite{Kim&Krantz-2008} without any
additional difficulties. The main step in the proof of the above
theorem is to show that
\begin{equation} \label{t8}
\big| df^k (z) v \big| \lesssim \frac{|v|}{ \big( \mbox{dist}(z,
\partial D^k_1) \big)^{1/2}}
\end{equation}
uniformly for all $k$, vectors $ v \in \mathbf{C}^n $ and $ z \in
D^k_1 $. This in turn relies on knowing that $ f^k $ uniformly
preserves the distance to the boundary, i.e.,
\[
\mbox{dist} \big( f^k(z), \partial D^k_2  \big) \approx
\mbox{dist}( z,\partial D^k_1)
\]
for $ z \in D^k_1 $. In the holomorphic case, this would be done
by pulling back a strongly plurisubharmonic defining function for
$ D^k_2 $ by $ f^k $ and applying the Hopf lemma to this
composition. This cannot be apriori applied in the case of
isometries and to circumvent the difficulty, the variation of the
integrated Kobayashi distance in the domains $ D^k_1, D^k_2 $ must
be studied. This is done in lemma \ref{F8} and lemma \ref{F18} and
these yield the $ C^2 $-stability of the estimates for the
integrated Kobayashi distance (cf proposition \ref{F2} and
\ref{F3}) obtained by Forstneric and Rosay in
\cite{Forstneric&Rosay-1987}. The proof concludes by integrating
(\ref{t8}) exactly as in the case of holomorphic mappings.

\medskip

\noindent Some of the material presented here has benefitted from 
conversations that the first author had with Kang-Tae Kim and Rasul Shafikov.
We would like to thank them for their valuable comments, suggestions and 
encouragement.


\section{Notation and Terminology}

\noindent Let $\Delta$ denote the open unit disc in the complex
plane and let $d_{hyp}(a, b)$ denote the distance between two
points $a, b \in \Delta $ with respect to the hyperbolic metric.
For $ r > 0$, $ \Delta(0, r) \subset \mathbf{C} $ will be the disc
of radius $r$ around the origin and $B(z, \delta) \subset
\mathbf{C}^n$ will be the Euclidean ball of radius $\delta > 0$
around $ z $. Let $ X $ be a complex manifold of dimension $n$.
The Kobayashi and the Carath\'{e}odory distances on $X$, denoted
by $d_X$ and $c_X$ respectively, are defined as follows:

\medskip

\no Let $z \in X$ and fix $\xi$ a holomorphic tangent vector at $z$.
Define the associated infinitesimal Carath\'{e}odory and Kobayashi
metrics as
\begin{eqnarray*} F^C_X(z,\xi) & = & \sup \{ | df(z) \xi |
: f \in \mathcal{O}(X, \Delta) \} \\
\mbox{and} \qquad F^K_X(z,\xi) & = & \inf \left \lbrace
\frac{1}{\alpha} : \alpha > 0, f \in \mathcal{O}(\Delta, X) \
\mbox{with} \ f(0) = z, f'(0) = \alpha \xi \right \rbrace
\end{eqnarray*}
respectively. The Kobayashi length of a piecewise $\mathcal{C}^1$-curve $\gamma
: [0, 1] \rightarrow X$ is given by
\[
L_X(\gamma) = \int _0^1 F^K_X( \gamma(t),\dot{\gamma}(t) ) \; dt,
\]
and finally the Kobayashi distance between $p, q \in X$ is defined as
\[
d_X(p,q) = \inf L_X(\gamma)
\]
where the infimum is taken over all piecewise differentiable curves $\gamma$ in $X$ joining
$p$ to $q$. Recall that $X$ is taut if $\cal O(\Delta, X)$ is a normal family.

\medskip

\no The Carath\'{e}odory distance $c_X$ between $p, q \in X$ is defined by setting
\[
c_X (p,q) = \sup_f d_{hyp}\big( (f(p),f(q) \big)
\]
where the supremum is taken over the family of
all holomorphic mappings $f : X \ra \Delta$.

\medskip

\noindent The notion of finite type for a real-analytic
hypersurface $ M \subset \mathbf{C}^n$ will be in the sense of
D'Angelo, i.e., there is no germ of a positive dimensional
subvariety in $M$.


\section{Some remarks on Fridman's invariant}

\no In this section we gather and prove several basic properties of $h_X(p, \cal H)$ that will be used in the sequel. This will be done in a slightly more
general setting by replacing the homogeneous domain $\cal H$ by a Kobayashi hyperbolic domain $\Om \subset \mbf C^n$ such that the quotient $\Om / {\rm
Aut}(\Om)$ is compact, where as usual ${\rm Aut}(\Om)$ denotes the group of holomorphic automorphisms of $\Om$ and as before $X$ will be Kobayashi
hyperbolic. An analogue of (1.1) can thus be defined as
\begin{equation}
h_X(p, \Om) = \inf_{r \in \cal R} \frac{1}{r}
\end{equation}
where $\cal R$ denotes the set of all $r > 0$ with the property
that there is a biholomorphic imbedding $f : \Om \ra X$ with
$f(\Om) \supset B_X(p, r)$.

\medskip

\noindent We begin with the following version of lemma 1.1 of
\cite{Fridman-1983} which will be crucial for our purposes:

\begin{lem} \label{0} Let $X$ be a Kobayashi hyperbolic manifold of
complex dimension $n$ and let $ D $ be a taut domain in $
\mathbf{C}^n $. Suppose that there exist two relatively compact
sets $ K_1 \subset D$ and $K_2 \subset X$ and a sequence $ \{F^k
\} $ of mappings $ F^k : D \rightarrow F^k(D) \subset X $
satisfying the following conditions:

\begin{enumerate}

\item[(i)] for each $ k \geq 1 $, $ F^k : D \rightarrow F^k(D) $
is a biholomorphism,

\item[(ii)] for each $ k \geq 1 $, there exists a point $ z^k \in
K_1 $ such that $ F^k(z^k) \in K_2$,

\item[(iii)] for any compact $ L \subset X$ there exists a number
$ s = s(L) $ such that $ L \subset F^s (D)$.

\end{enumerate}
Then $ X$ is biholomorphically equivalent to $D$.

\end{lem}

\begin{proof} Consider $ \phi^k = (F^k)^{-1} : F^k(D) \rightarrow D$ and
let $ \{ U_n \} $ be an exhaustion of $X$ by relatively compact
submanifolds such that $ U_n $ is compactly contained in $ U_{n+1}
$ for all $n$. In view of the tautness of $D$ and (ii), we may
assume that some subsequence of $ \{ \phi^k \}$ (which we continue
to denote by the same symbols) satisfies the following condition:
for each $ n \geq 1$, there exists a $ N $ such that $ \{ \phi^k
\}_{k \geq N} $ is defined on $ U_n$ and converges uniformly on
compact sets of $ U_n$ to $ \phi: U_n \rightarrow D $. Indeed, $
\phi : X \rightarrow D $. We show that $ \phi $ is a
biholomorphism from $ X$ onto $D$.

\medskip

\noindent To prove the injectivity of $ \phi $ consider $ x^1 $
and $ x^2$ any two points in $X$. Then for each $k$,
\begin{equation}
d_X(x^1,x^2) = d_X \big( F^k \circ \phi^k (x^1), F^k \circ \phi^k
(x^2) \big). \label{0.1}
\end{equation}
It follows from the distance non-increasing property of the
holomorphic mappings that
\begin{equation}
d_X \big( F^k \circ \phi^k (x^1), F^k \circ \phi^k (x^2) \big)
\leq d_D \big( \phi^k(x^1), \phi^k(x^2) \big) \label{0.2}
\end{equation}
and by using the triangle inequality, we get
\begin{equation}
d_D \big( \phi^k(x^1), \phi^k(x^2) \big) \leq d_D \big(
\phi^k(x^1), \phi(x^1) \big) + d_D \big( \phi(x^1), \phi(x^2)
\big) + d_D \big( \phi (x^2), \phi^k(x^2) \big). \label{0.3}
\end{equation}
Combining (\ref{0.1}), (\ref{0.2}) and (\ref{0.3}) and letting $ k
\rightarrow \infty$ gives
\[
d_X(x^1,x^2) \leq d_D \big( \phi(x^1), \phi(x^2) \big).
\]
Since $X$ is hyperbolic, we must have $ x^1 = x^2 $ whenever $
\phi(x^1) = \phi(x^2) $.

\medskip

It remains to show that $ \phi $ is surjective. For this note that
we have already proved that $X$ is biholomorphically equivalent to
a domain $ \phi(X) \subset D$. Hence we can consider as $X$ as a
subdomain of a taut domain $ D$. This together with (ii) implies
that $ \{ F^k\} $ admits a subsequence that converges uniformly on
compacts of $ D$ and since $ F^k( D) $ exhaust $X$, the limit
mapping $ F: D \rightarrow \overline{X} $. In view of relative
compactness of $ K_1$ and $ K_2$, we may assume that (after
passing to a subsequence if necessary) that $ \{ F^k(z^k) \}$
converges to $ x^0 \in \overline{K_2}$ and $ \{ z^k \} $ to $ z^0
\in \overline{K_1}$. It is immediate that $ F(z^0)= x^0$. Now
using the hyperbolicity of $X$, choose $ \epsilon > 0 $ small
enough so that $ B_X( x^0, \epsilon) $ is compactly contained in $
X$. Then it follows from (iii) that $ B_X( x^0, \epsilon) \subset
F^k(D) $ for all $k$ large. This implies that $ B_X( x^0,
\epsilon) \subset F(D)$.

\medskip

Let Jac$ F^k $ be the Jacobian of $ F^k$. Applying Hurwitz's
theorem to the sequence $ \{ \mbox{Jac}F^k \} $, we deduce that
either the Jacobian of $F$ is never zero at any point of $D$ or it
is identically zero on $D$. In the latter case, $ F(D) $ cannot
contain any open set. Since $ F(D) \supset B_X( x^0, \epsilon) $,
we conclude that $ F(D) $ is open. In particular, $ F(D) \subset X
$. For the mapping $ \phi \circ F : D \rightarrow D$ and any $ z
\in D$, it follows that
\[
\phi \circ F(z) = \displaystyle\lim_{k \rightarrow \infty} \phi^k
\circ F^k (z) = z
\]
This shows that $ D \subset \phi(X) $. This completes the proof of
the lemma.
\end{proof}

\begin{lem}
Let $\Omega$ be a hyperbolic domain in $ \mathbf{C}^n$ and assume that $ \Omega/ \mbox{Aut} (\Omega) $ is
compact. Then $\Omega$ is complete hyperbolic and hence taut.
\end{lem}

\begin{proof}
By hypothesis there exists a compact set $K
\subset \Omega$ such that for any point $ x \in \Omega$ there
exists a $y \in K$ and a $ g \in \mbox{Aut}(\Omega)$ with $ g(y)=
x$. Let $\epsilon > 0$ be sufficiently small so that
\[
L = \{ z \in \Omega : d_{\Omega} (z, K) < \epsilon \}
\]
is compactly contained in $\Omega$. Let $ \{ p^k \}$ be a Cauchy
sequence in $ \Omega$ in the Kobayashi metric. Then there exists a
positive integer $ k_0 $ such that $ d_{\Omega}(p^k, p^{k_0}) <
\epsilon $ for $k \geq k_0$. Let $ g \in \mbox{Aut}(\Omega) $ be
such that $ g(p^{k_0}) \in K $. Since biholomorphisms are
isometries for the Kobayashi metric, it follows that
\[
d_{\Om} \big(g(p^k), g(p^{k_0}) \big) = d_{\Om}(p^k, p^{k_0}) <
\ep
\]
which shows that $g(p^k) \in L$ for all $ k \geq k_0 $ and that
$\{g(p^k)\}$ is also Cauchy in the Kobayashi metric. Since $ L $
is compactly contained in $ \Omega$ and $ \Omega $ is hyperbolic,
there exists a $ q \in L $ such that $ g(p^k) \rightarrow q $ as $
k \rightarrow \infty $.  It follows that $ p^k \rightarrow g^{-1}
(q)$ as $ k \rightarrow \infty $. This shows that $ \Omega$ is
complete.
\end{proof}

\begin{prop} \label{A0} Let $X$ be a Kobayashi hyperbolic manifold of
complex dimension $n$ and let $ \Omega $ be a hyperbolic domain in
$ \mathbf{C}^n $ such that $ \Omega/ \mbox{Aut} (\Omega) $ is
compact. Then

\begin{enumerate}

\item [(i)] if there is an $ x^0 \in X $ such that $ h_X(x^0 ,
\Omega) = 0 $, then $ h_X(x^0, \Omega) \equiv 0 $ and $ X$ is
biholomorphically equivalent to $ \Omega $.

\item [(ii)] if there is an $ x^0 \in X $ such that $ h_X(x^0 ,
\Omega)
> 0 $ and $ X$ is taut, then there exists a biholomorphic
imbedding $ F : \Omega \rightarrow X, F( \Omega ) \supseteq
B_X(x^0,r) $ such that $ h_X (x^0 , \Omega) = 1/r $ (i.e., the
variational problem has an extremal).

\item [(iii)] $ h_X ( \cdot, \Omega) $ is continuous on $X$.

\end{enumerate}

\end{prop}

\begin{proof} To prove (i) let $x^0 \in X$ be such that $ h_X(x^0
, \Omega) = 0 $. By definition, there exists a sequence $ \{\phi^
k \} $ of biholomorphic imbeddings, $ \phi^k : \Omega \rightarrow
X $ such that $ B_X (x^0, k) \subset \phi^k (\Omega) $. Let $K$ be
a compact subset of $ \Omega $ such that for every $y$ in $ \Omega
$ there exists a $x \in K $ and a $ g \in \mbox{Aut}( \Omega)$
satisfying $ g(y) = x $. Composing with an appropriate
automorphism of $\Omega$, if necessary, we may assume that $ (\phi
^k) ^{-1}(x^0) = q^k \in K $. We now claim that $ \{
\phi^k(\Omega) \} $ is a sequence of subdomains of $ X $ that
exhausts $X$. Indeed, for any compact $ L \subset X$; since $ d_X
( \cdot, \cdot ) $ is continuous, there exists a constant $ c = c
(L) > 0 $ such that $ d_X ( x^0, x ) \leq c $ for all $ x \in L $.
It follows that $ L \subset B_X (x^0, k) \subset \phi^k (\Omega)\
\mbox{for all} \ k \geq c $. Now applying lemma \ref{0} to the
sequence $ \{ \phi^k \}$ with $ K_1 = \{ q^k \} $ and $ K_2 =
\{x^0 \}$, we obtain that $ X $ is biholomorphically equivalent to
$ \Omega$. Let $ f : X \rightarrow \Omega $ be a biholomorphism
from $X$ onto $\Omega$. Then for any $ x \in X$, $ h_X(x, \Omega)
= h_{\Omega} (f(x), \Omega)= 0 $. Hence, $ h_X ( \cdot, \Omega )
\equiv 0$.

\medskip

\no For (ii) observe that by definition, there exists a sequence
of biholomorphic imbeddings, $ F^k : \Omega \rightarrow X $ and
$R_k > 0$ satisfying $ B_X (x^0, R_k) \subset F^k (\Omega) $ and $
(R_k)^{-1} \rightarrow h_X(x^0 , \Omega)$ as $ k \rightarrow
\infty$. Let $K \subset \Omega $ be as in part (i). Then for each
$k$, there exists a $ f^k \in \mbox{Aut} (\Omega)$ such that $ f^k
\big( (F^k)^{-1}(x^0) \big) = q^k \in K $. Set $ \tilde{F}^k :=
F^k \circ (f^k)^{-1} : \Om \ra X $. Since $ X$ is taut, it follows
that $ \{ \tilde{F}^k \} $ is a normal family. Hence, $ {
\tilde{F}^k } $ admits a subsequence that converges uniformly on
compact sets of $ \Omega$ to a holomorphic mapping $ F: \Omega
\rightarrow X $ or $ F \equiv \infty $. The latter cannot be true
since $ \tilde{F}^k (q^k) = x^0 $ which implies that $ F(q) = x^0
$ where $ q = \displaystyle\lim_{k \rightarrow \infty} q^k  \in K
$. We must therefore have a holomorphic mapping $ F : \Omega
\rightarrow X $.

\medskip

To establish the injectivity of $ F$ fix $ \epsilon > 0 $
arbitrarily small. Then $ B_X \big(x^0, (h_X(x^0, \Omega ) )^{-1}
- \epsilon \big) \subset B_X(x^0, R_k) \subset \tilde { F}^k
(\Omega)$ for all $k$ large. It follows that $ B_X \big(x^0,
(h_X(x^0, \Omega ) )^{-1}  - \epsilon \big) \subset F( \Omega)$.
Since $ F( \Omega ) $ contains a non-empty open set, $F$ must be
non-constant. Consider any point $ c \in \Omega$. Each mapping $
\tilde{F}^k ( \cdot) - \tilde {F}^k (c) $ never vanishes in $
\Omega \setminus \{c\} $ because of the injectivity of $ \tilde
{F}^k $ in $ \Omega$. Applying Hurwitz's theorem to the sequence $
\{ \tilde{F}^k ( \cdot) - \tilde {F}^k (c) \} \subset \mathcal{O}
( \Omega \setminus \{c\}, \mathbf{C}^n ) $, we get that $ F(z)
\neq F(c) $ for all $ z \in \Omega \setminus \{c\} $. Since $c$ is
any arbitrary point of $ \Omega$, this is just the assertion that
$F$ is injective on $ \Omega$.

\medskip

It remains to show that $ F ( \Omega ) \supset B_X \big(x^0,
(h_X(x^0, \Omega) )^{-1} \big)$. For this, consider
\[
G^k := (\tilde F^k) ^{-1} : B_X \big(x^0, (h_X(x^0, \Omega ))^{-1}
- \epsilon \big) \rightarrow \Omega
\]
that are defined for all large $k$. Since $ \Omega $ is taut, $ \{
G^k \} $ forms a normal family. Also, $ G^k (x^0) = q^k $ implies
that $ \{ G^k \} $ admits a convergent subsequence that converges
uniformly on compact sets of $ B_X \big(x^0, (h_X(x^0, \Omega
))^{-1} - \epsilon \big)$ to a holomorphic mapping $ G : B_X
\big(x^0, (h_X(x^0, \Omega ))^{-1} - \epsilon \big) \rightarrow
\Omega $ and $ G(x^0) = q$. Now for any $ p \in B_X \big(x^0,
(h_X(x^0, \Omega ))^{-1} \big)$, there exists a compact set $ L
\subset \Omega $ such that $ G^k(p) \subset L $ for all $k$ large.
Consequently, $ p = \tilde{F}^k \circ G^k (p) \subset \tilde{F}^k
(L) $ for all $k $ large. Let $ \{ w^k \} \subset L $ such that $
\tilde{F}^k (w^k ) = p $. In view of compactness of $L$, there
exists a $ w^0 \in L $ such that $ w^k \rightarrow w^0$ as $ k
\rightarrow \infty$. Since $ \tilde F^k \rightarrow F $ uniformly
on compacts of $ \Omega$, we have $ \tilde{F}^k (w^k ) \rightarrow
F(w^0)$. Hence, $ p = F(w^0) \subset F( \Omega ) $. Therefore, $ F
( \Omega ) \supset B_X \big(x^0, (h_X(x^0, \Omega))^{-1} \big) $.

\medskip

\noindent To prove (iii) let $ x^0 \in X $ be such that $ h_X
(x^0, \Omega)
> 0 $ and consider $ H_X( x^0, \Omega ) := (h_X(x^0,
\Omega))^{-1} $. It is enough to show that $ H_X( \cdot, \Omega) $
is continuous. Let $ x^1, x^2 \in X $ such that $ 2 d_X ( x^1,x^2)
< H_X ( x^1, \Omega ) $. Let $ \epsilon > 0 $ be such that
\[
H_X( x^1, \Om) - 2 d_X (x^1, x^2 ) - 2 \epsilon > 0.
\]
Using the triangle inequality for $ d_X$, we get
\[
B_X \big( x^2, H_X (x^1, \Omega)  - \epsilon - d_X(x^1, x^2) \big)
\subset B_X \big( x^1, H_X( x^1, \Omega ) - \epsilon \big)
\]
which means that
\[
H_X (x^1, \Omega)  - \epsilon - d_X(x^1, x^2) \leq H_X( x^2,
\Omega)
\]
or equivalently that
\begin{eqnarray}
H_X (x^1, \Omega) - H_X(x^2, \Omega) \leq d_X(x^1, x^2) +
\epsilon. \label{1.1}
\end{eqnarray}
But by the choice of $\ep > 0$ we know that
\[
H_X(x^2, \Omega) - \epsilon - d_X(x^1,x^2) > 0
\]
and similarly the triangle inequality applied once again gives
\[
B_X \big( x^1, H_X (x^2, \Omega)  - \epsilon - d_X(x^1, x^2) \big)
\subset B_X \big( x^2, H_X( x^2, \Omega ) - \epsilon \big)
\]
or equivalently that
\begin{eqnarray}
H_X (x^2, \Omega) - H_X(x^1, \Omega) \leq d_X(x^1, x^2) +
\epsilon. \label{1.2}
\end{eqnarray}
Combining (\ref{1.1}) and (\ref{1.2}) yields
\[
\big |H_X (x^2, \Omega) - H_X(x^1, \Omega) \big| \leq d_X(x^1,
x^2)
\]
and this shows that $ H_X( \cdot, \Omega) $ is continuous
in the topology induced by the Kobayashi metric. Since $X$ is
hyperbolic, the topology induced by the Kobayashi metric coincides
with its intrinsic topology and hence the result follows.
\end{proof}

\no It must be mentioned that corresponding statements for $h_X(p,
\cal H)$ where $\cal H$ is a bounded homogeneous domain were
proved in \cite{Fridman-1983}. A similar situation that dealt with
a hyperbolic domain $\Om$ with $\Om / {\rm Aut}(\Om)$ compact was
also considered in \cite{Fornaess&Sibony-1981}.

\medskip

\noindent We shall henceforth be concerned primarily with
Fridman's invariant as defined in (1.1) using the Kobayashi
metric. To analyse its boundary behaviour the first step is to
show that it can be localised near peak points.

\begin{prop} \label{A1}
Let  $ D \subset \mathbf{C}^n $ be a bounded domain and let $ z^0
\in \partial D $ be a local holomorphic peak point. Then for every
neighbourhood $U$ of $ z^0$ there exists a neighbourhood $ V
\subset U $ of $ z^0$ with $V$ relatively compact in $U$ such that
for all $ z \in V \cap D$, we have
\[
c \ h_D (z, \mathbf{B}^n ) \leq  h_{U \cap D} (z, \mathbf{B}^n )
\leq h_D (z, \mathbf{B}^n )
\]
where $ c > 0 $ is a constant independent of $z \in V \cap D $.
\end{prop}

\begin{proof} Let $U$ be a  neighbourhood of $z^0$ and $ g \in
\mathcal{A}(U \cap D)$, the algebra of continuous functions on the
closure of $U \cap D$ that are holomorphic on $ U \cap D$, such
that $ g(z^0) = 1$ and $ | g(p)| < 1 $ for $ p \in \overline{ U
\cap D} \setminus \{z^0 \} $. Fix $ \epsilon > 0$. Then there
exists a neighbourhood $ U_1 \subset U$ of $ z^0$ such that
\begin{equation*}
 F^K_D(z, v ) \leq F^K_{U \cap D} (z,v) \leq ( 1 +
\epsilon) F^K_D(z, v)
\end{equation*}
for $ z \in U_1 \cap D $ and $v$ a tangent vector at $ z$. This is
possible by the localisation property of the Kobayashi metric (see
for example lemma 2 in \cite{Royden-1971} or \cite{Graham-1975}).

\medskip

The first inequality evidently implies that $ B_{U \cap D} (z, r)
\subset B_D(z, r) $ for all $ z  \in U_1 \cap D $ and all $r > 0$.
Let $ R > 0 $ be such that there exists a biholomorphic imbedding
$ f: \mathbf{B}^n \rightarrow D $ satisfying $ B_D (z, R) \subset
f( \mathbf{B}^n )$. Composing with an appropriate automorphism of
$ \mathbf{B}^n$, if necessary, we may assume that $ f(0) = z$. For
any $ \epsilon > 0$, there exists $ r \in (0,1) $ such that $
B_D(z, R - \epsilon) \subset f( \mathbf{B}^n (0, r)) $. Since $
z^0 \in \partial D $ is a local holomorphic peak point it follows
that (see for example lemma 15.2.2 in \cite{Rudin}) there is a
neighbourhood $ U_2 \subset U_1 $ of $ z^0$ such that $ f(
\mathbf{B}^n (0, r)) \subset U \cap D $ whenever $ z \in U_2 \cap
D$. Now define $\tilde{f}(z) = f(rz) : \mathbf{B}^n \rightarrow U
\cap D$ which is a biholomorphic imbedding such that
$\tilde{f}(\mathbf{B}^n) \supset B_{U \cap D}(z, R - \epsilon)$.
Hence $h_{U \cap D}(z, \mathbf{B}^n) \le 1/(R - \epsilon)$
whenever $z \in U_2 \cap D$. By taking the infimum over all such
$R > 0$ we get that $h_{U \cap D}(z, \mathbf{B}^n) \leq h_D(z,
\mathbf{B}^n)$ whenever $z \in U_2 \cap D$.

\medskip

\noindent For the lower estimate the following observation will be
needed. Fix neighbourhoods $ U_2 \subset U $ of $ z^0$ as above.
Then for every $ R > 0$ there is a neighbourhood $ V \subset U_2$
of $ z^0$ with the property that if $ z \in V \cap D $ then $ B_{U
\cap D} (z, R) \subset U_2 \cap D$. For this it suffices to show
that
\[
\displaystyle\lim_{z \rightarrow z^0 } d_{ U \cap D } ( z, ( U
\cap D) \setminus \overline{U_2 \cap D} ) = + \infty.
\]
Indeed for every $ p \in ( U \cap D) \setminus \overline{U_2 \cap
D}$,
\[
d_{ U \cap D} (z, p) \geq d_{\Delta} ( g(z), g(p)) \rightarrow +
\infty
\]
as $ z \rightarrow z^0 $ since $ g(z^0) = 1 $ and $ |g| < 1 $ on $
( U \cap D) \setminus \overline{U_2 \cap D}$. This proves the
claim.

\medskip

\noindent Now for a given $ R > 0 $ let $ V $ be a sufficiently
small neighbourhood of $ z^0$ so that $ B_{ U \cap D} (z, R)
\subset U_2 \cap D $ if $ z \in V \cap D $. Pick $ p \in D$ in the
complement of the closure of $ B_{U \cap D} (z, R) $ and let $
\gamma : [0,1] \rightarrow D $ be a differentiable path with $
\gamma(0) = z $ and $ \gamma (1) = p$. Then there is a $ t_0 \in
(0,1) $ such that $ \gamma( [ 0, t_0 ) ) \subset B_{U \cap D} (z,
R) $ and $ \gamma (t_0 ) \in
\partial B_{U \cap D} (z, R) $. Hence
\begin{eqnarray*}
\int_0^1 { F^K_D ( \gamma(t), \dot{\gamma}(t)) dt } & \geq &
\int_0^{t_0} { F^K_D ( \gamma(t), \dot{\gamma}(t)) dt } \\
& \geq & 1/ ( 1 + \epsilon ) \int_0^{t_0} { F^K_{U \cap D} (
\gamma(t), \dot{\gamma}(t)) dt } \\
& \geq & 1/ ( 1 + \epsilon ) \  d_{ U \cap D} ( z, \gamma(t_0) ) =
R / ( 1 + \epsilon )
\end{eqnarray*}
which implies that $ d_D(z, p) \geq R/ ( 1 + \epsilon ) $. In
other words,
\begin{equation}
B_D \left(z, R/ (2( 1 + \epsilon) ) \right) \subset B_ { U \cap D}
(z, R) \label{Q2}
\end{equation}
if $ z \in V \cap D$. To conclude, suppose that there is a
biholomorphic imbedding $ f : \mathbf{B}^n \rightarrow U \cap D $
with $ B_{U \cap D} (z, R) \subset f( \mathbf{B}^n) $. It follows
from (\ref{Q2}) that $ h_D( z, \mathbf{B}^n) \leq 2( 1 + \epsilon)
/ R $ \ or equivalently that
\[
h_D( z, \mathbf{B}^n) \leq 2 ( 1 + \epsilon ) h_{U \cap D} ( z,
\mathbf{B}^n)
\]
whenever $ z \in V \cap D$. Finally, observe that
\[
1/(2(1 + \epsilon)) h_D(z, \mathbf{B}^n) \leq h_{U \cap D} (z,
\mathbf{B}^n) \leq h_D(z, \mathbf{B}^n) \ \mbox{for all} \ z \in V
\cap D.
\]
\end{proof}

\noindent A bounded domain $P \subset \mathbf{C}^n$ is an
\textit{analytic polyhedron} if there exist holomorphic functions
$ f^1, \ldots, f^l$ defined on an open neighbourhood of the closure of
$P$ such that
\[
P = \{ z \in \mathbf{C}^n : |f^1 (z)| < 1, \ldots, |f^l(z)| < 1 \}.
\]
The generating set $ f^1, \ldots, f^l $ for $P$ will be assumed to be
minimal in the sense that none of the $ f^i$'s can be dropped
without distorting $P$. In addition, if
\[
( d f^{i_1} \wedge  \ldots\wedge d f^{i_l}) \neq 0 \ \mbox{at} \ p
\]
whenever $ |f^{i_1}(p)|= \ldots = |f^{i_k}(p)| = 1 $ for the unrepeated
indices $ i_1, \ldots, i_k \in \{1, \ldots, l \}$, then $P$ is said to be
a \textit{generic} analytic polyhedron.

\medskip

\noindent For an arbitrary bounded domain $ \Omega \subset
\mathbf{C}^n$, let $ A(\Omega) $ denote the algebra of those
continuous complex-valued functions on $ \overline{\Omega} $ which
separates the points of $ \overline{\Omega} $. A \textit{boundary}
of $ \Omega $ for $ A(\Omega) $  is a subset $S$ of
$\overline{\Omega}$ such that for each $f$ in $ A(\Omega) $ there
is a point $z$ in $S$ with $ |f(z)| = \displaystyle\sup |f| $ on $
\overline{\Omega} $. If the class of all boundaries of $ \Omega $
contains a smallest set $M$, the set $M$ is called the
\textit{minimal boundary} of $ \Omega $. The \textit{Shilov
boundary} $S$ is defined to be the smallest closed boundary. It is
already known (see for instance \cite{Alexander&Wermer}) that such
a boundary exists. Moreover, the minimal boundary $M$, if it
exists, must be is contained in the Shilov boundary. $S$ is
evidently closed. For $P$ a bounded analytic polyhedron as above,
there is also a \textit{distinguished boundary} which consists of
those points in $ \overline{P} $ at which at least $n$ of the
defining functions $ f^i$ are of modulus $1$. The following
theorem in \cite{Hoffman-1960} will be useful for our purposes:

\begin{thm} Assume that $P$ is a bounded generic analytic
polyhedron as above and let $ A(P) $ denote the algebra of
continuous functions on $ \overline{P}$ which can be uniformly
approximated on $ \overline{P} $ by functions holomorphic in a
neighbourhood of $ \overline{P}$. If $l=n$, then the Shilov and
the minimal boundaries for $ A(P) $ coincide with the
distinguished boundary of $ P$.
\end{thm}

\noindent Having identified the Shilov boundary for a generic
analytic polyhedron $P$, it is possible to show that the function
$ h_P( \cdot, \Delta^n) $ can be localised near it.

\begin{prop} \label {A3} Let $ \Omega \subset \mathbf{C}^n$ be a domain and
$ f^i \in \mathcal{O}(\Omega) $ for $ 1 \leq i \leq n$. Let $P$ be
a bounded component of $ \{ z \in \Omega : |f^i(z)| < 1 \
\mbox{for} \ 1 \leq i \leq n \} $. Assume that $P$ is a generic
analytic polyhedron in the sense described above. Let $ z^0 $ be a
point on the Shilov boundary of $P$, i.e., $ |f^i(z^0)| = 1 \
\mbox{for} \ 1 \leq i \leq n $. Then for every neighbourhood $U$
of $ z^0$, there exists a neighbourhood $ V \subset U $ of $ z^0 $
such that for all $ z \in V \cap P$, we have
\[ c \ h_P (z, \Delta^n ) \leq  h_{U \cap P} (z, \Delta^n )
\leq h_P (z, \Delta^n ) \] where $ c > 0 $ is a constant
independent of $z \in V \cap P $ and $\Delta ^n $ is the unit
polydisc in $\mathbf{C}^n $.
\end{prop}

\begin{proof} Let $U$ be a neighbourhood of $z^0$ and let $
\theta_i, i= 1, \ldots, n $ be real numbers satisfying the following
properties:

\begin{itemize}

\item $ F := ( e^{\iota \theta_1} f^1, \ldots, e^{\iota \theta_n}
f^n) : U \rightarrow \mathbf{C}^n $ is a biholomorphic imbedding;

\item $ F(z^0) = (1, \ldots, 1) \in \mathbf{C}^n ;$

\item $ F( U \cap P ) = F(U) \cap \Delta^n$.

\end{itemize}

\noindent This can be achieved by applying the implicit function
theorem at the point $z^0$. Let $ R > 0 $ such that there exists a
biholomorphic imbedding $ f: \Delta^n \rightarrow P $ satisfying $
B_P (z, R) \subset f( \Delta^n )$. Composing with an appropriate
automorphism of $ \Delta^n$, if necessary, we may assume that $
f(0) = z$. For any $ \epsilon > 0$, there exists $ r \in (0,1) $
such that $ B_P(z, R - \epsilon) \subset f( \Delta^n (0, r)) $. We
claim that there is a neighbourhood $ U_1 $ of $ z^0$ which is
relatively compact in $U $ such that $ f( \Delta^n (0, r)) \subset
U \cap P $ whenever $ z \in U_1 \cap P$. To see this, suppose the
above claim is not true. Then there exist a sequence of points $
\{w^k \}_{k=1}^{\infty} \subset \Delta^n(0,r) $ and a sequence $
\{ g^k \}_{k=1}^{\infty} $ of holomorphic mappings $ g^k :
\Delta^n \rightarrow P $ such that $ g^k (0) \rightarrow z^0 $ as
$ k \rightarrow \infty$ and
\begin{equation}
|g^k(w^k) - z^0 | > \epsilon_0  \label{Q4}
\end{equation}
for some $ \epsilon_0 > 0$. Since $ P $ is bounded, $\{g^k \} $ is
a normal family. Let $G$ be any limit of $ \{g^k \} $. Then $ G(0)
= z^0 $. Consider $ \tilde{f^i} := f^i \circ G $ for $ i = 1,
2,...,n $. Then $ \tilde{f^i} \in \mathcal{O} (\Delta^n),
|\tilde{f^i} (0)| = 1 $ for all $ 1 \leq i \leq n $ and $
|\tilde{f^i}(z)| \leq 1 $ for all $ z \in \Delta^n $ and for all $
1 \leq i \leq n $. By the maximum modulus theorem $ |\tilde{f^i}|
\equiv 1 $ for all $ 1 \leq i \leq n $. Hence, $ G(\Delta^n)
\subset M := \{ |f^i| = 1, \ 1 \leq i \leq n \} $. Since $M$ is a
totally real manifold, $ G$ is identically a constant. Therefore $
G(z) \equiv z^0 $ since $ G(0) = z^0 $. The constant map $ G(z)
\equiv z^0 $ is thus the only limit point of $ \{g^k \} $. This
contradicts (\ref{Q4}) and proves the claim.

\medskip

Moreover, the above argument also shows that for each $ \epsilon >
0$
\begin{equation}
F^K_P(z, v ) \leq F^K_{U \cap P} (z,v) \leq (1 + \epsilon)
F^K_P(z, v) \label{Q15}
\end{equation}
provided  $ z $ is close enough to $ z^0$ and $v$ a tangent vector
at $ z$.

\medskip

Hence, there is a neighbourhood $ U_1 \subset U$ of $ z^0$ such
that $ f \big( \Delta^n (0, r) \big) \subset U \cap P $ whenever $
z \in U_1 \cap P$. It follows that
\[
B_{U \cap P}(z, R - \epsilon) \subset B_P(z, R - \epsilon) \subset
f \big( \Delta^n(0,r) \big)
\]
which implies that
\[
h_{U \cap P}(z, \Delta^n) \le (R - \epsilon)^{-1}
\]
and hence
\[
h_{U \cap P}(z, \Delta^n) \leq h_P(z, \Delta^n)
\]
for all $z \in U_1 \cap P$.

\medskip

To establish the lower estimate, observe that for every $ R > 0 $
there is a neighbourhood $ V \subset U_1 $ of $ z^0$ such that $
B_{ U \cap P } (z, R) \subset U_1 \cap P $ whenever $ z \in V \cap
P $.  It suffices to show that $ F \left( B_{ U \cap P } (z, R)
\right) \subset F( U_1 \cap P ) $ if $z$ is close enough to $
z^0$. Since biholomorphisms are isometries for the Kobayashi
metric, this is the same as finding a neighbourhood $ V $ of $
z^0$ such that $ B_{ F( U \cap P) } ( F(z), R ) \subset F( U_1)
\cap \Delta^n $ if $ z \in V \cap P $. It is already known that $
B_{ F( U \cap P)} ( w, R) \subset B_{\Delta^n} (w, R) $ for all $w
\in F( U \cap P) $. We now claim that
\[
B_{\Delta^n} ( F(z), R) \subset F( U_1) \cap \Delta^n
\]
if $ z$ is sufficiently close to $ z^0$. But this is a straight
forward consequence of the fact that for any two points $ a =
(a_1, a_2, \ldots, a_n)$ and $ b= (b_1, b_2, \ldots, b_n)$ in $ \Delta^n $,
\[
d_{\Delta^n} (a,b) = \max \left( d_{\Delta}(a_1,b_1), d_{\Delta}
(a_2, b_2), \ldots, d_{\Delta} (a_n,b_n) \right).
\]
Now an argument similar to the one in proposition \ref{A1} that
uses (\ref{Q15}) shows that
\[
h_P(z, \Delta^n) \leq 2 ( 1 + \epsilon) h_{ U \cap P} ( z,
\Delta^n).
\]
\end{proof}


\section{Behaviour of $h$ near strongly pseudoconvex boundary
points}

\noindent We will be using $ h_X( \cdot ) $ to denote $ h_X(
\cdot, \mathbf{B}^n )$ in the sequel unless stated otherwise.

\medskip

\noindent The main purpose of this section is to prove the
following which is (i) of theorem 1.1.

\begin{thm} \label{B0} Let $ D \subset \mathbf{C}^n$ be a
bounded strongly pseudoconvex domain with $ C^2$-smooth
boundary. Then $ h_{D} (z) \rightarrow 0 $ as $  z \rightarrow
\partial D$.
\end{thm}

\begin{proof} Let $ \{ z^k \} \subset D $ be a sequence converging
to $z^0 \in \partial D$. It suffices to show that $  h_{D}( z^k )
\rightarrow 0 $ as $ k \rightarrow \infty$. Several lemmas will be
needed to complete the proof of this theorem. To start with, the
following lemma in \cite{Pinchuk-1980} will be useful.

\begin{lem} \label{F1} let $ D$ be a strongly pseudoconvex domain,
$\rho$ a defining function for $ \partial D $ and $ p \in
\partial D $. Then there exists a neighbourhood $U$ of $p$
and a family of biholomorphic mappings $ h_{\zeta} : \mathbf{C}^n
\rightarrow \mathbf{C}^n$ depending continuously on $ \zeta \in
\partial D \cap U$ that satisfy the following:

\begin{enumerate}

\item [(i)]$ h_{\zeta}(\zeta) = 0 $.

\item [(ii)] The defining function $ \rho_{\zeta} = \rho \circ
h_{\zeta}^{-1} $ of the domain $ D ^{\zeta} := h_{\zeta} ( D) $
has the form
\[ \rho_{\zeta}(z) = 2 \big(\Re z_n + K_{\zeta}(z)\big) + H_{\zeta}(z) +
\alpha_{\zeta}(z) \] where $ K_{\zeta}(z) =
\displaystyle\sum_{i,j=1} ^n a_{ij} (\zeta) z_i z_j, H_{\zeta}(z)
= \displaystyle\sum_{i,j=1} ^n b_{ij} (\zeta) z_i \bar{z_j}$ and $
\alpha_{\zeta}(z)= o (|z|^2) $ with $ K_{\zeta}('z,0) \equiv 0 $
and $ H_{\zeta}('z,0) \equiv |'z|^2 $.

\item [(iii)] The mapping $ h_{\zeta}$ takes the real normal to $
\partial D $ at $ \zeta$ to the real normal $ \{ 'z =
y_n = 0 \} $ to $ \partial D ^{\zeta}$ at the origin.

\end{enumerate}

\end{lem}

\noindent Here, $ z \in \mathbf{C}^n$ is written as $ z = ( 'z,
z_n) \in \mathbf{C}^{n-1} \times \mathbf{C} $.

\medskip

\noindent To apply this lemma, select $ \zeta^k \in \partial D$,
closest to $ z^k$. For $k$ large, the choice of $\zeta^k$ is
unique since $ \partial D $ is sufficiently smooth. Moreover, $
\zeta^k \rightarrow z^0 $ and $ z^k \rightarrow z^0$ as $ k
\rightarrow \infty$. Let $ h^k := h_{\zeta^k} $ be the
biholomorphisms provided by the lemma above. Let $ T^k :
\mathbf{C}^n \rightarrow \mathbf{C}^n $ be the anisotropic
dilation map given by
\[
T^k ( 'z, z_n ) = \left( \frac{'z}{ \sqrt{\delta_k}} , \frac{z_n}
{ \delta_k} \right)
\]
and let $ D^k = T^k \circ h^k ( D) $. We observe that for $k$
large, $ h^k(z^k) = ( '0, - \delta_k ) $ so that $ T^k \circ h^k (
z^k ) = ( '0, -1) $. Since $h$ is invariant under biholomorphisms,
it follows that
\[
h_{D} ( z^k) = h_{D^k} ( ('0, -1) )
\]
for all $k$ large. We will show that $ h_{ D^k } ( ('0, -1) )
\rightarrow 0 $ as $ k \rightarrow \infty $.

\noindent It has been shown in \cite {Pinchuk-1980} that the
sequence of domains $ \{ D^k \} $ converges in the Hausdorff
metric to the unbounded realization of the unit ball, namely to
\[
 D_{\infty} = \Big \{ z \in \mathbf{C}^n : 2 \ \Re z_n + | 'z |^2 < 0
\Big \}.
\]

\noindent It is natural to investigate the behaviour of $ d_{D ^k}
( z, \cdot )$ as $ k \rightarrow \infty$. To do this, we use ideas
from \cite{Seshadri&Verma-2006}.

\begin{lem} \label{D1} Let $ x^0 \in D_{\infty}$. Then $ d_{D ^k} ( x^0,
\cdot ) \rightarrow d_{D_{\infty}} (x^0, \cdot ) $ uniformly on
compact sets of $ D_{\infty}$.
\end{lem}

\begin{proof} Let $ K \subset D_{\infty} $ be compact and suppose that the
desired convergence does not occur. Then there exists a $ \epsilon
_0 > 0 $ and a sequence of points $ \{p^k \} \subset K $ which is
relatively compact in $D^k $ for $k$ large such that
\[
\big| d_{D ^k} ( x^0, p^k ) - d_{D_{\infty}} (x^0, p^k ) \big|
> \epsilon _0
\]
for all $k$ large. By passing to a subsequence, assume that  $ p^k
\rightarrow p^0 \in K $ as $ k \rightarrow \infty$. Since $
d_{D_{\infty}} ( x^0, \cdot ) $ is continuous, it follows that
\begin{equation}
\big| d_{D ^k} ( x^0, p^k ) - d_{D_{\infty}} (x^0, p^0 ) \big|
> \epsilon _0/2 \label{Q3}
\end{equation}
for all $k$ large. Fix $ \epsilon > 0$ and let $ \gamma : [0,1]
\rightarrow D_{\infty} $ be a path such that $ \gamma(0) = x^0,
\gamma(1) = p^0$ and
\[
\int_0 ^1 { F^K _{D_{\infty}} \big( \gamma(t), \dot{\gamma}(t)
\big)} dt < d_{D_{\infty}} (x^0, p^0 ) + \epsilon/2.
\]
Define $ \gamma ^k : [0,1] \rightarrow \mathbf{C}^n $ by
\[
\gamma^k(t) = \gamma(t) + ( p^k - p^0 ) t.
\]
Since the image $ \gamma([0,1])$ is compactly contained in
$D_{\infty}$ and $ p^k \rightarrow p^0 \in K $ as $ k \rightarrow
\infty$, it follows that $ \gamma ^k : [0,1] \rightarrow D^k $ for
$k$ large. In addition, $ \gamma^k (0) = \gamma (0) = x^0 $ and $
\gamma^k (1) = \gamma(1) + p^k - p^0 = p^k $. It is already known
that $ F^K_{D^k} ( \cdot, \cdot) \rightarrow F^K_{D_{\infty}}
(\cdot, \cdot) $ uniformly on compact sets of $ D_{\infty} \times
\mathbf{C}^n$ (see \cite{Seshadri&Verma-2006}). Also, note that $
\gamma^k \rightarrow \gamma $ and $ \dot{\gamma}^k \rightarrow
\dot{\gamma} $ uniformly on $ [0,1]$. Therefore for $k$ large, we
obtain
\[
\int_0 ^1 { F^K_{D^k} \big( \gamma^k(t), \dot{\gamma}^k(t)  \big)}
dt \leq \int_0 ^1 { F^K _{D_{\infty}} \big( \gamma(t),
\dot{\gamma}(t) \big)} dt + \epsilon/2 < d_{D_{\infty}} (x^0, p^0
) + \epsilon.
\]
By definition of $ d_{D ^k} ( x^0, p^k )$ it follows that
\[
d_{D ^k} ( x^0, p^k ) \leq \int_0 ^1 { F^K_{D^k} \big(
\gamma^k(t), \dot{\gamma}^k(t) \big)} dt \leq d_{D_{\infty}} (x^0,
p^0 ) + \epsilon.
\]
Thus
\begin{eqnarray}
\displaystyle\limsup_{k \rightarrow \infty} d_{D ^k} ( x^0, p^k )
\leq d_{D_{\infty}} (x^0, p^0 ). \label{2.1a}
\end{eqnarray}

\noindent Conversely, since $ K \cup \{ x^0 \}$ is a compact
subset of $ D_{\infty} $, it follows that $ K \cup \{ x^0 \}$  is
compactly contained  $D^k$ for all $k$ large. Fix $ \epsilon > 0$
and let $ V \subset U $ be sufficiently small neighbourhoods of $
z^0 \in \partial D $ with $ V$ compactly contained in $U$ so that
\begin{eqnarray}
F^K_{D}( z, v ) \leq F^K_{U \cap D} (z,v) \leq ( 1 + \epsilon)
F^K_{D}(z, v) \label{2.2a}
\end{eqnarray}

\noindent for $ z \in V \cap D $ and $v$ a tangent vector at $ z$.
If $k$ is sufficiently large, $ { (T^k \circ h^k ) }^{-1} (x^0) $
and $ { (T^k \circ h^k ) }^{-1} (p^k) $ belong to $ V \cap D $. If
$U$ is small enough, $U \cap D $ is strictly convex and it follows
from Lempert's work \cite{Lempert-1981} that there exist $ m_k >
1$ and holomorphic mappings
\[
\phi^k : \Delta(0, m_k) \rightarrow U \cap D
\]
such that $ \phi^k(0) = { (T^k \circ h^k ) }^{-1} (x^0),\ \phi^k
(1) = { (T^k \circ h^k ) }^{-1} (p^k) $ and
\begin{eqnarray}
d_{U \cap D } \left( { (T^k \circ h^k ) }^{-1}(x^0), { (T^k \circ
h^k ) }^{-1} (p^k) \right)  & = &
d_{\Delta(0, m_k)} (0,1) \nonumber  \\
& =  & \int_0 ^1 { F^K _{ U \cap D } \big( \phi^k(t), \dot{\phi}^k
(t) \big) } dt . \label{2.2b}
\end{eqnarray}

\noindent By \cite{Venturini-1989}, it follows that
\[
d_{U \cap D } \left( { (T^k \circ h^k ) }^{-1}(x^0), { (T^k \circ
h^k ) }^{-1} (p^k) \right)  \leq (1 + \epsilon) d_{ D } \left( {
(T^k \circ h^k ) }^{-1}(x^0), { (T^k \circ h^k ) }^{-1} (p^k)
\right)
\]
for all $k$ large. Since $ T^k \circ h^k $ are biholomorphisms and
hence Kobayashi isometries,
\begin{eqnarray}
d_{T^k \circ h^k (U \cap D)} (x^0, p^k) \leq (1 + \epsilon)
d_{D^k}( x^0, p^k). \no
\end{eqnarray}
Now (\ref{2.2b}) shows that
\begin{eqnarray*}
\frac{1}{2} \log \left( \frac{m_k + 1} { m_k - 1} \right) =
d_{\Delta(0, m_k)} (0,1) & =  &
d_{U \cap D } \left( { (T^k \circ h^k ) }^{-1}(x^0), { (T^k \circ h^k ) }^{-1} (p^k)\right) \\
&  = & d_{T^k \circ h^k (U \cap D)} (x^0, p^k) \\
& \leq & (1 + \epsilon) d_{D^k}( x^0, p^k).
\end{eqnarray*}
However from (\ref{2.1a}) we have that
\[
d_{D^k}( x^0, p^k) \leq d_{D_{\infty}} (x^0, p^0) + \epsilon <
\infty
\]
and hence $ m_k > 1 + \delta$ for some uniform $ \delta > 0 $ for
all $ k $ large. Thus the holomorphic mappings $ \sigma^k = T^k
\circ h^k \circ \phi^k : \Delta (0, 1 + \delta) \rightarrow T^k
\circ h^k (U \cap D) \subset D^k $ are well-defined and satisfy $
\sigma^k(0) = x^0 $ and $ \sigma^k(1) = p^k $.

\medskip

We claim that $ \{ \sigma^k \} $ admits a subsequence that
converges uniformly on compact sets of $ \Delta (0, 1 + \delta) $
to a holomorphic mapping $ \sigma: \Delta (0, 1 + \delta)
\rightarrow D_{\infty} $. Indeed consider the disc $ \Delta (0, r)
$ of radius $ r \in (0,1 + \delta) $. Observe that $ (T^k \circ
h^k ) ^{-1} \circ \sigma^k(0) = \phi^k (0)= (T^k \circ h^k
)^{-1}(x^0) \rightarrow z^0 \in \partial D $ as $ k \rightarrow
\infty$. Let $ W $ be a sufficiently small neighbourhood of $
z^0$. Since $ z^0 \in \partial D $ is a local peak point, it
follows that $(T^k \circ h^k )^{-1} \circ \sigma^k (\Delta (0,r))
\subset W \cap D $ for all $k$ large. If $W$ is small enough,
there exists $ R > 1$ such that for all $k$ large
\[
h^k ( W \cap D) \subset \Big \{ z \in \mathbf{C}^n : | z_n + R|^2
+ |'z|^2 < R^2 \Big \} \subset \Omega_0
\]
where
\[
\Omega_0 = \Big \{ z \in \mathbf{C}^n : 2 R (\Re z_n) + |'z |^2 <
0 \Big \}.
\]
Note that $ \Omega_0$ is invariant under $T^k$ and $ \Omega_0 $ is
biholomorphically equivalent to $ \mathbf{B}^n $. Hence $
\sigma^k( \Delta (0, r)) \subset T^k \circ h^k ( W \cap D) \subset
\Omega_0$ for all $k$ large. If $ \sigma^k(z) = ( '\sigma^k(z),
\sigma^k_n(z)) $ for each $k$, this exactly means that
\[
2 R \big( \Re (\sigma^k_n (z)) \big) + | '{\sigma}^k(z)|^2 < 0
\]
whenever $ z \in \Delta(0,r)$. It follows that $ \{ \sigma^k_n(z)
\}_{k=1} ^{\infty} $ and hence each component of $
\{'{\sigma}^k(z)\}$ forms a normal family on $ \Delta(0, r) $.
Since $ r \in (0,1 + \delta) $ was arbitrary, the usual diagonal
subsequence yields a holomorphic mapping $ \sigma: \Delta (0, 1 +
\delta) \rightarrow \mathbf{C}^n $ or $ \sigma \equiv \infty$ on $
\Delta (0, 1 + \delta) $. The latter is not possible since $
\sigma(0) = x^0 $.

\medskip

It remains to show that $ \sigma: \Delta (0, 1 + \delta)
\rightarrow D_{\infty} $. Following \cite{Pinchuk-1980}, note that
$ D^k$ are defined by
\begin{eqnarray*}
\rho^k (z) = 2 \ \Re z_n + | 'z|^2 + A^k (z)
\end{eqnarray*}
where
\[
|A^k(z)| \leq |z|^2 \big( c\sqrt {\delta_k} + \eta( \delta_k
|z|^2)\big).
\]
Thus for $ z \in \Delta(0, r)$ and $ r \in (0,1 + \delta) $,
\begin{equation}
2 R \big( \Re (\sigma^k_n (z)) \big) + |'{\sigma}^k(z)|^2 + A^k(
\sigma^k(z)) < 0 \label{2.2c}
\end{equation}
where
\[
|A^k(\sigma^k(z))| < | \sigma^k(z) |^2 \big(c \sqrt {\delta_k} +
\eta( \delta_k |\sigma^k(z)|^2)\big).
\]
Letting $ k \rightarrow \infty $ in (\ref{2.2c}) yields
\[
2 R \big( \Re (\sigma^k_n (z)) \big ) + |'{\sigma}^k(z)|^2 \leq 0
\]
for $ z \in \Delta (0, r) $ or equivalently that $ \sigma
(\Delta(0, r) ) \subset \overline{D}_{\infty} $. Since $ r \in
(0,1 + \delta) $ was arbitrary, it follows that $ \sigma (
\Delta(0, 1 + \delta)) \subset \overline{D}_{\infty} $. Since $
\sigma(0) = x^0$, the maximum principle shows that $ \sigma (
\Delta(0, 1 + \delta)) \subset D_{\infty} $. Using (\ref{2.2a})
and (\ref{2.2b}), we get
\begin{eqnarray*}
\int_0 ^1 { F^K_{D^k} \big( \sigma^k(t),  \dot{\sigma}^k(t) \big
)} dt & \leq &  \int_0 ^1 { F^K _{D} \big( \phi^k(t),
\dot{\phi}^k(t) \big)}
dt  \\
& \leq & \int_0 ^1 { F^K _{U \cap D} \big( \phi^k(t),
\dot{\phi}^k(t) \big)} dt \\
& = & d_{T^k \circ h^k (U \cap D)} (x^0, p^k) \\
& \leq & (1 + \epsilon) d_{D^k}( x^0, p^k)
\end{eqnarray*}
Since $ \sigma^k \rightarrow \sigma $ and $ \dot{\sigma}^k
\rightarrow \dot{\sigma} $ uniformly on $ [0,1]$, again exploiting
the uniform convergence of $ F^K_{D^k}( \cdot, \cdot) \rightarrow
F^K_{D_{\infty}} (\cdot, \cdot)$ on compact sets of $ D_{\infty}
\times \mathbf{C}^n$, we see that
\begin{eqnarray*}
\int_0 ^1 {F^K_{D_{\infty}} \big(\sigma(t),  \dot{\sigma}(t)\big)
dt} \leq \int_0 ^1{F^K_{D^k} \big(\sigma^k(t), \dot{\sigma}^k(t)
\big) dt } + \epsilon \leq d_{D^k}( x^0, p^k) + C \epsilon
\end{eqnarray*}
for all $k$ large. Finally, observe that $ \sigma |_{[0,1]}$ is a
differentiable path in $D_{\infty}$ joining $ x^0$ and $ p^0 $.
Hence by definition
\begin{eqnarray}
d_{D_{\infty}} (x^0, p^0 ) \leq \int_0 ^1 {F^K_{D_{\infty}}
\big(\sigma(t), \dot{\sigma}(t) \big) dt} \leq d_{D^k}( x^0, p^k)
+ C \epsilon \label{2.2d}
\end{eqnarray}
Combining (\ref{2.1a}) and (\ref{2.2d}) shows that
\[ \displaystyle\lim_{k \rightarrow \infty} d_{D ^k} ( x^0,
p^k ) = d_{D_{\infty}} (x^0, p^0 ) \] which contradicts the
assumption (\ref{Q3}) and proves the required result.
\end{proof}

\begin{lem} \label{B2} Fix $ x^0 \in D_{\infty}$ and $ R > 0$. Then
\[ B_{D^k} (x^0, R) \rightarrow B_{D_{\infty}}(x^0, R) \] in the
Hausdorff sense. Moreover, for any $ \epsilon > 0 $ and for all
$k$ large

\begin{enumerate}

\item [(i)] $ B_{D_{\infty}}(x^0, R) \subset B_{D^k} (x^0, R +
\epsilon),$

\medskip

\item [(ii)] $ B_{D^k} (x^0, R - \epsilon)  \subset
B_{D_{\infty}}(x^0, R).$

\end{enumerate}

\end{lem}

\begin{proof} Let $ K \subset B_{D_{\infty}}(x^0, R)$ be compact. Then $ K
$ is a relatively compact subset of $ D^k$ for all $k$ large and
there exists a positive constant $ c=c(K) \in (0, R) $ such that $
d_{D_{\infty}} (x^0,z) < c$ for all $ z \in K $. Pick $ \tilde{c}
\in (c, R) $. It follows from lemma \ref{D1} that
\[
d_{D^k} (x^0, z) \leq d_{D_{\infty}} (x^0,z) + \tilde{c} - c
\]
for all $z$ in $K$ and for all $k$ large. Therefore
\[
d_{D^k }(x^0, z) \leq \tilde{c} < R
\]
for all $z \in K$ and for all $k$ large. This is just the
assertion that $ K $ is compactly contained in $ B_{D^k} (x^0, R)
$ for all $k$ large. Conversely, let $ K \subset \mathbf{C}^n$ be
a compact set such that $K $ is compactly contained in $ B_{D^k}
(x^0, R)$ for all $ k$ large. Then $ K $ is relatively compact
subset of $D_{\infty} $. Additionally, there exists a positive
constant $ c= c(K) \in (0, R) $ such that $ d_{D^k} (x^0, z) \leq
c $ for all $ z \in K $ and for all $ k$ large. Pick $ \tilde{c}
\in (c, R) $. Again applying lemma \ref{D1}, we see that
\[
d_{D_{\infty}}(x^0, z) < d_{D^k}(x^0,z) + \tilde{c} - c \qquad
\]
for all $ z \in K $ and all $k$ large. Thus for all $ z \in K$, we
obtain
\[
d_{D_{\infty}} (x^0, z) < \tilde{c} < R
\]
or equivalently $ K $ is compactly contained in $
B_{D_{\infty}}(x^0, R)$. This shows that the sequence of domains $
\big \{ B_{D^k}(x^0, R) \big \}$ converges in the Hausdorff metric
to $ B_{D_\infty}(x^0, R)$.

\medskip

\noindent To verify (i), first observe that the closure of $
B_{D_{\infty}}(x^0, R) $ is compact since $ D_{\infty}$ is
Kobayashi complete. Then using lemma \ref{D1}, we get that
\[
d_{D^k} (x^0, z) \leq d_{D_{\infty}} (x^0, z) + \epsilon
\]
for all $ z $ in the closure of $ B_{D_{\infty}}(x^0, R) $ and for
all $ k $ large. Said differently,
\[
B_{D_{\infty}}(x^0, R) \subset B_{D^k} (x^0, R + \epsilon)
\]
for all $k$ large.

\medskip

\noindent For (ii) suppose that the desired result is not true.
Then there exists a $ \epsilon_0 > 0$ and a sequence of points $
\{ a^l \} _{l=1} ^{\infty} \subset \partial B_{D_{\infty}}(x^0,
R)$ such that $ a^l \in B_{D^l} (x^0, R - \epsilon_0) $. In view
of compactness of $\partial B_{D_{\infty}}(x^0, R)$, we may assume
that $ a^l \rightarrow a \in \partial B_{D_{\infty}}(x^0, R)$ as $
l \rightarrow \infty$. It follows from lemma \ref{D1} that
\[
d_{D^l} (a^l, x^0) \rightarrow d_{D_{\infty}} (a, x^0)
\]
as $ l \rightarrow \infty$. Consequently, $ d_{D_{\infty}} (a,
x^0) \leq R - \epsilon_0 $. This violates the fact that $
d_{D_{\infty}} (a, x^0)= R$ thereby proving (ii).
\end{proof}

\medskip

\noindent \textit{Proof of Theorem \ref{B0}}: It suffices to show
that $ h_{ D^k } ( ('0, -1) ) \rightarrow 0$ as $ k \rightarrow
\infty $. For any $ R > 0$, there exists a biholomorphism $ \theta
: \mathbf{B}^n \rightarrow B_{D_{\infty}} ( ('0,-1), R) $ as $
D_{\infty}$ is biholomorphically equivalent to $ \mathbf{B}^n$.
For $ \epsilon > 0 $ given, there exists a $ \delta \in (0,1) $
such that
\begin{eqnarray}
B_{D_{\infty}} ( ('0,-1), R - \epsilon ) \subset \theta \big(
\mathbf{B}^n(0,\delta) \big) \subset B_{D_{\infty}} ( ('0,-1), R).
\label{3.1}
\end{eqnarray}
It follows from lemma \ref{B2} that
\begin{eqnarray}
\theta \big( \mathbf{B}^n(0,\delta) \big) \subset B_{D^k} (
('0,-1), R) \subset D^k \label{3.2}
\end{eqnarray}
and
\begin{eqnarray}
B_{D^k} (('0,-1), R - 2 \epsilon)  \subset B_{D_{\infty}}(('0,-1),
R - \epsilon). \label{3.3}
\end{eqnarray}
for all $k$ large. Finally note that (\ref{3.1}), (\ref{3.2}) and
(\ref{3.3}) altogether yield that there exists a biholomorphic
imbedding $ \theta : \mathbf{B}^n(0, \delta) \rightarrow D^k $
such that
\[
B_{D^k} (x^0, R - 2 \epsilon) \subset \theta \big(
\mathbf{B}^n(0,\delta) \big)
\]
so that
\[
h_{D^k}(( '0, -1)) \leq 1/(R - 2 \epsilon).
\]
Since $ R > 0 $ was arbitrary, we have $ h_{ D^k } ( ('0, -1) )
\rightarrow 0 $ as $ k \rightarrow \infty $. This finishes the
proof.
\end{proof}

\begin{rem} Observe that the above result can be stated as
\[
 h_{D} (z) \rightarrow h_{D_{\infty}} ( ('0,-1)) = 0
\]
as $ z \rightarrow z^0 $ where $ D_{\infty} \simeq \mathbf{B}^n$
is the model domain at the point $ z^0$.
\end{rem}

\medskip

\noindent \textbf{Behaviour of $h$ on a strongly pseudoconvex
domain $(D,c_D)$:}

\medskip

\noindent In this section, we intend to focus on Fridman's
invariant defined using the Carath\'{e}odory metric. To be more
concrete, let $ X  $ be a \textit{c-hyperbolic} complex manifold
of dimension $n$, i.e., the Carath\'{e}odory distance $ c_X$ is a
distance and the topology induced by $ c_X $ coincides with the
Euclidean topology on $X$. Then for $ p \in X$
\begin{equation}
h_X(p, \mathbf{B}^n) = \displaystyle\inf_{r \in \mathcal{R}}
\frac{1}{r} \label{CC1}
\end{equation}
where $ \mathcal{R}$ denotes the set of all $ r> 0 $ with the
property that there is a biholomorphic imbedding $ f :
\mathbf{B}^n \rightarrow X$ with $ f(\mathbf{B}^n) \supset
B^C_X(p, r)$, the ball in the Carath\'{e}odory metric with radius
$r$ around $ p \in X$. Evidently, $ h_X(p, \mathbf{B}^n) $ is a
biholomorphic invariant. The notation $ h_X(p, \mathbf{B}^n) $ is
to interpreted in the sense described above for the rest of this
section. The following result will be needed for our purposes. The
proof is exactly that of proposition \ref{A0} and is hence
omitted.

\begin{prop} \label{CC} Let $X$ be a c-hyperbolic manifold of
complex dimension $n$. Then

\begin{enumerate}

\item [(i)] if there is an $ x^0 \in X $ such that $ h_X(x^0 ,
\Omega) = 0 $, then $ h_X(x^0, \mathbf{B}^n) \equiv 0 $ and $ X$
is biholomorphically equivalent to $ \mathbf{B}^n $.

\item [(ii)] $ h_X ( \cdot, \mathbf{B}^n) $ is continuous on $X$.

\end{enumerate}

\end{prop}

\noindent The goal now will be to investigate the boundary
behaviour of $ h_X( \cdot, \mathbf{B}^n) $ as defined in
(\ref{CC1}) for a strongly pseudoconvex domain $X$. More
precisely, the following global statement can be proved:

\begin{thm} \label{CC2} Let $ D \subset
\mathbf{C}^n$ be a bounded strongly pseudoconvex domain with $
C^2$-smooth boundary. Then $ h_{D} (z, \mathbf{B}^n)
\rightarrow 0 $ as $ z \rightarrow \partial D$.
\end{thm}

\begin{proof} Let $ \{ z^k \} \subset D $ be a sequence converging
to $ z^0 \in \partial D $. It suffices to show that $ h_{D}( z^k )
\rightarrow 0 $ as $ {k \rightarrow \infty}$. Choose $ \zeta^k \in
\partial D$, closest to $ z^k$. Then $ \zeta^k \rightarrow z^0$ as
$ k \rightarrow \infty $.

\medskip

\noindent We now scale the domain $D$ with respect to the base
point $ z^0 \in \partial D$ and the sequence $ \{\zeta^k \} $. Let
$ h^k, T^k, D^k $ and $ D_{\infty}$ be defined as before. The
first step towards proving theorem \ref{CC2} is to control $
c_{D^k}(z, \cdot) $ as $ k \rightarrow \infty $. To do this, one
has to essentially repeat the arguments in
\cite{Seshadri&Verma-2006}. We indicate here the necessary changes
in \cite{Seshadri&Verma-2006} to infer stability of the
Carath\'{e}odory distance.

\begin{lem} \label{CC7} For any $ x^0 \in D_{\infty}$
\[
\displaystyle \lim_{k \rightarrow \infty} c_{D^k}(x^0, \cdot) =
c_{D_{\infty}}(x^0, \cdot).
\]
Moreover, the convergence is uniform on compact sets of $
D_{\infty}$.
\end{lem}

\begin{proof} Let $ K \subset D_{\infty} $ be compact and suppose that the
desired convergence does not occur. Then there exists a $ \epsilon
_0 > 0 $ and a sequence of points $ \{p^k \} \subset K $ which is
relatively compact in $D^k $ for $k$ large such that
\[
\big| c_{D ^k} ( x^0, p^k ) - c_{D_{\infty}} (x^0, p^k ) \big| >
\epsilon _0
\]
for all $k$ large. By passing to a subsequence, assume that  $ p^k
\rightarrow p^0 \in K $ as $ k \rightarrow \infty$. Since $
c_{D_{\infty}} ( x^0, \cdot ) $ is continuous, it follows that
\begin{equation}
\big| c_{D ^k} ( x^0, p^k ) - c_{D_{\infty}} (x^0, p^0 ) \big| >
\epsilon _0/2 \label{CC3}
\end{equation}
for all $k$ large.

\medskip

\noindent Let $ \phi^k : D^k \rightarrow \Delta $ be holomorphic
maps such that $ \phi^k(x^0) = 0 $ and $ d_{hyp} \big( 0,
\phi^k(p^k) \big) = c_{D^k} (x^0, p^k)$. The family $ \{ \phi^k \}
$ is uniformly bounded above. Additionally, since $ \{D^k \}$
converges to $ D_{\infty} $ in the Hausdorff sense, these maps are
defined on an arbitrary compact subset of $ D_{\infty} $ and hence
some subsequence of $ \{ \phi^k \} $, which we will be denoting by
the same symbols, converges to $ \phi : D_{\infty} \rightarrow
\overline{\Delta} $. Evidently, $ \phi(x^0) = 0 $. Then the
maximum principle forces that $ \phi: D_{\infty} \rightarrow
\Delta $.

\medskip

\noindent Note that $ \phi^k(p^k) \rightarrow \phi(p^0) $ as $ k
\rightarrow \infty$. As a consequence
\[
d_{hyp} \big ( 0, \phi^k(p^k) \big) \rightarrow d_{hyp} \big(0,
\phi(p^0) \big)
\]
as $ k \rightarrow \infty$. Fix $ \epsilon > 0 $ arbitrarily
small. Then it follows that
\[
c_{D^k} ( x^0, p^k) = d_{hyp} \big ( 0, \phi^k(p^k) \big) \leq
d_{hyp} \big(0, \phi(p^0) \big) + \epsilon
\]
for all $k$ large. However from the definition we have that
\[
d_{hyp} \big(0, \phi(p^0) \big) \leq c_{D_{\infty}} ( x^0, p^0).
\]
The above argument shows that
\begin{equation}
\displaystyle\limsup_{k \rightarrow \infty} c_{D^k} ( x^0, p^k)
\leq c_{D_{\infty}} ( x^0, p^0). \label{CC4}
\end{equation}
For the converse, fix $ \epsilon > 0 $ small and note that
\[
c_{D_{\infty}}(x^0, p^0) \leq d_{D_{\infty}} (x^0, p^0) \leq
d_{D^k} (x^0, p^k) + \epsilon
\]
for all $k$ large. The first inequality is well-known. The second
inequality follows from (\ref{2.2d}). Moreover, it is already
known that $ { (T^k \circ h^k ) }^{-1} (x^0) $ and $ { (T^k \circ
h^k ) }^{-1} (p^k) $ both approach $ z^0 \in \partial D $ as $ k
\rightarrow \infty$ and that
\[
d_{D^k}(x^0, p^k) = d_D \Big ( \big( T^k \circ h^k \big)^{-1}
(x^0), \big( T^k \circ h^k \big)^{-1}(p^k) \Big)
\]
for all $k$. Let $ U $ be a sufficiently small neighbourhood of $
z^0 \in \partial D$. If $U$ is small enough, $U \cap D $ is
strictly convex and it follows from Lempert's work
\cite{Lempert-1981} that $ c_D = d_D $ on $ U \cap D$. It now
follows that
\begin{eqnarray*}
d_D \Big( \big( T^k \circ h^k \big)^{-1} (x^0), \big( T^k \circ
h^k \big)^{-1}(p^k) \Big) & \leq & d_{U \cap D} \Big( \big( T^k
\circ h^k \big)^{-1} (x^0), \big( T^k \circ h^k \big)^{-1}(p^k)
\Big) \\
& = & c_{U \cap D} \Big( \big( T^k \circ h^k \big)^{-1} (x^0),
\big( T^k \circ h^k \big)^{-1}(p^k) \Big)
\end{eqnarray*}
Applying corollary 10.5.3 of \cite{Jarnicki&Pflug}, which is
essentially a statement about localizing the Carath\'{e}odory
distance near the boundary, we see that
\[
c_{U \cap D} \Big( \big( T^k \circ h^k \big)^{-1} (x^0), \big( T^k
\circ h^k \big)^{-1}(p^k) \Big) \leq (1+ \epsilon) c_{D} \Big(
\big( T^k \circ h^k \big)^{-1} (x^0), \big( T^k \circ h^k
\big)^{-1}(p^k) \Big)
\]
for all $k $ large. Since $ T^k \circ h^k $ are biholomorphisms
and hence Kobayashi isometries,
\[
c_{D} \Big( \big( T^k \circ h^k \big)^{-1} (x^0), \big( T^k \circ
h^k \big)^{-1}(p^k) \Big) = c_{D^k} ( x^0, p^k).
\]
Since $ \big \{ c_{D^k} ( x^0, p^k) \big \} $ is uniformly bounded
by (\ref{CC4}), it follows that
\begin{equation}
c_{D_{\infty}} (x^0 , p^0) \leq c_{D^k} ( x^0, p^k) + C \epsilon
\label{CC5}
\end{equation}
for all $k$ large. Combining (\ref{CC4}) and (\ref{CC5}), we see
that
\[
\displaystyle\lim_{k \rightarrow \infty} c_{D^k} ( x^0, p^k) =
c_{D_{\infty}} (x^0,p^0).
\]
This violates (\ref{CC3}) and hence the result.
\end{proof}

\medskip

\noindent Now, an argument similar to one in lemma \ref{B2} that
uses the fact that $ \big( D_{\infty}, c_{D_{\infty}} \big) $ is
complete and hence closed Carath\'{e}odory metric balls are
compact shows that

\begin{lem} \label{CC6} Fix $ x^0 \in D_{\infty}$ and $ R > 0$. Then
$ \big\{ B^C_{D^k} (x^0, R) \big\} $ converges to $
B^C_{D_{\infty}}(x^0, R)$ in the Hausdorff sense. Moreover, for
any $ \epsilon > 0 $ and for all $k$ large

\begin{enumerate}

\item [(i)] $ B^C_{D_{\infty}}(x^0, R) \subset B^C_{D^k} (x^0, R +
\epsilon),$

\medskip

\item [(ii)] $ B^C_{D^k} (x^0, R - \epsilon)  \subset
B^C_{D_{\infty}}(x^0, R).$

\end{enumerate}

\end{lem}

\medskip

\noindent \textit{ Proof of Theorem \ref{CC2}}: This can be proved
by making the relevant changes in the proof of theorem \ref{B0}
using lemmas \ref{CC7} and \ref{CC6}.
\end{proof}


\section{Behaviour of $h$ near weakly pseudoconvex points of
finite type in $ \mathbf{C}^2 $}

\noindent The main objective of this section is to establish the
following:

\begin{thm}\label{C0} Let $ D \subset \mathbf{C}^2 $ be a
smoothly bounded weakly pseudoconvex domain of finite type. Let $
\{ p^j \}$ be a sequence of points in $ D$ converging to $p^0 \in
\partial D$. Then
\[
 h_{D} (p^j) \rightarrow h_{D_{\infty}}((0, -1))
\]
as $ j \rightarrow \infty $ where $ D_{\infty}$ is the limiting
domain obtained by scaling $D$ with respect to the base point
$p^0$ and the sequence $ \{ p^j \}$.

\end{thm}

\begin{proof} There are two cases to be considered. After passing to a
subsequence if needed,

\begin{enumerate}

\item[(i)] $ \displaystyle\lim_{ j \rightarrow \infty} h_{D} (p^j)
= 0 $, or

\item[(ii)] $ \displaystyle\lim_{ j \rightarrow \infty} h_{D}
(p^j) > c $ for some positive constant $c$.

\end{enumerate}

\noindent In case (i) the domain $ D_{\infty} $ will turn out to
be biholomorphic to $ \mathbf{B}^n$ while this will not be the
case in (ii).

\medskip

\noindent It will be useful to briefly describe the scaling of $D$
and the corresponding model domain in terms of the base point
$p^0$ and the sequence $ p^j $ converging to $ p^0 $. These will
require some basic facts about the local geometry of a weakly
pseudoconvex domain in $ \mathbf{C}^2$ near a boundary point of
finite type.

\medskip

\noindent \textbf{Scaling the domain $D$:}

\medskip

\noindent Let $ D \subset \mathbf{C}^2$ be a smoothly bounded
pseudoconvex domain of finite type defined by $ \{ \rho(z,\bar{z})
= 0 \} $ for some smooth function $ \rho $. We may assume that $
p^0 = (0,0) $ and $ \nabla \rho (0,0) = (0,1).$ Then there exists
a local coordinate system in a neighbourhood of $(0,0)$ such that
the domain $ D$ can be written as
\[
\Big \{ (z_1,z_2) \in \mathbf{C}^2 : 2 \Re z_2 +
H_{2m}(z_1,\bar{z}_1 ) + o ( |z| ^{2m} + \Im z_2 ) < 0 \Big \}
\]
where $H_{2m}$ is a homogeneous subharmonic polynomial of degree $
2m \geq 2 $ in $ z_1$ and $\bar{z}_1 $ which does not contain any
harmonic terms. Choose $ \zeta^j \in \partial D$ defined by
\[
\zeta^j = p^j + (0, \epsilon_j), \qquad \epsilon_j > 0
\]
From \cite{Catlin-1989} it follows that there exists a sequence $
\{ \phi ^{\zeta^j} \} $ of automorphisms of $ \mathbf{C}^2 $
defined by
\begin{equation*}
\phi ^{\zeta^j}( z_1, z_2) = \Bigg( z_1 - \zeta^j_1, \Big (z_2 -
\zeta^j_2 - \displaystyle\sum_{l=1}^{2m} d^l( \zeta^j) ( z_1 -
\zeta^j_1 )^l \Big) \left (d^0(\zeta^j)\right) ^{-1} \Bigg)
\label{5.1}
\end{equation*}
where $ d^l(\zeta^j)$ are non-zero functions depending smoothly on
$ \zeta^j$ and $ d^0(\zeta^j) \rightarrow 1 $ as $ j \rightarrow
\infty$. Observe that
\[
\phi ^{\zeta^j} (\zeta^j) = (0,0), \ \phi ^{\zeta^j}(p^j) =
\Big(0, - \epsilon_j \big(d^0(\zeta^j)\big)^{-1} \Big)
\]
and the defining function for $ \phi ^{\zeta^j}(\partial D) $
around the origin is
\[
2 \Re z_2 + \displaystyle\sum _{l=2} ^{2m} P_{l,\zeta^j}
(z_1,\bar{z}_1) + R _{\zeta^j} ( \Im z_2, z_1) = 0
\]
where $ P_{l,\zeta^j} (z_1,\bar{z}_1)$ are real-valued homogeneous
polynomials of degree $l$ without any harmonic terms. Also, $
P_{l,\zeta^j} (z_1,\bar{z}_1) \rightarrow 0 $ for $ l < 2m $ and $
P_{2m,\zeta^j} (z_1,\bar{z}_1) \rightarrow H_{2m}(z_1,\bar{z}_1)$
as $ j \rightarrow \infty$. Let $ \| \cdot \| $ be a fixed norm on
the finite dimensional space of all real-valued polynomials on the
complex plane with degree at most $2m$ that do not contain any
harmonic terms. Define
\[
\tau ( \zeta^j, \epsilon_j) = \min_{ 2 \leq l \leq 2m } \left (
\frac{\epsilon_j}{ \| P_{l,\zeta^j} (z_1,\bar{z}_1) \| } \right)
^{1/l}.
\]
Since $ P_{2m, \zeta^j} \rightarrow H $ which is a non-zero
polynomial, it follows that $ \displaystyle\sup_{j}
\big({\epsilon_j}^{-1} \tau ( \zeta^j, \epsilon_j)^{2m} \big) <
\infty $. Let $ \Delta_{ \zeta^j}^{\epsilon_j}: \mathbf{C}^2
\rightarrow \mathbf{C}^2 $ be a sequence of dilations defined by
\begin{equation*}
\Delta_{ \zeta^j}^{\epsilon_j}(z_1,z_2)  = \left( \frac{z_1} {\tau
( \zeta^j, \epsilon_j)}, \frac{z_2}{ \epsilon_j } \right).
\end{equation*}
A useful set for approximating the geometry of $D$ near $ p^0$ is
the Catlin's bidisc $Q (\zeta^j, \epsilon_j )$ determined by the
quantities $ \tau( \zeta^j, \epsilon_j) $ where
\begin{equation*}
Q (\zeta^j, \epsilon_j ) = \Big(\Delta_{ \zeta^j}^{\epsilon_j}
\circ \phi ^{\zeta^j} \Big)^{-1}( \Delta \times \Delta ).
\end{equation*}
Then the domains $ D ^j = \Delta_{ \zeta^j}^{\epsilon_j} \circ
\phi ^{\zeta^j}( D)$ converge in the Hausdorff metric to
\[
D_{\infty} = \Big \{ (z_1, z_2) \in \mathbf{C}^2 : 2 \Re z_2 +
P_{\infty} (z_1, \bar{z}_1) < 0 \Big \}
\]
where
\[
P_{\infty} (z_1, \bar{z}_1)= \displaystyle \lim_{ j \rightarrow
\infty} \frac{1}{ \epsilon_j} \displaystyle\sum _{l=2} ^{2m} \tau
( \zeta^j, \epsilon_j)^l P_{l,\zeta^j} (z_1,\bar{z}_1)
\]
is a real-valued subharmonic polynomial of degree at most $2m$
without harmonic terms. Note that if the sequence $ p^j$ converges
normally to the point $ p^0$, i.e., $ p^j = p^0 - \epsilon_j
n(p^0)$ where $n(p^0)$ denotes the unit outward normal to $
\partial D$ at $p^0$, then it turns out that $ P_{\infty}
(z_1,\bar{z}_1) \equiv H_{2m}(z_1,\bar{z}_1)$. Therefore,
\[
D_{\infty} = \{ (z_1, z_2) \in \mathbf{C}^2 : 2 \Re z_2 + H_{2m}
(z_1, \bar{z}_1) < 0 \}.
\]

\medskip

\noindent \textbf{Stability of the infinitesimal Kobayashi
metric:}

\medskip

\begin{lem} \label{K1} For $(a,v) \in D_{\infty} \times
\mathbf{C}^n$,
\begin{equation*}
\displaystyle\lim_{j \rightarrow \infty} F^K_{D^j} (a,v) =
F^K_{D_{\infty}} (a,v). \label{7.1}
\end{equation*}
Moreover, the convergence is uniform on compact sets of $
D_{\infty} \times \mathbf{C}^n$.
\end{lem}

\begin{proof} Let $ S \subset D_{\infty}$ and $ G \subset \mathbf{C}^n$ be
compact and suppose that the desired convergence does not occur.
Then there is a $ \epsilon_0 > 0$ such that after passing to a
subsequence, if necessary, we may assume that there exists a
sequence of points $ \{ a^j \} \subset S $ which is relatively
compact in $ D^j$ and a sequence $ \{ v^j \} \subset G$ such that
\begin{equation*}
\big| F^K_{ D^j }( a^j, v^j ) - F^K _{D_{\infty}}( a^j, v^j )
\big|
> \epsilon_0
\end{equation*}
for $j$ large. Additionally, $ a^j \rightarrow a \in S $ and $ v^j
\rightarrow v \in G $ as $ j \rightarrow \infty$. Since $
F^K_{D_{\infty}} ( a, \cdot) $ is homogeneous, we may assume that
$ | v^j| = 1 $ for all $j$. Observe that $ D_{\infty} $ is
complete hyperbolic and hence taut. The tautness of $ D_{\infty}$
implies via a normal family argument that $ F^K_{D_{\infty}} (
\cdot, \cdot) $ is jointly continuous, $ 0 < F^K_{D_{\infty}}
(a,v) < \infty $ and there exists a holomorphic extremal disc $ g
: \Delta \rightarrow D_{\infty} $ that by definition satisfies $
g(0) = a, {g}'(0) = \mu v $ where $ \mu
> 0 $ and $ F^K_{D_{\infty}} (a,v)= 1/{\mu} $. Hence
\begin{equation}
\big| F^K_{ D^j }( a^j, v^j ) - F^K_{D_{\infty}} (a,v) \big| >
\epsilon_0 / 2     \label{k1}
\end{equation}
for $j$ sufficiently large. Fix $ \delta \in (0,1) $ and define
the holomorphic mappings $ g^j : \Delta \rightarrow \mathbf{C}^2 $
by
\[
g^j (z) = g \left( (1 - \delta) z \right) + (a^j - a) + \mu ( 1 -
\delta)z( v^j - v ).
\]
Since the image $ g \left(( 1 - \delta)\Delta \right)$ is
compactly contained in $ D_{\infty}$ and $ a^j \rightarrow a, v^j
\rightarrow v$ as $ j \rightarrow \infty$, it follows that $ g^j:
\Delta \rightarrow D^j $ for $j$ large. Also, $ g^j(0) = g(0) +
a^j - a = a^j$ and $ (g^j)'(0) = ( 1 - \delta) g'(0) + \mu (1 -
\delta)(v^j - v) = \mu (1 - \delta) v^j $. By the definition of
the infinitesimal metric it follows that
\begin{equation*}
F^K_{ D^j }( a^j, v^j ) \leq \frac{1}{ \mu (1 - \delta) } =
\frac{F^K_{D_{\infty}} (a,v)}{\mu (1 - \delta) }.
\end{equation*}
Letting $ \delta \rightarrow 0^+ $ yields
\begin{equation}
\displaystyle \limsup_{j \rightarrow \infty}F^K_{ D^j }( a^j, v^j
) \leq F^K_{D_{\infty}} (a,v). \label{4.4}
\end{equation}
Conversely, fix $\epsilon > 0$ arbitrarily small. By definition,
there are holomorphic mappings $ f^j: \Delta \rightarrow D^j$
satisfying $ f^j(0) = a^j$ and $ (f^j)'(0) = \mu^j $ where $ \mu^j
> 0 $ and
\begin{equation}
F^K_{ D^j }( a^j, v^j ) \geq \frac{1}{ \mu^j} - \epsilon
\label{4.5}
\end{equation}
The sequence $ \{ f^j \} $ has a subsequence that converges to a
holomorphic mapping $ f : \Delta \rightarrow D_{\infty}$ uniformly
on compact sets of $ \Delta $. To see this, consider $\Delta(0,r)$
for $r \in (0,1)$. We may assume that $S$ is compactly contained
in $ \Delta(0,C_1^{1/2m}) \times \Delta (0,C_1)$ for some $ C_1 >
1$. As a consequence
\begin{eqnarray*}
 a^j  \in \Delta(0,C_1^{1/2m}) \times \Delta (0,C_1) \subset
\Delta \left( 0, \frac{\tau ( \zeta^j, C_1 \epsilon_j)} { \tau (
\zeta^j, \epsilon_j)} \right) \times \Delta \left( 0, \frac{C_1
\epsilon_j}{ \epsilon_j} \right)
\end{eqnarray*}
for all $j$. In particular, for all $j$
\[
(\Delta_{ \zeta^j}^{\epsilon_j} \circ \phi ^{\zeta^j})^{-1} (a^j)
\in Q (\zeta^j, C_1 \epsilon_j).
\]
Also, note that
\[
( \Delta_{ \zeta^j}^{\epsilon_j} \circ \phi ^{\zeta^j})^{-1} ( a^j
) \rightarrow p^0 \in \partial D
\]
as $ j \rightarrow \infty $. Now, applying proposition $1$ in
\cite{Berteloot&Coeure-1991} to the mappings
\[
(\Delta_{\zeta^j}^{\epsilon_j} \circ \phi ^{\zeta^j})^{-1} \circ
f^j : \Delta \rightarrow D
\]
shows that there exists a uniform positive constant $ C_2= C_2(r)
$ with the property that
\[
\big(\Delta_{ \zeta^j}^{\epsilon_j} \circ \phi ^{\zeta^j}
\big)^{-1} \circ f^j \big(\Delta(0, r)\big) \subset Q \big
(\zeta^j, C_2 C_1 \epsilon_j \big)
\]
or equivalently that
\[
f^j \big( \Delta(0,r)\big ) \subset \Delta \left(0, \sqrt{C_1
C_2}\right) \times \Delta(0, C_1 C_2 ).
\]
Therefore, $ \{ f^j \} $ is a normal family. Hence, the sequence $
\{ f^j \}$ has a subsequence that converges uniformly on compact
sets of $ \Delta $ to a holomorphic mapping $ f : \Delta
\rightarrow \mathbf{C}^2$ or $ f \equiv \infty$. The latter cannot
be true since $ f(0) = a $. It remains to show that $ f : \Delta
\rightarrow D_{\infty}$. For this note that $ D^j $ are defined by
\[
2 \epsilon_j \ \Re z_2 + \displaystyle\sum_{l=2}^m \tau ( \zeta^j,
\epsilon_j)^l P_{l,\zeta^j} (z_1,\bar{z}_1) + R_{\zeta^j } \big(
\epsilon_j \Im z_2, \tau ( \zeta^j, \epsilon_j) z_1 \big) < 0
\]
where $ R_{\zeta^j } \big(\epsilon_j \Im z_2, \tau ( \zeta^j,
\epsilon_j) z_1 \big) = \epsilon_j o(1) $ and the term $ o(1)$ is
uniformly convergent to zero as $ j \rightarrow \infty$. Thus, for
$w \in \Delta(0,r)$ and $ r \in (0,1)$
\[
2 \epsilon_j \ \Re \big( f^j_2 (w)\big) +
\displaystyle\sum_{l=2}^m \tau ( \zeta^j, \epsilon_j)^l
P_{l,\zeta^j} \left(f^j_1(w),\overline{f^j_1(w)} \right) +
R_{\zeta^j } \left( \epsilon_j \ \Im \big( f^j_2(w) \big), \tau (
\zeta^j, \epsilon_j) f^j_1(w) \right) < 0.
\]
Letting $ j \rightarrow \infty$ yields
\[
2 \ \Re \big(f_2(w)\big) + P_{\infty} \big(f_1(w), \overline{ f_1
(w) } \big) \leq 0
\]
or equivalently that $ f ( \Delta (0,r)) \subset \overline{
D}_{\infty}$. Since $ r \in (0,1) $ was arbitrary, it follows that
$ f( \Delta ) \subset \overline{ D}_{\infty} $. Since $ f(0,0) = a
$ the maximum principle forces that $ f : \Delta \rightarrow
D_{\infty} $. Note that
\[
f'(0) = \displaystyle \lim_{j \rightarrow \infty} (f^j)'(0) =
\displaystyle \lim_{j \rightarrow \infty} \mu ^j v^j= \mu v
\]
for some $ \mu > 0$. It follows from the definition of the
infinitesimal metric that
\[
F^K_{D_{\infty}} (a,v) \leq 1/ {\mu}.
\]
The above observation together with (\ref{4.5}) yields
\begin{equation}
\displaystyle \liminf_{j \rightarrow \infty}F^K_{ D^j }( a^j, v^j
) \geq F^K_{D_{\infty}} (a,v). \label{4.6}
\end{equation}
Combining (\ref{4.4}) and (\ref{4.6}) shows that
\begin{equation*}
\displaystyle \lim_{j \rightarrow \infty}F^K_{ D^j }( a^j, v^j ) =
F^K_{D_{\infty}} (a,v)
\end{equation*}
which contradicts the assumption (\ref{k1}) and proves the lemma.
\end{proof}

\begin{rem} This lemma does not directly follow from \cite{Yu-1995}
since there is no taut domain that contains all the scaled domains
$ D^j$.
\end{rem}

\noindent Write $ (0,-1) = z^0 $ and $ \big(0, -1/d^0(\zeta^j)
\big) = z^j$ for brevity.

\medskip

\noindent \textit{ Proof of Theorem \ref{C0}(i):} For each $j$,
let $ 1/R_j$ be a positive number that almost realizes $ h_{D}
(p^j) $, i.e., $ 1/ R_j < h_{D} (p^j) + \epsilon$ for some fixed $
\epsilon > 0$. Evidently, the sequence $ R_j \rightarrow \infty$
and there exists a sequence of biholomorphic imbeddings $ F^j :
\mathbf{B}^2 \rightarrow D $ satisfying $ F^j (0) = p^j $ and $
B_{D}( p^j, R_j) \subset F^j( \mathbf{B}^2)$. Consider the dilated
maps
\[
\psi^j := \Delta_{ \zeta^j}^{\epsilon_j} \circ \phi ^{\zeta^j}
\circ F^j : \mathbf{B}^2 \rightarrow D^j.
\]
Note that $ \Delta_{ \zeta^j}^{\epsilon_j} \circ \phi ^{\zeta^j}
\circ F^j (0,0) = \big(0, -1/d^0(\zeta^j) \big) \rightarrow
(0,-1)$ as $ j \rightarrow \infty$. In this setting, theorem 2 of
\cite {Berteloot&Coeure-1991} (see proposition 2.2 in \cite
{Berteloot-1994} also) shows that the sequence $ \{ \Delta_{
\zeta^j}^{\epsilon_j} \circ \phi ^{\zeta^j} \circ F^j \} $ admits
a subsequence that will still be denoted by the same indices, that
converges uniformly on compact sets of $ \mathbf{B}^2 $ to a
holomorphic mapping $ \psi : \mathbf{B}^2 \rightarrow \mathbf{C}^2
$. Now an argument similar to the one in the proof of lemma
\ref{K1} shows that $ \psi : \mathbf{B}^2 \rightarrow D_{\infty}
$.

\medskip

\noindent Then $ \psi $ is a biholomorphism. To establish this, it
will suffice to show that for each $ \epsilon > 0 $,
\begin{equation}
B_{D_{\infty}} ( z^0, R - \epsilon ) \subset B_{D^j} ( z^j, R)
\label{k8}
\end{equation}
for all $ R > 0 $ and all $j$ large and this will follow from
\[
\displaystyle\limsup_{j \rightarrow \infty} d_{D^j}( z^j, \cdot)
\leq d_{D_{\infty}}(z^0, \cdot).
\]
For this fix $ q \in D_{\infty}$ and let $ \gamma : [0,1]
\rightarrow D_{\infty} $ be a piecewise $C^1$-smooth path in $
D_{\infty} $ such that $ \gamma(0) = z^0, \gamma(1) = q$ and
\[
\int_0^1 { F^K_{D_{\infty}} \big( \gamma(t), \dot{\gamma}(t) \big)
dt} \leq d_{D_{\infty}} ( z^0, q) + \epsilon/2.
\]
Define $ \gamma^j: [0,1] \rightarrow \mathbf{C}^2 $ by
\[
\gamma^j(t) = \gamma(t) + (z^j - z^0)(1-t).
\]
Since the trace of $ \gamma $ is relatively compact in $
D_{\infty}$ and $ z^j \rightarrow z^0 $, it follows that the trace
of $\gamma $ is contained uniformly relatively compactly in $ D^j
$ for all large $j$. Note that $ \gamma^j(0) = \gamma(0) + z^j -
z^0 = z^j $ and $ \gamma^j(1) = q $. In addition, $ \gamma^j
\rightarrow \gamma $ and $ \dot{\gamma}^j \rightarrow \dot{\gamma}
$ uniformly on $ [0,1]$. It follows from lemma \ref{K1} that
\[
\int_0^1 {F^K_{D^j} \big( \gamma^j(t), \dot{\gamma}^j(t) \big) dt}
\leq \int_0^1 { F^K_{D_{\infty}} \big( \gamma(t), \dot{\gamma}(t)
\big) dt} + \epsilon/2 \leq d_{D_{\infty}} ( z^0, q) + \epsilon.
\]
Consequently,
\[
d_{D^j} (z^j, q) \leq \int_0^1 {F^K_{D^j} \big( \gamma^j(t),
\dot{\gamma}^j(t) \big) dt} \leq d_{D_{\infty}} (z^0, q) +
\epsilon
\]
which implies that
\[
\displaystyle\limsup_{j \rightarrow \infty} d_{D^j}( z^j, \cdot)
\leq d_{D_{\infty}}(z^0, \cdot).
\]

\noindent Note that $ B_{ D } ( p^j, R_j) \subset F^j(
\mathbf{B}^2)$. Since $ \Delta_{ \zeta^j}^{\epsilon_j} \circ \phi
^{\zeta^j}$ are biholomorphisms and hence Kobayashi isometries, it
follows that
\begin{eqnarray*}
B_{ D^j} ( z^j, R_j) \subset \Delta_{ \zeta^j}^{\epsilon_j} \circ
\phi ^{\zeta^j} \circ F^j (\mathbf{B}^2). \label{4.1}
\end{eqnarray*}
Since $ ( D_{\infty}, d_{D_{\infty}}) $ is complete, it is
possible to write
\begin{eqnarray*}
D_{\infty} = \displaystyle \bigcup_{ \nu = 1} ^{\infty}
B_{D_{\infty}} \big( (0,-1), \nu \big) \label{4.3}
\end{eqnarray*}
which is an exhaustion of $ D_{\infty}$ by an increasing union of
relatively compact domains. Consider
\[
\theta^j := \big(\Delta_{ \zeta^j}^{\epsilon_j} \circ \phi
^{\zeta^j} \circ F^j \big) ^{-1}: \Delta_{ \zeta^j}^{\epsilon_j}
\circ \phi ^{\zeta^j} \circ F^j (\mathbf{B}^2) \rightarrow
\mathbf{B}^2.
\]
These mappings are evidently defined on an arbitrary compact
subset of $ D_{\infty} $ for large $j$ and hence some subsequence
of $ \{ \theta^j \}$ converges to $ \theta: D_{\infty} \rightarrow
\overline{\mathbf{B}}^2$. Moreover, $ \theta (0,-1) = (0,0) $
together with the maximum principle shows that $ \theta :
D_{\infty} \rightarrow \mathbf{B}^2 $. Finally observe that for
$w$ in a fixed compact set in $ D_{\infty} $,
\begin{eqnarray*}
| \psi \circ \theta (w) - w | & = & | \psi \circ \theta (w) -
\psi^j \circ \theta^j(w) | \\
& = & | \psi \circ \theta (w) - \psi \circ \theta^j (w) | + |
\psi \circ \theta^j(w) - \psi^j \circ \theta^j | \\
& \rightarrow & 0 \ \mbox{as } j \rightarrow \infty
\end{eqnarray*}
This shows that $ \psi \circ \theta = id $. Similarly, it can be
proved that $ \theta \circ \psi = id$. This shows that $
D_{\infty} $ is biholomorphically equivalent to $ \mathbf{B}^2$.
In particular, $ h_{D_{\infty}} (\cdot) \equiv 0 $ so that $ h_{D}
(p^j) \rightarrow h_{D_{\infty}}((0, -1))$ as $ j \rightarrow
\infty $. This completes the proof of case (i). \qed

\medskip

\noindent Case (ii) differs from (i) in one important way. To
prove that the integrated Kobayashi distance is stable under
scaling in the strongly pseodoconvex case, Lempert's theorem that
guarantees the existence of complex geodesics in strongly convex
domains was used. This approach will evidently not work for weakly
pseudoconvex domains. To control the integrated Kobayashi distance
in weakly pseudoconvex domain under scaling, the following two
ingredients will be required that serve to circumvent the need for
complex geodesics.

\medskip

\begin{lem} \label{D2} Let $ D$ be a Kobayashi hyperbolic
domain in $ \mathbf{C}^n$ with a subdomain $ D' \subset D$. Let $
p,q \in D'$, $ d_{D}(p,q) = a$ and $ b > a$. If $ D'$ satisfies
the condition $ B_{D} (q,b) \subset D'$, then the following two
inequalities hold:
\begin{eqnarray*}
d_{D'} (p,q) \leq \frac{1}{\tanh (b-a) } d_{D}(p,q), \\
F^K_{D'} (p, v) \leq \frac{1}{\tanh (b-a) } F^K_{D}(p,v).
\end{eqnarray*}
\end{lem}

\noindent The reader is referred to \cite{Kim&Krantz-2008} (or
\cite{Kim&Ma-2003}) for a proof, but it should be noted that this
statement emphasizes an upper bound for $ d_{D'} $ in terms of $
d_D$. An estimate with the inequality reversed is an immediate
consequence of the definition of the Kobayashi metric.

\medskip

\noindent The second ingredient is an estimate for the Kobayashi
metric between two points in a weakly pseudoconvex finite type
domain $D$ in $ \mathbf{C}^2$ due to Herbort
(\cite{Herbort-2005}). To state this, let $ d(\cdot, \partial D)$
be the Euclidean distance to the boundary and $ \rho $ a smooth
defining function for $ \partial D$. For $ a, b \in D$, define
\begin{eqnarray*}
\rho^*(a,b) & = & \log \left ( 1 + \frac{ d(a,b)}{d(a, \partial
D)} + \frac{| \langle L(a), a -b \rangle|}{ \tau ( a, d(a,
\partial
D)) } \right) \\
L(a) & = & \left( - \frac{ \partial \rho} {\partial z_2}(a),
\frac{ \partial \rho} {\partial z_1}(a) \right) \\
d(a,b) & = & \min \big \{ d'(a,b), |a-b| \big \} \\
d'(a,b) & = & \inf \big \{ \delta > 0 : a \in Q ( b, \delta) \big
\},
\end{eqnarray*}
where $ \langle \cdot, \cdot \rangle $ denotes the standard
hermitian inner product in $ \mathbf{C}^2 $.

\medskip

\noindent The main result of \cite{Herbort-2005} that is needed
is:

\begin{thm}
Assume that $ D = \{ \rho < 0 \} \subset \mathbf{C}^2$ be a
bounded pseudoconvex domain with smooth boundary such that all
boundary points are of finite type. Then there exists a positive
constant $ C_* $ such that for any two points $ a, b \in D$
\[
C_* \big ( \rho^*(a,b) + \rho^*(b,a) \big) \leq d_D(a,b) \leq
1/C_* \big ( \rho^*(a,b) + \rho^*(b,a) \big).
\]
\end{thm}

\medskip

\noindent Observe that in case (ii) the largest radii admissible
in the definition of the Fridman's invariant function $ h_{D}
(p^j)$ is at most $\nu_0$ where $ \nu_0 = 1/c$. Several lemmas
will be needed to complete the proof in this case. We first note
the following:

\begin{lem} \label{K2} For all $ R > 0 $ and for all $j$ large,
$B_{D^j} ( z^j, R ) $ is compactly contained in $ D_{\infty}$.
\end{lem}

\begin{proof} The proof divides into two parts. In the first part we show
that the sets $ B_{D^j} ( z^j, R) $ cannot accumulate at the point
at infinity in $ \partial D_{\infty} $ and in the second part we
show that the sets $ B_{D^j} ( z^j, R) $ do not cluster at any
finite boundary point. First note that
\[
B_{D^j} ( z^j, R ) = \Delta_{ \zeta^j}^{\epsilon_j} \circ \phi
^{\zeta^j} \big(B_{D}( p^j, R) \big)
\]
Assume that $ q \in B_{D}( p^j, R)$. Using Herbort's lower
estimate for the Kobayashi metric gives us
\begin{equation*}
C_* \left( \rho^* (p^j, q) + \rho^* (q, p^j)\right) \leq d_{D}
(p^j,q).  \label{4.7}
\end{equation*}
As a consequence
\[
d(p^j,q) < \exp( R/ { C_*} )d(p^j, \partial D)
\]
which in turn implies that

\begin{itemize}

\item either $ |p^j - q| < d( p^j, \partial D) \exp( R/{C_*}) $ or

\medskip

\item for each $j$, there exists a $ \delta_j \in ( 0, d( p^j,
\partial D)\exp( R/{C_*}) ) $ such that $ p^j \in Q (q, \delta_j) $.

\end{itemize}
It follows from proposition 1.7 in \cite{Catlin-1989} that there
exists a uniform positive constant $C$ such that for each $j$, the
following holds: if $ p^j \in Q (q, \delta_j)$, then $ q \in Q
\big( p^j, C \delta_j \big) $. Hence, the second statement above
can be rewritten as: there exists a positive constant $C$ such
that for each $j$, there exists a $ \delta_j \in ( 0, d( p^j,
\partial D)\exp( R/{C_*}) ) $ with the property that
\[
q \in (\phi^{p^j})^{-1} \Big ( \Delta(0, \tau( p^j, C \delta_j) )
\times \Delta(0, C \delta_j ) \Big).
\]
Said differently, $ B_{D^j} (p^j, R)$ is contained in the union
\begin{eqnarray*}
B_{D^j} (p^j, R) \subset  B \Big( p^j, d(p^j, \partial D) \exp{(
R/ C_*)} \Big) \cup (\phi^{p^j})^{-1} \Big ( \Delta(0, \tau( p^j,
C \delta_j) ) \times \Delta(0, C \delta_j ) \Big)
\end{eqnarray*}
with $ \delta_j $ as described above. Now, using the explicit
expression for $ \phi^{\zeta^j} $ we get
\[
\phi^{\zeta^j} \left \lbrace (z_1,z_2) \in \mathbf{C}^2: | z_1 -
p^j_1 | ^2 + | z_2 - p^j_2 | ^2 < \big(d( p^j, \partial D)\big)^2
\exp(2R/{C_*}) \right \rbrace =
\]
\begin{equation*}
\left \lbrace (w_1,w_2): |w_1|^2 + \left| d^0(\zeta^j) w_2 +
\epsilon_j + \displaystyle\sum_{l=1}^{2m} d^l(\zeta^j) w_1^l
\right|^2 < \big (d( p^j, \partial D) \big)^2 \exp(2R/{C_*})
\right \rbrace.
\end{equation*}
Hence
\[
\Delta_{\zeta^j}^{\epsilon_j} \circ \phi^{\zeta^j} \left \lbrace
(z_1,z_2) \in \mathbf{C}^2: | z_1 - p^j_1 | ^2 + | z_2 - p^j_2 |
^2 < \big( d( p^j, \partial D) \big) ^2 \exp(2R/{C_*}) \right
\rbrace =
\]
\begin{equation}
\left \lbrace w: |w_1|^2 + \left(\frac{\epsilon_j}{\tau(\zeta^j,
\epsilon_j)} \right) ^2 \left| d^0(\zeta^j) w_2 + 1 +
{\epsilon_j}^{-1} \Big( \displaystyle \sum_{l=1}^{2m} \alpha^{j,l}
w_1^l \Big) \right|^2 < \frac{ \big( d( p^j,\partial D) \big) ^2
\exp(2R/{C_*})}{{\tau(\zeta^j, \epsilon_j)}^2} \right \rbrace
\label{5.3}
\end{equation}
where
\[
\alpha^{j,l} = d^l(\zeta^j) {\tau(\zeta^j, \epsilon_j)}^l.
\]

\medskip

\noindent If $ w= (w_1,w_2) $ belongs to the set described by
(\ref{5.3}) above, then
\begin{eqnarray}
|w_1|  & \leq &  \frac{ d( p^j,
\partial D)\exp(R/{C_*})}{\tau(\zeta^j, \epsilon_j)} \lesssim
\frac{
\epsilon_j \exp(R/{C_*})}{\tau(\zeta^j, \epsilon_j)} \qquad \mbox{and} \label{W1} \\
\left| d^0(\zeta^j) w_2 + 1 + {\epsilon_j}^{-1} \Big(
\displaystyle \sum_{l=1}^{2m} \alpha^{j,l} w_1^l \Big) \right| &
\leq  & \frac{ d( p^j, \partial D) \exp(R/{C_*)} } { \epsilon_j }
\lesssim \exp(R/{C_*)}. \label{W2}
\end{eqnarray}
Moreover, for $\delta_j \in \big( 0, d( p^j, \partial D) \exp(
R/{C_*}) \big)$,
\[
(\phi^{p^j})^{-1} \Big \lbrace (z_1,z_2) \in \mathbf{C}^2 : |z_1|
< \tau( p^j, C \delta_j) ,|z_2| <  C \delta_j  \Big \rbrace =
\]
\begin{equation*}
\left \lbrace (w_1,w_2): |w_1 - p^j_1| < \tau( p^j, C \delta_j),
\left| w_2 - p^j_2 - \displaystyle\sum_{l=1}^{2m} d^l(p^j) ( w_1 -
p^j_1)^l \right| < C \delta_j d^0(p^j) \right \rbrace
\end{equation*}
so that
\[
\Delta_{\zeta^j}^{\epsilon_j} \circ \phi^{\zeta^j} \circ
(\phi^{p^j})^{-1} \Big (\Delta(0, \tau( p^j, C \delta_j) ) \times
\Delta(0, C \delta_j ) \Big) =
\]
\begin{equation}
\left \lbrace w: |w_1| < \frac{ \tau( p^j, C \delta_j)}{\tau(
\zeta^j, \epsilon_j)}, \left| d^0(\zeta^j) w_2 + 1 +
{\epsilon_j}^{-1}  \Big( \displaystyle\sum_{l=1}^{2m}  \beta^{j,l}
w_1^l \Big) \right| < \frac{C \delta_j d^0(p^j)} { \epsilon_j}
\right \rbrace \label{W3}
\end{equation}
where
\[
\beta^{j,l} = \left (d^l(\zeta^j) - d^l(p^j) \right){\tau(\zeta^j,
\epsilon_j)}^l.
\]

\medskip

\noindent If $ w = (w_1,w_2) $ belongs to the set given by
(\ref{W3}), then
\begin{eqnarray}
|w_1| & < & \frac{ \tau( p^j, C \delta_j)}{\tau( \zeta^j,
\epsilon_j)} \qquad \mbox{and} \label{W4} \\
\left| d^0(\zeta^j) w_2 + 1 + {\epsilon_j}^{-1}  \Big(
\displaystyle\sum_{l=1}^{2m} \beta^{j,l} w_1^l \Big) \right| & < &
\frac{C \delta_j d^0(p^j)} { \epsilon_j } \nonumber \\
& < & \frac{C d( p^j,\partial D) \exp(R/{C_*}) d^0(p^j)} {
\epsilon_j} \nonumber \\
& \lesssim & \exp(R/{C_*}) d^0(p^j). \label{W5}
\end{eqnarray}

\noindent It follows from Catlin's work that

\begin{itemize}

\item $ {\epsilon_j}^{1/2} \lesssim \tau( \zeta^j, \epsilon_j)
\lesssim {\epsilon_j}^{1/2m}$

\medskip

\item $ \tau(p^j, C \epsilon_j) \approx \tau( \zeta^j,\epsilon_j)$

\medskip

\item $ | d^l(\zeta^j)| \lesssim \epsilon_j ( \tau(\zeta^j,
\epsilon_j) )^{-l} $ for all $ 1 \leq l \leq 2m $

\medskip

\item $ | d^l(p^j)| \lesssim \epsilon_j ( \tau(p^j,
\epsilon_j))^{-l} $ for all $ 1 \leq l \leq 2m$

\medskip

\item $ d^0(\zeta^j) \approx 1 $ and $ d^0(p^j) \approx 1 $.

\end{itemize}

\noindent These estimates together with (\ref{W1}), (\ref{W2}),
(\ref{W4}) and (\ref{W5}) show that if $ w = (w_1, w_2) $ belongs
either of (\ref{5.3}) or (\ref{W3}), then $ |w| $ is uniformly
bounded. In other words, the sets
\[
\Delta_{\zeta^j}^{\epsilon_j} \circ \phi^{\zeta^j} \left( B( p^j,
d(p^j, \partial D) \exp(R/ C_*)) \right) \bigcup
\Delta_{\zeta^j}^{\epsilon_j} \circ \phi^{\zeta^j} \left(
(\phi^{p^j})^{-1} \left ( \Delta(0, \tau( p^j, C \delta_j) )
\times \Delta(0, C \delta_j ) \right) \right)
\]
are uniformly bounded. Therefore, $B_{D^j} ( z^j, R)$ as a set
cannot cluster at the point at infinity on $ \partial D_{\infty}$.

\medskip

It remains to show that $B_{D^j} ( z^j, R )$ do not cluster on the
finite part of $ \partial D$. Suppose there exists a sequence of
points $ \{ q^j\}, q^j \in B_{D^j} ( z^j, R )$ such that $ q^j
\rightarrow q^0$ as $j \rightarrow \infty$ for some $q^0$ a finite
boundary point of $ D_{\infty}$. Then proposition 4.1 of
\cite{Coupet&Sukhov-1997} shows that there exists a neighbourhood
$V$ of $q^0$ in $ \mathbf{C}^2$ and a uniform positive constant
$C$ such that for all $z \in V \cap D^j$ and $v$ a tangent vector
at $z$,
\begin{equation}
 F^K_{D^j}(z, v) \geq  C \frac{|v|}{d(z, \partial D^j)^{1/2m}} \label{5.5}
\end{equation}
for all $j$ large. Choose a neighbourhood $ \tilde{V} $ of $ z^0$
which is compactly contained in $ D$ and disjoint from $V$. We may
assume that $ \{ z^j \} \subset \tilde{V} $ for all $ j $ large.
Let $ \gamma^j$ be a arbitrary piecewise $ C^1$ curve in $ D^j$
joining $q^j$ and $ z^j$. As we travel along $ \gamma^j$, there is
a first point $ \alpha^j$ on the curve with $ \alpha^j \in
\partial V \cap D^j$. Let $ \sigma^j$ be the subcurve of $
\gamma^j$ with end-points $ q^j$ and $ \alpha^j $. Then $ \sigma^j
$ is contained in an $ \epsilon $ - neighbourhood of $
\partial D^j$ for $ \epsilon > 0 $ small and for all $j$ large.
Using (\ref{5.5}) we get after integration,
\begin{eqnarray*}
\int_0^1 F^K_{D^j} \big( \gamma^j(t), \dot{\gamma}^j(t) \big) dt
\geq \int_0^1 F^K_{D^j} \big( \sigma^j(t), \dot{\sigma}^j(t) \big)
dt \geq C \int_0^1 \frac{ | \dot{\sigma}^j(t) |} {d(\sigma^j(t),
\partial D^j)^{1/2m}} \gtrsim \frac{1}{{\epsilon}^{1/2m}}
\end{eqnarray*}
Taking infimum over all admissible curve $ \gamma^j$ yields
\[
d_{D^j} (q^j, z^j) \gtrsim {\epsilon}^{- 1/2m}.
\]
which violates the fact that $ q^j \in B_{D^j} ( z^j, R )$ for $
\epsilon $ small enough. This completes the proof of the lemma.
\end{proof}

\medskip

\begin{lem} \label{K3}
\begin{equation*}
\displaystyle\lim_{j \rightarrow \infty} d_{D ^j} ( z^j, \cdot )=
d_{D_{\infty}} (z^0, \cdot ). \label{5.6}
\end{equation*}
Moreover, the convergence is uniform on compact sets of $
D_{\infty}$.
\end{lem}

\begin{proof} Let $ K $ be a compact subdomain of $D_{\infty}$ and suppose
that the desired convergence does not occur. Then there exists a $
\epsilon_0 > 0 $ and a sequence of points $ \{q^j \} \subset K $
which is relatively compact in $ D^j $ for all $j$ large such that
\[
\big| d_{D ^j} ( z^j, q^j ) - d_{D_{\infty}} (z^0, q^j ) \big|
> \epsilon _0.
\]
By passing to a subsequence, we may assume that $ q^j \rightarrow
q^0 \in K$ as $ j \rightarrow \infty $. Then using the continuity
of $ d_{D_{\infty}}(z^0, \cdot) $ we have
\[
\big| d_{D ^j} ( z^j, q^j ) - d_{D_{\infty}} (z^0, q^0 ) \big|
> \epsilon _0/2
\]
for all $j$ large. Fix $ \epsilon > 0$ arbitrarily small. It is
easy to see that
\begin{equation}
d_{D^j} ( z^0, q^0 ) \leq d_{D_{\infty}} ( z^0, q^0) + \epsilon/2
\label{k9}
\end{equation}
for all $j$ large. Let $ B(z^0, \delta_1) $ and $ B(q^0, \delta_2)
$ be sufficiently small neighbourhoods of $ z^0 $ and $ q^0$
respectively which are compactly contained in $ D^j $ for all
large $j$. It now follows that
\begin{eqnarray}
d_{D^j} (z^j, q^j) & \leq & d_{D^j} ( z^j, z^0) + d_{D^j} (z^0,
q^0) + d_{D^j} ( q^0,q^j) \nonumber \\
& \leq & d_{B(z^0, \delta_1)} ( z^j, z^0) + d_{D^j}
(z^0, q^0) + d_{ B(q^0, \delta_2)} ( q^0,q^j) \nonumber \\
& \leq & d_{D_{\infty}} (z^0, q^0) + \epsilon. \label{k10}
\end{eqnarray}
for all $j$ large. The second inequality uses the distance
decreasing property of the Kobayashi metric. The third inequality
follows from (\ref{k9}) and the following observation: since $ z^j
\rightarrow z^0 $ and the domains $ D^j$ converge to $ D_{\infty}
$, it follows that the ball $ B(z^0, \delta_1) $ contains $ z^j$
for large $j$ and is contained in $ D^j$ for all large $j$. Thus
\[
d_{D^j} ( z^j, z^0) \leq d_{B(z^0, \delta_1)} ( z^j, z^0) \lesssim
| z^j - z^0 |.
\]
The same argument works for showing that $ d_{D^j} ( q^0,q^j) $ is
small. Hence
\begin{equation}
d_{D^j} (z^j, q^j) \leq d_{D_{\infty}} (z^0, q^0 ) + \epsilon
\label{k13}
\end{equation}
for all $j$ large. For the converse, we intend to use lemma
\ref{D2}. First recall from (\ref{k8}) that for each $ \epsilon >
0$
\[
B_{ D_{\infty} } ( z^0, R - \epsilon ) \subset B_{D^j} ( z^j, R )
\]
for all $ R > 0$ and for all $j$ large. The Kobayashi completeness
of $ D_{\infty} $ implies that
\[
D_{\infty} = \displaystyle \bigcup _{\nu=1}^{\infty}
B_{D_{\infty}} (z^0,\nu),
\]
i.e., $ D_{\infty} $ can be exhausted by an increasing union of
relatively compact domains $B_{D_{\infty}}(z^0, \nu) $. As a
result, there exist uniform positive constants $ \nu^0 $ and $
\tilde{R} $ depending only on $K$ such that
\[
K \subset  B_{D_{\infty}} (z^0, \nu^0) \subset B_{D^j} ( z^j,
\tilde{R} )
\]
for all $j$ large. By lemma \ref{K2}
\begin{equation*}
d_{D_{\infty}}(z^j,q^j) \leq d_{B_{D^j} (z^j, R')} (z^j, q^j)
\end{equation*}
where  $ R' > 0 $ is chosen such that $ R' \gg 2 \tilde{R} $. Now,
apply lemma \ref{D2} to the domain $ D^j$. Let the Kobayashi
metric ball $ B_{D^j} (z^j, R')$ play the role of the subdomain $
D'$. Then
\begin{equation*}
d_{B_{D^j} (z^j, R')} (z^j,q^j) \leq \frac{d_{D^j} ( z^j, q^j)} {
\tanh \big( R'/2 - d_{D^j} (z^j, q^j) \big)}.
\end{equation*}
Since $ q^j \in B_{D^j} (z^j, \tilde{R} ) $ for all $j$ large and
$\tanh$ is increasing on $ [0,\infty)$, it follows that
\begin{equation*}
d_{D_{\infty}} (z^j,q^j) \leq \frac{d_{D^j} ( z^j, q^j)} { \tanh
\left( R'/2 - \tilde{R} \right) }
\end{equation*}
Letting $ R' \rightarrow \infty$ yields
\begin{equation*}
d_{D_{\infty}}(z^j,q^j) \leq \frac {d_{D^j} ( z^j, q^j)} { 1 -
\epsilon }
\end{equation*}
for all $j$ large. Again exploiting the continuity of $
d_{D_{\infty}}(\cdot, \cdot)$ and (\ref{k10}), we see that
\begin{equation}
d_{D_{\infty}}(z^0,q^0) \leq d_{D^j} ( z^j, q^j) + C \epsilon
\label{k11}
\end{equation}
for all $j$ large. Combining the estimates (\ref{k13}) and
(\ref{k11}), we get
\[
\displaystyle\lim_{j \rightarrow \infty} d_{D ^j} ( z^j, q^j )=
d_{D_{\infty}} (z^0, q^0 ).
\]
This is a contradiction and hence the result follows.
\end{proof}

\medskip

\noindent The following is an immediate corollary of the above
result.

\begin{cor} \label{K4}
\begin{equation*}
\displaystyle\lim_{j \rightarrow \infty} d_{D ^j} ( z^0, \cdot )=
d_{D_{\infty}} (z^0, \cdot ).
\end{equation*}
Moreover, the convergence is uniform on compact sets of
$D_{\infty}$.
\end{cor}

\begin{proof} For all $w$ in a fixed compact set $K$ of $ D_{\infty} $, we
have
\begin{eqnarray*}
\big |d_{D ^j} ( z^0, w ) - d_{D_{\infty}} (z^0, w ) \big | & \leq
& \big |d_{D ^j} ( z^0, w ) - d_{D ^j} ( z^j, w ) \big| + \big|
d_{D ^j} ( z^j, w ) - d_{D_{\infty}} (z^0, w ) \big|
\end{eqnarray*}
It follows from lemma \ref{K3} that
\[
\big| d_{D ^j} ( z^j, w ) - d_{D_{\infty}} (z^0, w ) \big|
\rightarrow 0
\]
uniformly for all $w \in K$ as $ j \rightarrow \infty$. Now, let $
B(z^0, \delta) $ be a sufficiently small neighbourhood of $ z^0 $.
Then the triangle inequality gives us that
\begin{eqnarray*}
\big |d_{D ^j} ( z^0, w ) - d_{D ^j} ( z^j, w ) \big|  \leq
d_{D^j} (z^0, z^j) \leq  d_{B(z^0, \delta)} (z^0, z^j) \rightarrow
0
\end{eqnarray*}
as $ j \rightarrow \infty$. This finishes the proof.
\end{proof}

\medskip

\noindent  We record the following consequence of corollary
\ref{K4}. The proof is exactly as that of lemma \ref{B2} and is
hence omitted.

\begin{lem} \label{E1} For all $ R > 0$, the sequence of domains
$ \big \{ B_{D^j} (z^0, R) \big \} $ converges in the Hausdorff
metric to $ B_{D_{\infty}}(z^0, R) $. Moreover, for any $ \epsilon
> 0 $ and for all $j$ large

\begin{itemize}

\item $B_{D_{\infty}}(z^0, R) \subset B_{D^j} (z^0, R +
\epsilon)$,

\medskip

\item $ B_{D^j} (z^0, R - \epsilon)  \subset
B_{D_{\infty}}(z^0,R). $

\end{itemize}

\end{lem}

\medskip

\noindent \textit{Proof of Theorem \ref{C0}(ii):} Since $h$ is
invariant under biholomorphisms,
\[
h_D(p^j) = h_{D^j} (z^0)
\]
for all $j$. Hence, it suffices to show that
\[
\displaystyle\lim_{j \rightarrow \infty} h_{D^j} (z^0) =
h_{D_{\infty}}(z^0).
\]
Fix $ \epsilon > 0$ arbitrarily small and let $ R > 0$ be such
that
\[
h_{D_{\infty}}(z^0) > \frac{1}{R} - \epsilon.
\]
There exists a biholomorphic imbedding $ F : \mathbf{B}^2
\rightarrow D_{\infty} $ such that $ F(0) = z^0 $ and $
B_{D_{\infty}}(z^0,R) \subset F(\mathbf{B}^2)$.

\noindent There exists a $ \delta > 0$ such that $ F \big(
(1-\delta) \mathbf{B}^2 \big) \supset B_{D_{\infty}}(z^0,R -
\epsilon)$. Since $ F \big( (1-\delta) \mathbf{B}^2 \big)$ is
compactly contained in $ D_{\infty}$, $ F \big( (1-\delta)
\mathbf{B}^2 \big) \subset D^j $ for all $j$ large. It follows
from lemma \ref{E1} that
\begin{equation*}
B_{D^j} (z^0, R - 2\epsilon)  \subset B_{D_{\infty}}(z^0, R -
\epsilon)
\end{equation*}
for all $j$ large. As a consequence,
\begin{equation*}
B_{D^j} (z^0, R - 2\epsilon)  \subset F \big( (1-\delta)
\mathbf{B}^2 \big) \subset D^j
\end{equation*}
or equivalently
\begin{equation*}
h_{D^j} (z^0) \leq \frac{1}{ R - 2 \epsilon}
\end{equation*}
for all $j$ large. This implies that
\begin{equation}
\displaystyle\limsup_{j \rightarrow \infty} h_{D^j} (z^0) \leq
h_{D_{\infty}}(z^0). \label{5.9}
\end{equation}
Conversely, there exist a sequence of biholomorphic imbeddings $
F^j: \mathbf{B}^2 \rightarrow D^j$ and positive numbers $ R_j $
such that $ F^j(0)= z^0, B_{D^j} (z^0, R_j) \subset F^j
(\mathbf{B}^2)$ and
\begin{equation}
h_{D^j} (z^0) \geq \frac{1}{R_j} - \epsilon. \label{6.0}
\end{equation}
The sequence $ \{ F^j \} $ admits a subsequence that converges
uniformly on compact sets of $ \mathbf{B}^2$ to a holomorphic
mapping $ F: \mathbf{B}^2 \rightarrow D_{\infty}$. Indeed,
consider the mappings
\[
\big(\Delta_{\zeta^j}^{\epsilon_j} \circ \phi^{\zeta^j} \big)^{-1}
\circ F^j : \mathbf{B}^2 \rightarrow D
\]
Observe that
\[
\big(\Delta_{\zeta^j}^{\epsilon_j} \circ \phi^{\zeta^j} \big)^{-1}
\circ F^j(0) = p^j \rightarrow p^0 \qquad \mbox{as} \ j
\rightarrow \infty
\]
and
\[
\big(\Delta_{\zeta^j}^{\epsilon_j} \circ \phi^{\zeta^j} \big)^{-1}
\circ F^j(0) \in Q(\zeta^j, C_1\epsilon_j)
\]
for some constant $ C_1 \geq 1$. For $ r \in (0,1)$ fixed, it
follows from \cite{Berteloot&Coeure-1991} that there exists a
positive constant $ C_2=C_2(r)$ such that
\begin{eqnarray*}
\big(\Delta_{\zeta^j}^{\epsilon_j} \circ \phi^{\zeta^j} \big)^{-1}
\circ F^j \big( \mathbf{B}^2(0,r) \big)  \subset Q( \zeta^j, C_1
C_2 \epsilon_j) \\
\end{eqnarray*}
which exactly means that
\begin{eqnarray*}
 F^j \big( \mathbf{B}^2(0,r) \big)  \subset
\Delta \left(0, \frac{\tau ( \zeta^j, C_1 C_2 \epsilon_j)}{\tau (
\zeta^j, \epsilon_j)} \right) \times \Delta \left( 0, \frac{C_1
C_2 \epsilon_j}{\epsilon_j} \right) \subset \Delta \big(0,
\sqrt{C_1 C_2}\big) \times \Delta(0, C_1 C_2).
\end{eqnarray*}
This shows that the sequence $ \{ F^j \}$ is a normal family.
Hence, $ \{ F^j \}$ admits a subsequence that converges uniformly
on compact sets of $ \mathbf{B}^2$ to a holomorphic mapping $ F:
\mathbf{B}^2 \rightarrow \mathbf{C}^2$ or $ F \equiv \infty$. The
latter is not possible since $ F(0) = \displaystyle \lim_{j
\rightarrow \infty} F^j(0) = z^0 $. As before, we can infer that $
F : \mathbf{B}^2 \rightarrow D_{\infty} $. Note that
\[
\frac{1}{R_j} \leq h_{D^j}(z^0) + \epsilon \leq
h_{D_{\infty}}(z^0) < \infty
\]
for all $j$ large. Hence, we may assume that the sequence $\{ R_j
\}$ converges to some $ R_0 > 0$. It follows that
\begin{equation}
B_{D^j} (z^0, R_0 - \epsilon) \subset B_{D^j} (z^0, R_j) \subset
F^j(\mathbf{B}^2) \label{6.1}
\end{equation}
for all $j$ large. Also, lemma \ref{E1} implies that
\begin{equation}
B_{D_{\infty}}( z^0, R_0 - 2 \epsilon) \subset B_{D^j} (z^0, R_0 -
\epsilon) \label{6.2}
\end{equation}
for all $j$ large. Combining (\ref{6.1}) and (\ref{6.2}), we get
\[
B_{D_{\infty}}( z^0, R_0 - 2 \epsilon) \subset F( \mathbf{B}^2).
\]
Therefore, $ F$ is non-constant. It follows from Hurwitz's theorem
that $F$ is injective on $ \mathbf{B}^2$ and hence
\[
h_{D_{\infty}}(z^0) \leq 1/(R_0 - 2 \epsilon)
\]
which together with (\ref{6.0}) implies that
\begin{equation}
h_{D_{\infty}}(z^0) \leq \displaystyle\liminf_{j \rightarrow
\infty} h_{D^j} (z^0) \label{6.3}
\end{equation}
It follows from (\ref{5.9}) and (\ref{6.3}) that
\begin{equation*}
h_{D^j} (z^0) \rightarrow h_{D_{\infty}}(z^0)
\end{equation*}
as $j \rightarrow \infty$. This completes the proof of theorem
\ref{C0}.
\end{proof}


\section{Behaviour of $h$ near convex finite type boundary
points}

\noindent The main result is as follows:

\begin{thm}\label{G0} Let $ D \subset \mathbf{C}^n $ be a smoothly bounded convex
domain of finite type. Let $ \{ q^j \}$ be a sequence of points in
$ D$ converging to $ q^0 \in
\partial D$. Then
\[
h_{D} (q^j) \rightarrow h_{D_{\infty}}(('0,-1))
\]
as $ j \rightarrow \infty $ where $ D_{\infty}$ is a biholomorph
of the limiting domain $ D_0 $ obtained by scaling $D$ with
respect to the sequence $ \{ q^j \}$.

\end{thm}

\noindent As in section 5, there are two cases to be considered,
i.e., after passing to a subsequence if needed,

\begin{enumerate}

\item[(i)] $ \displaystyle\lim_{ j \rightarrow \infty} h_{D} (p^j)
= 0 $, or

\item[(ii)] $ \displaystyle\lim_{ j \rightarrow \infty} h_{D}
(p^j) > c $ for some positive constant $c$.

\end{enumerate}

\noindent Here too the domain $ D_{\infty} $ will be
biholomorphic to $ \mathbf{B}^n$ in case (i) while this will not be the case in
(ii).

\medskip

\noindent In order to be able to prove the above result, we need
to introduce the following special coordinates constructed for
convex finite type domains in \cite{Mcneal-1992} (see
\cite{Mcneal-1994} also).

\medskip

\noindent Let $ D \subset \mathbf{C}^n$ be as stated above. Assume
that $ q^0 = 0 $ without loss of generality. Let $ D = \{ \rho (z,
\bar{z}) < 0\}$ where $ \rho $ is a smooth defining function for $
\partial D$ which has the form
\[
\rho(z, \bar{z} ) = \Re z_n + \psi ('z, \Im z_n)
\]
near the origin with $ \psi $ a smooth convex function (i.e., the
real Hessian of $\psi$ is positive semi-definite). We may assume
that $ \rho $ has the property that all the sets $ \{ z : \rho(z)
< \eta \} $ are convex for some $ \eta $ in some range $ - \eta_0
< \eta < \eta_0, \eta_0
> 0 $. For $ q \in D$ sufficiently close to $ \partial D $, let
\[
D_{q,\epsilon} = \big\{ z : \rho(z) < \rho(q) + \epsilon \big \}.
\]
Working with sufficiently small $ \epsilon > 0 $, there is a
unique point $ p^n_{q,\epsilon} \in \partial D_{q, \epsilon} $
where the distance of $q$ to $ \partial D_{q, \epsilon} $ is
achieved. Denote the complex line containing $q$ and $
p^n_{q,\epsilon} $ by $ L_n $ and let $ \tau_n(q, \epsilon) =
\big| q - p^n_{q,\epsilon} \big| $. Consider $ ( L_n)^{\bot} $ the
orthogonal complement of the complex line $ L_n $ in $
\mathbf{C}^n$. Since $ \partial D$ is of finite type, the distance
from $q$ to $ \partial D_{q, \epsilon} $ along each complex line
in $ ( L_n)^{\bot} $ is uniformly bounded. Let $ \tau_{n-1} (q,
\epsilon) $ be the largest such distance and $
p^{n-1}_{q,\epsilon} \in \partial D_{q, \epsilon} $ be any point
such that $ \big | q - p^{n-1}_{q,\epsilon} \big| = \tau_{n-1}(q,
\epsilon)$. Denote the complex line containing $q$ and $
p^{n-1}_{q,\epsilon} $ by $ L_{n-1} $. Now consider the orthogonal
complement of the $ \mathbf{C}$-subspace spanned by $ L_n $ and $
L_{n-1} $ and find the largest distance from $ q $ to $ \partial
D_{q, \epsilon} $ therein. There exists $  p^{n-2}_{q,\epsilon}
\in \partial D_{q, \epsilon} $ where this distance is achieved.
Let $ \tau_{n-2} (q, \epsilon) = \big | q -  p^{n-2}_{q,\epsilon}
\big| $ and $ L_{n-2} $ denote the complex line containing $q$ and
$  p^{n-2}_{q,\epsilon} $. Repeating this process, we get
orthogonal lines $ L_n, L_{n-1}, \ldots, L_1 $. Let $ T^{q, \epsilon}
$ be the translation sending $ q $ to the origin and $ U^{q,
\epsilon} $ be a unitary mapping of $ \mathbf{C}^n $ sending $
L_i$ to the $ z_i $-axis and $  p^{i}_{q,\epsilon} - q $ to a
point on the $ \Re z_i $-axis. Note that
\begin{eqnarray*}
U^{q, \epsilon} \circ T^{q, \epsilon} (q) & = & 0 \\
\mbox{and} \qquad U^{q, \epsilon} \circ T^{q, \epsilon}
\big(p^i_{q, \epsilon} \big) & = & \big ( 0, \ldots, \tau_i(q, \epsilon), \ldots, 0)
\end{eqnarray*}
for all $ 1 \leq i \leq n$. Denote the new coordinates by
\[
\big( z_1^{q, \epsilon}, \ldots, z_n^{q, \epsilon} \big) = U^{q,
\epsilon} \circ T^{q, \epsilon} (z_1, \ldots, z_n )
\]
and the corresponding defining function is given by
\[
\rho^{q, \epsilon} = \rho \circ \big( U^{q, \epsilon} \circ T^{q,
\epsilon} \big)^{-1}.
\]
Moreover, McNeal defined the polydiscs $ P(q, \epsilon) $ in the
new coordinates $ ( z_1^{q, \epsilon}, \ldots, z_n^{q, \epsilon} ) $
centered at $ q$ as
\begin{equation} \label{t3}
P(q, \epsilon) = \big \{ ( z_1^{q, \epsilon}, \ldots, z_n^{q,
\epsilon} ): | z_1^{q, \epsilon} | < \tau_1(q, \epsilon), \ldots, |z_n^{q, \epsilon} | < \tau_n(q, \epsilon) \big\}.
\end{equation}

\noindent Although we may not write explicitly all the time, the
reader must be aware of the dependence of all the coordinates,
points and numbers on $q$ and $ \epsilon$. Having recalled certain
basic facts about the local geometry of convex domains of finite
type, we briefly describe the scaling of the domain.

\medskip

\noindent \textbf{Scaling the domain $ D $: }

\medskip

\noindent Let $ \{ q^j \} $ be a sequence of points in $D$
accumulating at $ q^0 = ('0,0) \in \partial D $. Set $ \epsilon_j
= - \rho (q^j) $. The positive numbers $ \tau_1(q^j, \epsilon_j
), \ldots, \tau_n(q^j, \epsilon_j ) $ and $ p^{1,j}, \ldots, p^{n,j} $ are
those associated with $ q^j $ and $ \epsilon_j $. Define the
dilations
\[
\Lambda^{\epsilon_j}_{q^j} (z) = \big( \tau_1(q^j, \epsilon_j )
z_1, \ldots, \tau_n(q^j, \epsilon_j ) z_n \big )
\]
and the dilated domains $ D^j = \left( \Lambda^{\epsilon_j}_{q^j}
\right)^{-1} \circ U ^{q^j, \epsilon_j} \circ T^{q^j, \epsilon_j}
(D) $ are defined by
\[
\rho^j(z) = \rho \circ \big( U^{q^j, \epsilon_j} \circ T^{q^j,
\epsilon_j} \big)^{-1} \circ \Lambda^{\epsilon_j}_{q^j} (z)
\]
Note that $ D^j $ is convex and $ ('0,0) \in D^j $ for all $j$.
Among other things, the following two claims were proved in
\cite{Gaussier-1997}. First, that $ D^j$ converges to
\[
D_0 = \big \{ ('z, z_n) \in \mathbf{C}^n : \tilde{\rho}(z) = -1 +
\Re \big ( \displaystyle\sum_{k=1}^n b_k z_k  \big) + P('z) < 0
\big \}
\]
where $ b_k $ are complex numbers and $ P$ is a real convex
polynomial of degree less than or equal to $2m$. Secondly, for all
large $j$, $ D^j$ and hence $ D_0 $ are contained in the
intersection of half spaces $ H_1 \cap H_2 \cap \ldots \cap H_n $,
where
\[
H_n = \big \{ z \in \mathbf{C}^n : \Re \big( ( z_n - 1 ) \frac{
\partial \tilde{\rho}}{ \partial z_n} (e^n) \big) \leq 0 \big \}
\]
and for $ k < n$
\[
H_k = \big \{ z \in \mathbf{C}^n : \Re \big ( ( z_k -1 ) \frac{
\partial \tilde{\rho}}{ \partial z_k} (e^k) +
\displaystyle\sum_{i= k+1} ^n \frac{ \partial \tilde{\rho}}{
\partial z_i} (e^k) z_i \big) \leq 0 \big \}
\]
where
\[
e^i = \big( \Lambda^{\epsilon_j}_{q^j} \big)^{-1} (p^{i,j}) \qquad
\mbox{and} \qquad  \frac{ \partial \tilde{\rho}}{ \partial z_i}
(e^i) \in \mathbf{R} \setminus \{ 0 \} \ \mbox{for all} \ 1 \leq i
\leq n.
\]

\medskip

\noindent As a consequence, there exists a rational biholomorphism
from $ D_0 $ to a bounded domain contained in the polydisc (the
Cayley transformation in each variable). In particular, $ D_0$ is
hyperbolic. Therefore, $ D_0 $ is Brody hyperbolic, i.e., $ D_0 $
contains no nontrivial complex affine line. Then there is no
complex line in $ \partial D_0 $ and according to theorem 1.1 of
\cite{Mcneal-1992}, $ D_0 $ is of finite type and $P$ is
nondegenerate. It follows that $ D_0 $ is complete hyperbolic.

\medskip

\noindent Moreover, it follows from \cite{Mcneal-1992} that the
constant $ b_n$ is different from zero. Then by a
$\mathbf{C}$-affine change of coordinates, $ D_0 $ is equivalent
to the convex domain $ D_{\infty} = \big \{ ('w, w_n) \in
\mathbf{C}^n : 2 \Re w_n + P('w) < 0 \big \}$. As a result,
\[
h_{D_{\infty}} \big( ('0,-1) \big) = h_{D_0} \big( ('0,0) \big).
\]

\medskip

\noindent In the particular case when the sequence $ \{q^j \} $
converges normally to the point $ q^0$. We denote the multitype of
$\partial D$ at the origin by $ \mathcal{M}(\partial D, 0) =
(m_1, \ldots, m_{n-1},1)$; the points $ q^j$ will be $ ('0,
-\epsilon_n) $ in the coordinates defined by Yu in \cite{Yu-1992}.
Thus we may assume that for all $j$, $ \tau(q^j, \epsilon_j) =
\epsilon_j $ and the function $ \rho $ is defined in a fixed
neighbourhood of the origin by
\[
\rho('z, z_n) = 2 \Re z_n + P_0('z) + R(z)
\]
where $ P_0 $ is a nondegenerate weighted homogeneous polynomial
of degree $1$ with respect to the weights $ \mathcal{M}(\partial
D, 0)$ and $ R $ denotes terms of degree at least two. In this
case, it turns out that $ P('z) = P_0 ('z) $ so that the limiting
domain $ D_{\infty} $ is a biholomorph of the domain
\[
\left \lbrace ('z, z_n) \in \mathbf{C}^n : 2 \Re z_n + P_0('z) < 0
\right \rbrace.
\]
\noindent Now, arguments similar to those in lemma \ref{K1} using
the fact that the domains $ D^j$ are all contained in $  H_1 \cap
H_2 \cap ... \cap H_n $ for large $j$, shows that

\begin{lem} \label{Z1} For $(a,v) \in D_{0} \times
\mathbf{C}^n$,
\begin{equation*}
\displaystyle\lim_{j \rightarrow \infty} F^K_{D^j} (a,v) =
F^K_{D_0}(a,v).
\end{equation*}
Moreover, the convergence is uniform on compact sets of $ D_0
\times \mathbf{C}^n$.
\end{lem}

\noindent \textbf{Stability of the Kobayashi distance:}

\medskip

\noindent The goal now will be to recover the behaviour of the
distance function related to the Kobayashi metric. To do this, we
use ideas from \cite{Seshadri&Verma-2009}. A key ingredient will
be Lempert's theorem (see \cite{Lempert-1981}) on the existence of
complex geodesics in strictly convex domains. Write $ ('0,0) = z^0
$ for brevity.

\begin{lem} \label{G3}
\begin{equation*}
\displaystyle\lim_{j \rightarrow \infty} d_{D ^j} ( z^0, \cdot )=
d_{D_0} (z^0, \cdot ). \label{g4}
\end{equation*}
Moreover, the convergence is uniform on compact sets of $ D_0$.
\end{lem}

\begin{proof} Let $ K \subset D_0 $ be compact and suppose that the
desired convergence does not occur. Then there exists a $ \epsilon
_0 > 0 $ and a sequence of points $ \{p^j \} \subset K $ which is
relatively compact in $D^j $ for $j$ large such that
\[
\big| d_{D^j} ( z^0, p^j ) - d_{D_0} (z^0, p^j ) \big| > \epsilon
_0
\]
for all $j$ large. By passing to a subsequence, we may assume that
$ p^j \rightarrow p^0 \in K $ as $ j \rightarrow \infty$. Since $
d_{D_0} ( z^0, \cdot ) $ is continuous, it follows that
\begin{equation}
\big| d_{D ^j} ( z^0, p^j ) - d_{D_0} (z^0, p^0 ) \big| > \epsilon
_0/2 \label{g5}
\end{equation}
for all $j$ large. Fix $ \epsilon > 0$ and let $ \gamma : [0,1]
\rightarrow D_0 $ be a path such that $ \gamma(0) = z^0, \gamma(1)
= p^0$ and
\[
\int_0 ^1 { F^K _{D_0} \big( \gamma(t), \dot{\gamma}(t) \big)} dt
< d_{D_0} (z^0, p^0 ) + \epsilon/2.
\]
Define $ \gamma ^j : [0,1] \rightarrow \mathbf{C}^n $ by
\[
\gamma^j(t) = \gamma(t) + ( p^j - p^0 ) t.
\]
Since the image $ \gamma([0,1])$ is compactly contained in $D_0$
and $ p^j \rightarrow p^0 \in K $ as $ j \rightarrow \infty$, it
follows that $ \gamma ^j : [0,1] \rightarrow D^j $ for $j$ large.
In addition, $ \gamma^j (0) = \gamma (0) = z^0 $ and $ \gamma^j
(1) = \gamma(1) + p^j - p^0 = p^j $.

\medskip

\noindent  Note that $ \gamma^j \rightarrow \gamma $ and $
\dot{\gamma}^j \rightarrow \dot{\gamma} $ uniformly on $ [0,1]$.
Further, it is already known from lemma \ref{Z1} that $ F^K_{D^j}
( \cdot, \cdot) \rightarrow F^K_{D_0} (\cdot, \cdot) $ uniformly
on compact sets of $ D_0 \times \mathbf{C}^n$ Therefore for $j$
large, we obtain
\[
\int_0 ^1 { F^K_{D^j} \big( \gamma^j(t), \dot{\gamma}^j(t)  \big)}
dt \leq \int_0 ^1 { F^K _{D_0} \big( \gamma(t), \dot{\gamma}(t)
\big)} dt + \epsilon/2 < d_{D_0} (z^0, p^0 ) + \epsilon.
\]
By definition of $ d_{D ^j} ( z^0, p^j )$ it follows that
\[
d_{D ^j} ( z^0, p^j ) \leq \int_0 ^1 { F^K_{D^j} \big(
\gamma^j(t), \dot{\gamma}^j(t) \big)} dt \leq d_{D_0} (z^0, p^0 )
+ \epsilon.
\]
Thus
\begin{eqnarray}
\displaystyle\limsup_{j \rightarrow \infty} d_{D ^j} ( z^0, p^j )
\leq d_{D_0} (z^0, p^0 ). \label{g6}
\end{eqnarray}

\medskip

\noindent Conversely, it follows by Lempert's work that there
exist $ m_j > 1$ and holomorphic mappings
\[
\phi^j : \Delta(0, m_j) \rightarrow D^j
\]
such that $ \phi^j(0) = z^0, \phi^j (1) = p^j $ and
\begin{eqnarray*}
d_{\Delta(0, m_j)} (0,1) & = & d_{D^j} \big(z^0, p^j \big) =
\int_0 ^1 { F^K _{D^j } \big( \phi^j(t), \dot{\phi}^j (t) \big) }
dt.
\end{eqnarray*}
Therefore,
\begin{eqnarray*}
\frac{1}{2} \log \left( \frac{m_j + 1} { m_j - 1} \right) =
d_{\Delta(0, m_j)} (0,1)  =   d_{D^j} \big( z^0, p^j \big)
\end{eqnarray*}
However from (\ref{g6}) we have that
\[
d_{D^j}( z^0, p^j) \leq d_{D_0} (z^0, p^0) + \epsilon < \infty
\]
and hence $ m_j > 1 + \delta $ for some uniform $ \delta > 0 $ for
all $ j $ large. Moreover, the domains $D^j$ are all contained in
the intersection of the half planes. Hence the family $ \phi^j
|_{\Delta (0, 1 + \delta) }: \Delta (0, 1+ \delta) \rightarrow D^j
$ is normal. Consequently, the limit map $ \phi : \Delta(0, 1 +
\delta) \rightarrow \overline{D}_{0} $. Note that by construction
$ \phi(0) = z^0 $ and $ \phi(1) = p^0 $. Using the maximum
principle, we conclude that $ \phi : \Delta(0, 1 + \delta)
\rightarrow D_0 $.

\medskip

\noindent Since $ \phi^j \rightarrow \phi $ and $ \dot{\phi}^j
\rightarrow \dot{\phi} $ uniformly on $ [0,1]$, again exploiting
the uniform convergence of $ F^K_{D^j}( \cdot, \cdot) \rightarrow
F^K_{D_0} (\cdot, \cdot)$ on compact sets of $ D_0 \times
\mathbf{C}^n$, we see that
\begin{eqnarray*}
\int_0 ^1 {F^K_{D_0} \big(\phi(t),  \dot{\phi}(t)\big) dt} \leq
\int_0 ^1{F^K_{D^j} \big(\phi^j(t), \dot{\phi}^j(t) \big) dt } +
\epsilon = d_{D^j}( z^0, p^j) + \epsilon
\end{eqnarray*}
for all $j$ large. Finally, observe that $ \phi |_{[0,1]}$ is a
differentiable path in $D_0$ joining $ z^0$ and $ p^0 $. Hence by
definition
\begin{eqnarray}
d_{D_0} (z^0, p^0 ) \leq \int_0 ^1 {F^K_{D_0} \big(\phi(t),
\dot{\phi}(t) \big) dt} \leq d_{D^j}( z^0, p^j) + \epsilon
\label{g7}
\end{eqnarray}
Combining (\ref{g6}) and (\ref{g7}) shows that
\[
\displaystyle\lim_{j \rightarrow \infty} d_{D ^j} ( z^0, p^j ) =
d_{D_0} (z^0, p^0 )
\]
which contradicts the assumption (\ref{g5}) and proves the
required result.
\end{proof}

\begin{lem} \label{G4} For all $ R > 0$, the sequence of domains
$ \big \{ B_{D^j} (z^0, R) \big \} $ converges in the Hausdorff
metric to $ B_{D_0}(z^0, R) $. Moreover, for any $ \epsilon
> 0 $ and for all $j$ large

\begin{itemize}

\item $B_{D_0}(z^0, R) \subset B_{D^j} (z^0, R + \epsilon)$,

\medskip

\item $ B_{D^j} (z^0, R - \epsilon)  \subset B_{D_0}(z^0,R). $

\end{itemize}

\end{lem}

\begin{proof} This follows by making the relevant changes in the proof of
lemma \ref{B2}.
\end{proof}

\noindent \textit{Proof of theorem \ref{G0}:} Using lemma \ref{G3}
and \ref{G4}, an argument similar to the one employed in theorem
\ref{C0} (ii) can be used to complete the proof of theorem
\ref{G0}. \qed


\section{Behaviour of $h$ near the corners of a generic analytic
polyhedron}

\noindent To investigate the behaviour of $ h_P( \cdot, \Delta^n)$
near a singular boundary point of an analytic poyhedron, we use
the rescaling technique given in \cite{Kim&Pagano-2001}.

\begin{thm} \label{B4} Let $ D \subset \mathbf{C}^n$ be a bounded domain.
Let $ z^0 \in \partial D$ have a neighbourhood $U$ such that $ U
\cap \partial D  = \{ z \in \mathbf{C}^n : |f^i(z) | < 1, \
\mbox{for all} \ 1 \leq i \leq n \} $ and $ |f^i(z^0) |=1 \
\mbox{for all} \ 1 \leq i \leq n $ where $ f^i \in \mathcal{O}(U)
\ \mbox{for all} \ 1 \leq i \leq n $. Assume that $ df^1 \wedge
\ldots \wedge df^n \neq 0 $ at $ z^0$. Then $ h_D( z, \Delta^n)
\rightarrow  0 $ as $  {z \rightarrow z^0}$.

\end{thm}

\begin{proof} Let $ \{ z^k \} \subset D $ be a sequence converging
to $z^0$. It suffices to show that for any sequence $ h_{D}( z^k )
\rightarrow 0 $ as $ k \rightarrow \infty$. An application of the
implicit function theorem gives an open neighbourhood $U$ of $z^0$
in $ \mathbf{C}^n$ and real numbers $ \theta_i, i= 1, \ldots, n $
satisfying the following properties:

\begin{itemize}

\item $ F := ( e^{\iota \theta_1} f^1, \ldots, e^{\iota \theta_n}
f^n) : U \rightarrow \mathbf{C}^n $ is a biholomorphic imbedding,

\medskip

\item $ F(z^0) = (1, \ldots, 1) \in \mathbf{C}^n, $

\medskip

\item $ F( U \cap D ) = \Delta^n \cap F(U). $

\end{itemize}

\noindent Now, consider $ \phi : \Delta^n \rightarrow \mathbf{H}^n
= \mathbf{H} \times  \ldots \times \mathbf{H} $ defined by
\begin{eqnarray*}
\phi(z) & = & \left( \phi_1(z), \ldots, \phi_n(z) \right) \\
& = & \left( \iota \frac{ 1 - z_1} { 1 + z_1 }, \ldots, \iota \frac{
1 - z_n} { 1 + z_n } \right)
\end{eqnarray*}
where $ \mathbf{H} = \{ w \in \mathbf{C} : \Im w > 0 \}$. Then $
\phi \circ F(z^0) = (0, \ldots, 0) $ and $ \phi \circ F $ maps $U$
biholomorphically onto $ \mathbf{H}^n \cap \phi \circ F (U) $. For
$ l = 1, \ldots, n$ let
\begin{eqnarray*}
\tau_k^{(l)} =  \Re \ \phi_l \circ F ( z^k), \qquad
\lambda_k^{(l)} =  \Im \ \phi_l \circ F ( z^k).
\end{eqnarray*}
Then
\[
A^k(z) := \left( \frac{ z_1 - \tau_k^{(1)}} {
\lambda_k^{(1)}}, \ldots, \frac{ z_n - \tau_k^{(n)}} {
\lambda_k^{(n)}} \right)
\]
is an automorphism of $\mathbf{H}^n$. Consider
\[
\Lambda^k := \phi^{-1} \circ A^k \circ \phi \circ F : U \cap D
\rightarrow \Delta^n.
\]
Then $ \Lambda^k$ are biholomorphic imbeddings of $ U \cap D $
into $ \Delta^n$. Observe that
\[
\Lambda^k(z^k) = \phi^{-1} \circ A^k ( \tau_k^{(1)} + \iota
\lambda_k^{(1)}, \ldots, \tau_k^{(n)} + \iota \lambda_k^{(n)}) =
(0, \ldots, 0).
\]
Additionally, $ \{ \Lambda^k (U \cap D) \}$ is a sequence of
subdomains of $ \Delta^n$ that exhausts $ \Delta^n$. Indeed, let $
L $ be a compact subdomain of $ \Delta^n$. Exploiting the
continuity of the mapping $ \phi \circ F$, we see that $ \phi
\circ F (z^k) \rightarrow 0$ as $ k \rightarrow \infty$. In other
words,
\[
( \tau_k^{(1)} + \iota \lambda_k^{(1)}, \ldots, \tau_k^{(n)} + \iota
\lambda_k^{(n)}) \rightarrow (0, \ldots, 0)
\]
as $ k \rightarrow \infty $. This gives that for each $
l=1, 2, \ldots, n $, both $ \tau_k^{(l)} \rightarrow 0 $ and $
\lambda_k^{(l)} \rightarrow 0 $ as $ k \rightarrow \infty$. As a
result, $ (A^k)^{-1}$ maps $ \phi(L) $ onto a neighbourhood $ W
\subset \phi \circ F ( U \cap D) $ of $ (0, \ldots, 0) $ in $
\mathbf{H}^n$ for sufficiently large $k$. This is just the
assertion that $ L$ is compactly contained in $ \Lambda^k( U \cap
D) $ for all $k$ large. This finishes the proof of the claim. Now,
fix $ R
> 0$ and let $ \theta : \Delta^n \rightarrow B_{\Delta^n} ( 0, R)
$ be a biholomorphism. For $ \epsilon > 0 $ given, there exists a
$ \delta \in (0,1) $ such that
\begin{eqnarray}
B_{\Delta^n} ( 0, R - \epsilon ) \subset \theta(
\Delta^n(0,\delta) ) \subset B_{\Delta^n} ( 0, R). \label{3.4}
\end{eqnarray}
Using lemma 2.3 in \cite{Fridman-1983}, we get that
\[
B_{\Delta^n} ( 0, R) \subset B_{\Lambda^k( U \cap D)}(0, R +
\epsilon)
\]
for all $k$ large. Consequently,
\begin{eqnarray}
\theta( \Delta^n(0,\delta) ) \subset B_{\Lambda^k (U \cap D)} (0,
R + \epsilon) \subset \Lambda^k(U \cap D)  \label{3.5}
\end{eqnarray}
for all $k$ large. It is already known that $ \Lambda^k(U \cap D)
\subset \Delta^n$, therefore
\begin{eqnarray}
B_{\Lambda^k (U \cap D)} (0, r)  \subset B_{\Delta^n}(0, r)
\label{3.6}
\end{eqnarray}
for all $r > 0$ and all $k$. From (\ref{3.4}), (\ref{3.5}) and
(\ref{3.6}), we conclude that there exists a biholomorphic
imbedding $ \theta : \Delta^n(0, \delta) \rightarrow \Lambda^k (U
\cap D) $ such that
\[
B_{\Lambda^k (U \cap D)} (0, R - \epsilon) \subset \theta(
\Delta^n(0,\delta) )
\]
so that
\[
h_{\Lambda^k (U \cap D)}(0, \Delta ^n) \leq (R - \epsilon)^{-1}.
\]
Since $ R > 0 $ was arbitrary, we have $ \displaystyle\lim _{k
\rightarrow \infty } h_{ \Lambda^k( U \cap D) } ( 0, \Delta^n ) =
0 $. In addition,
\[
h_{ \Lambda^k (U \cap D) } (0, \Delta^n) = h_{ \Lambda^k (U \cap
D) } (\Lambda^k(z^k), \Delta^n) = h_{ U \cap D } (z^k, \Delta^n).
\]
It follows that
\[
h_{ U \cap D } ( z^k, \Delta^n )\rightarrow 0
\]
as $ k \rightarrow \infty $. Since $h$ can be localised near $z^0$
by proposition \ref{A3}, we have
\[
h_{ D } ( z^k, \Delta^n ) \rightarrow 0
\]
as $ k \rightarrow \infty $.
\end{proof}


\section{Applications}

\noindent In this section we collect several consequences of the
results presented in the prequel. To begin with, we present an
alternate proof of the following theorem of Wong-Rosay
(\cite{Wong-1977}, \cite{Rosay-1979}) using theorem \ref{B0} and
characterise $ \Delta^n$ in the class of generic analytic
polyhedra. It must be mentioned that this is motivated by the main
theorem of Kim and Pagano from \cite{Kim&Pagano-2001} and in fact
provides an alternate proof of their result in case the orbit
accumulates at a singular boundary point. The question of
recovering their full theorem using Fridman's invariant seems
interesting.

\begin{thm}
Let $ D \subset \mathbf{C}^n $ be a bounded domain. Let $ z^0 \in
\partial D $ have a neighbourhood $U$ such that $ U \cap \partial
D $ is $C^2$-smooth strongly pseudoconvex . If $ z^0$ is an
orbit accumulation point for the action of $ \mbox{Aut} (D)$ on $
D$, then $D$ is biholomorphically equivalent to $\mathbf{B}^n$.
\end{thm}

\begin{proof} By hypothesis, there exists a sequence $ \{ F^k \} \subset
\mbox{Aut}(D)$ and $ p^0 \in D$ such that $ F^k(p^0) \rightarrow
z^0 \in \partial D$ as $ k \rightarrow \infty$. Then using theorem
\ref{B0}, we have
\[
h_{ D } \big( F^k(p^0) \big) \rightarrow 0
\]
as $ k \rightarrow \infty $. Since
\[
h_{ D } \big( F^k(p^0) \big) = h_{F^k( D) } \big( F^k(p^0) \big) =
h_D(p^0)
\]
for all $k$, it follows that $ h_D(p^0)= 0$. Applying proposition
\ref{A0}, we conclude that $ D$ is biholomorphic to $
\mathbf{B}^n$.
\end{proof}

\begin{thm}
Let $ D \subset \mathbf{C}^n $ be a bounded domain. Let $ z^0 \in
\partial D $ have a neighbourhood $U$ such that $ U \cap \partial D
= \big \{ z \in \mathbf{C}^n | \ |f^i(z)| < 1 \ \mbox{for all} \ 1
\leq i \leq n \big \}$ and $ \big |f^i(z^0) \big| = 1 \ \mbox{for
all} \ 1 \leq i \leq n  $ where $ f^i \in \mathcal{O}(U) \
\mbox{for all} \ 1 \leq i \leq n $. Assume that $ d f^1 \wedge \ldots
\wedge df^n \neq 0 $ at $ z^0$. If $ z^0$ is an orbit accumulation
point for the action of $ \mbox{Aut} (D)$ on $ D$, then $D$ is
biholomorphically equivalent to $\Delta^n$.
\end{thm}

\begin{proof} By hypothesis, there exists a sequence $ \{ F^k \} \subset
\mbox{Aut}(D)$ and $ p^0 \in D$ such that $ F^k(p^0) \rightarrow
z^0 \in \partial D$ as $ k \rightarrow \infty$. It follows from
theorem \ref{B4} that
\[
h_{ D } \big( F^k(p^0), \Delta^n \big) \rightarrow 0
\]
as $ k \rightarrow \infty $. Since
\[
h_{ D } \big( F^k(p^0), \Delta^n \big) = h_{F^k( D) } \big(
F^k(p^0), \Delta^n \big) = h_D \big(p^0, \Delta^n \big)
\]
for all $k$, we infer that $ h_D \big(p^0, \Delta^n \big)= 0$.
Applying proposition \ref{A0}, we conclude that $ D$ is
biholomorphic to $ \Delta^n$.
\end{proof}

\noindent As a consequence of theorem \ref{C0}, it is also
possible to recover parts of the main theorem of
\cite{CoupetPinchuk&Sukhov-1996}, the emphasis here being a
different approach - wherein there is no need apriori to establish
the boundary distance estimate.

\begin{thm} \label{T1} A strongly pseudoconvex domain with $
{C}^2$-smooth boundary cannot be mapped biholomorphically onto a
smoothly bounded weakly pseudoconvex domain of finite type in $\mathbf{C}^2$.
\end{thm}

\begin{proof} Suppose there exists a biholomorphism $f$ from a strongly
pseudoconvex domain $ D_1$ onto a weakly pseudoconvex domain $
D_2$ of finite type in $ \mathbf{C}^2$. Let $ q^0 \in \partial
D_2$ be a weakly pseudoconvex boundary point. Then there exists a
local coordinate system in a neighbourhood of $ q^0 $ taking $q^0$
to the origin such that the domain $ D_2$ near origin can be
written as
\[
\big \{ (z_1,z_2) \in \mathbf{C}^2 : 2 \Re z_2 +
H_{2m}(z_1,\bar{z}_1 ) + o ( |z| ^{2m} + \Im z_2 )  < 0 \big \}
\]
where $H_{2m}$ is a homogeneous subharmonic polynomial of degree $
2m \geq 2 $ in $ z_1$ and $\bar{z}_1 $ which does not contain any
harmonic terms. Let $ \{ q^j \} $ be a sequence of points in $
D_2$ converging to $ q^0 $ normally, i.e.,
\[
q^j = q^0 - d(q^j, \partial D_2) n(q^0)
\]
where $ n(q^0) $ denotes the unit outward normal to $ \partial D_2
$ at $ q^0 $. Choose points $ \{ p^j \} \subset D_1$ such that $
f(p^j) = q^j$. Then for any $ x \in D_1$, we have
\[
d_{D_1}( x, p^j) = d_{D_2} \big( f(x), q^j \big).
\]
The above observation together with the completeness of $D_1 $ and
$D_2$ implies that the sequence $ \{ p^j \} $ must cluster on
$\partial D_1 $. Hence, we may assume that the sequence $ \{ p^j
\}$ converges to $ p^0 \in \partial D_1 $. It follows from the
biholomorphic invariance of the function $h$ that
\[
h_{D_2} ( q^j) = h_{D_1} (p^j).
\]
Applying theorem \ref{B0} and theorem \ref{C0}, we get after
letting $ j \rightarrow \infty$ that
\begin{equation}
h_{D_{2,\infty}} ((0,-1)) = 0. \label{g30}
\end{equation}
Here $ D_{2,\infty} $ is the limiting domain obtained by scaling
$D_2$ with respect to the base point $ q^0 $ and the sequence $ \{
q^j \}$. More concretely,
\[
D_{2,\infty} = \big \{ (z_1,z_2) \in \mathbf{C}^2 : 2 \Re z_2 +
H_{2m}(z_1,\bar{z}_1 ) < 0 \big \}.
\]
Now, using proposition \ref{A0}, we infer from (\ref{g30}) that $
D_{2,\infty}$ is biholomorphically equivalent to $ \mathbf{B}^2$.
Let $ \tilde{f}$ denote the biholomorphism from the domain $
D_{2,\infty}$ to $ \mathbf{B}^2$. Then for any point $ (0, \iota
a^0), a^0 \in \mathbf{R} $ and any sequence of points $ \{ z^j \}
\subset D_{2, \infty}$ converging to $ (0, \iota a) \in \partial
D_{2, \infty}$, the corresponding image sequence $ \{
\tilde{f}(z^j) \} $ clusters at some point of $ \partial
\mathbf{B}^2 $. Composing with an appropriate Cayley transform, we
may assume that there exists a biholomorphism $ F$ from $
D_{2,\infty}$ onto the unbounded realization of the unit ball,
namely to
\[
\Sigma = \big \{ z \in \mathbf{C}^2 : 2 \Re z_2 + | z_1 |^2 < 0
\big \}
\]
with the property that the cluster set of $F$ at the point $ (0,
\iota a^0) \in \partial D_{2, \infty}$ contains a boundary point
of $ \Sigma $ different from the point at infinity on $
\partial \Sigma $. Then it follows from theorem 2.1 of
\cite{Coupet&Pinchuk-2001} that $F$ extends holomorphically to a
neighbourhood of $ (0, \iota a^0)$. Since $ D_{2, \infty} $ is
invariant under translations in the imaginary $ z_2$- direction,
we can assume that $(0, \iota a^0) = (0,0) $. This shows that $F$
extends holomorphically to a neighbourhood of the origin. In
addition, by composing with an automorphism of $ \Sigma $, if
necessary we may assume that $ F$ takes the origin to the origin.
Now, we intend to use the arguments from
\cite{CoupetPinchuk&Sukhov-1996} to show that $ H_{2m} ( z_1,
\bar{z}_1) \equiv | z_1|^2 $. Since $F$ extends holomorphically
near the origin, we can write
\begin{equation}
\Re \big (F_2(z) \big) + |F_1(z)|^2 = h(z) \big( \Re z_2 +
H_{2m}(z_1, \bar{z}_1)\big) \label{Q1}
\end{equation}
in a neighbourhood of the origin where $h$ is a real analytic
function near the origin and $ h(0,0) \neq 0$. Now let
\[
F_1(z_1,0) = \displaystyle \sum_{\mu \geq s} b_{\mu} z_1^{\mu}
\]
where $ s \geq 1$. Also, $ F_2(z) = \beta z_2 + o(|z|) $ where $
\beta \in \mathbf{R} \setminus \{0\} $. Indeed, applying Hopf's
lemma to the function $ \Re F_2$, we see that $  \partial (\Re
F_2)/ { \partial x_2 } (0) > 0 $ where $ x_2 = \Re z_2 $.
Furthermore, since $ F : \big( \partial D_{2, \infty}, (0,0) \big)
\rightarrow \big( \partial \Sigma, (0,0) \big)$ and $F$ extends
holomorphically near the origin, it takes the complex normal to
$\partial D_{2, \infty}$ at the origin to the complex normal to $
\partial \Sigma $ at the origin. In particular, $
\partial  F_2/ { \partial z_1 } \equiv 0 $. Now, setting $ z_2 = 0 $ in (\ref{Q1}) yields
\[
\Re \big( o(|z_1|) \big) + b_s \bar{b}_s |z_1|^{2 s} + o(|z_1|^{2
s}) = h(z_1,0) H_{2m}( z_1, \bar{z}_1).
\]
Since $ H_{2m }$ does not contain any harmonic terms, we have
\[
b_s \bar{b}_s |z_1|^{2 s} + o(|z_1|^{2 s}) = h(0,0) H_{2m}( z_1,
\bar{z}_1).
\]
Since the polynomial $ H_{2m}$ is homogeneous of degree $2m$, we
conclude that $ 2m = 2 s $ and $ H_{2m}(z_1, \bar{z}_1) =
|z_1|^{2m} $. It follows that $ \mathbf{B}^2 \simeq D_{2, \infty}
\simeq \big \{ (z_1, z_2) \in \mathbf{C}^2 : 2 \Re z_2 +
|z_1|^{2m} < 0 \big \} \simeq \tilde{D} = \big \{ (z_1,z_2) \in
\mathbf{C}^2 : |z_1|^{2m} + |z_2|^2 < 1 \big \} $. Let $ G :
\mathbf{B}^2 \rightarrow \tilde{D} $ be the biholomorphism which
in addition may be assumed to presume the origin. Since $
\mathbf{B}^2$ and $ \tilde{D}$ are both circular domains, it
follows that $G$ is linear. This forces that $ 2m =2$ and hence $
H_{2m}(z_1, \bar{z}_1) = |z_1|^2 $.

\medskip

\noindent But this exactly means that there exists a local
coordinate system in a neighbourhood of $ q^0 $ which takes the
point $ q^0 $ to the origin and the domain $ D_2 $ near the origin
can be written as
\[
\big \{( z_1, z_2) \in \mathbf{C}^2 : 2 \Re z_2 + |z_1|^2 + o(
|z_1|^2 + \Im z_2 ) < 0 \big \}
\]
This contradicts the fact that $ q^0 = (0,0)$ is a weakly
pseudoconvex point and proves the theorem.
\end{proof}

\noindent The local version of the preceding theorem can also be
recovered similarly:

\begin{thm} Let $ D_1$ and $ D_2$ be two domains in
$ \mathbf{C}^2$. Assume that $ \partial D_1$ is $ {C}^2$-smooth
strongly pseudoconvex in a neighbourhood of a point $ p^0 \in
\partial D_1$. Suppose that $ \partial D_2$ is $
{C}^{\infty}$-smooth weakly pseudoconvex of finite type in a
neighbourhood of a point $ q^0 \in \partial D_2$. Then there
cannot be a biholomorphism $f$ from $ D_1$ onto $D_2$ with the
property that $ q^0 $ belongs to the cluster set of $f$ at $ p^0$.
\end{thm}

\noindent If such a biholomorphism $f$ exists, the first step in
proving the above result is to establish that $f$ extends to
continuous mapping on a neighbourhood of $ p^0 $ in $
\overline{D}_1 $. This requires the fact that $f$ decreases the
distance to the boundary, which is a consequence of well known
estimates for the Kobayashi metric. This is the content of
proposition 2.1 of \cite{CoupetPinchuk&Sukhov-1996}. Scaling the
domain $ D_ 2 $ with respect a sequence converging normally to $
q^0 $ and proceeding as in the proof of theorem \ref{T1}, we
obtain that $ D_2 $ near the $ q^0 = (0,0) $ is given by
\[
\big \{( z_1, z_2) \in \mathbf{C}^2 : 2 \Re z_2 + |z_1|^2 + o(
|z_1|^2 + \Im z_2 ) < 0 \big \}
\]
which leads to a contradiction.

\medskip

\noindent Using similar ideas, the following theorem of
Coupet-Gaussier-Sukhov (\cite{CoupetGaussier&Sukhov-1999}) can
also be recovered:

\begin{thm} \label{T2} A strongly pseudoconvex domain with $ {C}^2$-smooth
boundary cannot be mapped biholomorphically onto a smoothly
bounded convex (but not strongly pseudoconvex) domain of finite
type in $ \mathbf{C}^n (n > 1)$.
\end{thm}

\begin{proof} Suppose there exists a biholomorphism $f$
from a $ C^2 $-smooth strongly pseudoconvex domain $ D_1$ onto
a smoothly bounded convex but not strongly pseudoconvex domain $
D_2$ of finite type in $ \mathbf{C}^n$. Let $ q^0 \in
\partial D_2$. By \cite{Yu-1992} there exists a local coordinate
system in a neighbourhood of $ q^0 $ taking $q^0$ to the origin
such that the domain $ D_2$ near origin can be written as
\[
\{ ('z,z_n) \in \mathbf{C}^n : 2 \Re z_n + P_0('z) + R(z) < 0 \}
\]
where $P_0$ is a nondegenerate weighted homogeneous polynomial of
degree $1$ with respect to the weights $ \mathcal{M}(\partial D,
0)$ and $ R $ denotes terms of degree at least two. Scaling the
domain $D_2 $ with respect to a sequence converging normally to
the point $ q^0$ as in the proof of theorem \ref{T1}, we infer
that $ D_{2,\infty}$ is biholomorphically equivalent to $
\mathbf{B}^n$. Here $ D_{2,\infty}$ denotes the limiting domain
obtained by scaling $D_2$:
\[
D_{2,\infty} = \{ ('z,z_n) \in \mathbf{C}^n : 2 \Re z_n + P_0('z)
< 0 \}.
\]
Then as before, we may assume that for any point $ (0, \iota a^0),
a^0 \in \mathbf{R} $ there exists a biholomorphism $ F$ from $
D_{2,\infty}$ onto the unbounded realization of the unit ball,
namely to
\[
\Sigma = \{ z \in \mathbf{C}^n : 2 \Re z_n + | 'z |^2 < 0 \}
\]
with the property that the cluster set of $F$ at the point $ ('0,
\iota a^0) \in \partial D_{2, \infty}$ contains a point $ \zeta^0
\in \partial \Sigma $ where $ \zeta^0 $ is different from the
point at infinity on $ \partial \Sigma $. Applying theorem 2.1 of
\cite{Coupet&Pinchuk-2001}, we get that $F$ extends
holomorphically to a neighbourhood $ U $ of $ ('0, \iota a^0)$. We
now claim that $F$ extends biholomorphically near $ ('0, \iota
a^0)$. To see this, denote by $ l = \{ ('0, \iota a) \in
\mathbf{C}^n : a \in \mathbf{R} \} \subset \partial D_{2, \infty}
$. Assume on the contrary that the Jacobian of $F$ vanishes
identically on $ U \cap l $. In that case the Jacobian of $F$
vanishes on the entire $ z_n $ - axis which intersects the domain.
This violates the fact that $F$ is injective on $ D_{2, \infty} $.
This contradiction completes the proof of the claim. Now, pick $
('0, \iota \tilde{a} ) \in U \cap l $ such that the Jacobian of
$F$ does not vanish at $ ('0, \iota \tilde{a} )$. Consequently,
$F$ extends biholomorphically past $ ('0, \iota \tilde{a} ) \in
\partial D_{2, \infty}$. Since $ D_{2, \infty} $ is invariant under
translations in the imaginary $ z_n$- direction, we can assume
that $('0, \iota \tilde{a}) = ('0,0) $. This shows that $F$
extends biholomorphically to a neighbourhood of the origin. In
addition, by composing with an automorphism of $ \Sigma $, if
necessary we may assume that $ F$ maps the origin to the origin.
Since the Levi-form is preserved under local biholomorphisms
around a boundary point, this forces that $ ('0,0) \in \partial
D_{2, \infty} $ be strongly pseudoconvex. This contradiction
proves the theorem.
\end{proof}

\noindent A local version of this theorem also holds:

\begin{thm}
Let $ D_1$ and $ D_2$ be domains in $ \mathbf{C}^n$. Assume that $
\partial D_1$ is $ {C}^2$-smooth strongly pseudoconvex in a
neighbourhood of a point $ p^0 \in \partial D_1$ and that $
\partial D_2$ is $ {C}^{\infty}$-smooth convex (but not strongly
pseudoconvex) of finite type in a neighbourhood of a point $ q^0
\in \partial D_2$. Then there cannot be a biholomorphism $f$ from
$ D_1$ onto $D_2$ with the property that $ q^0 $ belongs to the
cluster set of $f$ at $ p^0$.
\end{thm}

\noindent Suppose that such a $f$ exists. Then proposition 2.1 of
\cite{CoupetPinchuk&Sukhov-1996} implies that $f$ extends
continuously near $ p^0 $. Now, scaling the domain $ D_2 $ with
respect to a sequence converging normally to $ q^0 $ and arguing
as in the proof of theorem \ref{T2}, we obtain that the domain $
D_2 $ is strongly psequdoconvex near $ q^0 $, thereby leading to a
contradiction.

\medskip

\noindent Next, we apply the invariance property of the function
$h$ to deduce the biholomorphic inequivalence of a bounded domain
with smooth and piecewise smooth boundaries.

\medskip

\noindent A domain $ D \subset \mathbf{C}^n $ is said to have
piecewise $ C^r $-smooth boundary if there are real-valued
functions $ \rho_i, i = 1, \ldots, m$, in a neighbourhood $V$ of the
closure of $D$  satisfying the following conditions:
\begin{itemize}

\item $ \partial D \subset \displaystyle\bigcup_{i=1}^m \{ z \in V
: \rho_i(z) = 0 \}$
\item For every subset $ \{ i_1, \ldots, i_k \} \subset \{1, \ldots, m\}$
the form $ \partial \rho_{i_1} \wedge \ldots \wedge \partial
\rho_{i_k} \neq 0$ on the intersection $ \displaystyle\bigcap
_{i=1}^m \{ z \in V : \rho_i(z) = 0 \}$

\end{itemize}
The latter condition means that $ \{ z \in V : \rho_i(z) = 0 \}$
is a $ {C}^r$-smooth hypersurface in general position.

\begin{thm} Let $ D_1$ and $ D_2$ be two bounded
domains in $\mathbf{C}^n (n > 1)$, $D_1$ having strongly
pseudoconvex boundary of class $ {C}^2$ and $ D_2$ having
piecewise $ {C}^2$-smooth, but not smooth boundary. Then $ D_1$
and $D_2$ are biholomorphically inequivalent.
\end{thm}

\begin{proof} Suppose there exists a biholomorphism $f : D_1 \rightarrow D_2 $.
Then the strong pseudoconvexity of $ D_1$ implies that $ D_2$ is
pseudoconvex. Then from the results of \cite{Pinchuk-1980} there
follows the existence of a point $ q^0 \in \partial D_2$ (for
definiteness we shall assume $ q^0$ is the origin) and $
{C}^2$-smooth real-valued functions $ \rho_i, i=1, \ldots, k \ ( 2
\leq k \leq n)$ defined in some neighbourhood $U$ of the origin
such that:

\begin{enumerate}

\item [(i)] $ ('0,0) \in \partial D_2 \cap \{z \in U : \rho_i(z) =
0, i=1, 2, \ldots, k \} $

\medskip

\item [(ii)] $ D_2 \cap U = \{ z \in U : \rho_i(z) < 0,
i=1, 2, \ldots, k \} $

\medskip

\item [(iii)] $ \bar{\partial} \rho_1 \wedge \ldots \wedge
\bar{\partial} \rho_k(z) \neq 0 $ for $ z \in U$,

\item [(iv)] for some $A > 0$ the function $ \rho= \displaystyle
\sum _{i=1}^k \rho_i + A \displaystyle \sum_{i=1}^k \rho_i^2$ is
strictly plurisubharmonic in $U$ and $D_2 \cap U \subset \{ z \in
U : \rho(z) < 0 \} $.

\end{enumerate}

\noindent Scaling the domain $D_2 $ as in \cite{Pinchuk-1982} with
respect to a sequence converging to the point $ q^0 $ and using
the Fridman's invariant as before, we infer that the limiting
domain $ D_{2, \infty} $ is biholomorphic to $ \mathbf{B}^n$. It
turns out (for details see \cite{Pinchuk-1982}) that $ D_{2,
\infty} $ is biholomorphic to a product of balls $
\mathbf{B}^{n_1} \times \ldots \times \mathbf{B}^{n_k} $ where $
\mathbf{B}^{n_j} = \big \{ z \in \mathbf{C}^{n_j} : |z| < 1 \big
\} $ and $ k \geq 2 $. This shows that $ \mathbf{B}^n $ is
biholomorphic to a product of balls which is a contradiction.
Hence the result follows.
\end{proof}

\medskip

\noindent In general, it is very difficult to provide an explicit
expression for the Kobayashi distance between two points of a
given domain. the most that can be done is to describe the
Kobayashi distance in terms of Euclidean parameters such as the
distance to the boundary. For instance, it is known that if $D$ is
a bounded strongly pseudoconvex domain with $ C^2 $-smooth
boundary, then for all $ R > 0 $, there exists positive constants
$ c_i$ and $C_i$ ($i=1, 2$) depending only on $ R$ and $D$ such that
for each $ q$ in $ D$ sufficiently close to $ \partial D$
\[
P \left( q; c_1 \ d(q, \partial D), c_2 \ {d(q, \partial D)}^{1/2}
\right) \subset B_D(q, r) \subset P \left( q; C_1 \ d(q, \partial
D), C_2 \ {d(q, \partial D)}^{1/2} \right)
\]
where $ P(q; r_1,r_2) $ denotes the polydisc centered at $q$ with
radius $ r_1$ in the complex normal direction and the radius $
r_2$ in each complex tangential direction.

\medskip

\noindent These estimates on the unit ball in $ \mathbf{C}^n$ can
be obtained by direct computation using the explicit formula for
the Kobayashi distance on $ \mathbf{B}^n$. For a strongly
pseudoconvex domain $D$, one has to work with suitable ellipsoids
tangent to $\partial D$ to get these results.

\medskip

\noindent This approach will not work for weakly pseudoconvex
domains of finite type in $ \mathbf{C}^2$ because of
non-availability of a suitable model domain. Instead, we make use
of the rescaling technique given in \cite{Catlin-1989}.

\begin{prop} Let $ D$ be a bounded weakly pseudoconvex
domain of finite type in $ \mathbf{C}^2 $ with $ C^{\infty}$-
smooth boundary. Then for all $ R > 0$, there exist positive
constants $ C_1$ and $C_2$ depending only on $ R$ and $ D $ such
that for each $ q$ in $ D $ sufficiently close to $ \partial D$
\[
Q \big( q, C_1 \ d(q, \partial D ) \big) \subset B_{D} (q, R)
\subset Q \big( q, C_2 \ d(q, \partial D) \big)
\]
where $ Q (q, r) $ denotes the Catlin's bidisc.
\end{prop}

\begin{proof} Firstly, observe that for each $q$ in $ D $, we can always
find constants $ C_i (q, R), i = 1,2 $ such that
\[
Q \big( q, C_1 (q, R) \ d(q, \partial D) \big) \subset B_{D} (q,
R) \subset Q \big( q, C_2(q, R)  \ d(q, \partial D ) \big)
\]
by virtue of the following fact: the Kobayashi ball $ B_{D} (q,
R)$ and the pseudo-bidisc $ Q \big(q, d(q, \partial D) \big) $ are
both open sets. The above proposition asserts that the constants
can be chosen independent of $q$.

\noindent To establish that $ Q \big( q, C_1 \ d(q, \partial D)
\big) \subset B_{D} (q, R) $ for some uniform constant $ C_1 $,
suppose that this is not true. Then there exists a sequence of
points $ \{ q^j \} \subset D $ converging to some point $q^0 \in
\partial D $ and a sequence of positive real numbers $ C_j
\rightarrow 0 $ as $ j \rightarrow \infty $ such that
\[
Q \left( q^j, C_j \ d(q^j, \partial D ) \right) \nsubseteq B_{D}
(q^j, R).
\]
Assume that $ q^0 =(0,0) $ without loss of generality. Let $ D =
\{ \rho (z, \bar{z}) < 0 \} $ where $ \rho $ is a smooth defining
function for $ \partial D $ which has the form
\[
\rho(z, \bar{z}) = 2 \Re z_2 + H_{2m}(z_1,\bar{z}_1 ) + o \big(
|z| ^{2m} + \Im z_2 \big)
\]
near the origin, with $H_{2m}$ a homogeneous subharmonic
polynomial of degree $ 2m \geq 2 $ in $ z_1$ and $\bar{z}_1 $
which does not contain any harmonic terms. Pick points $ \zeta^j
\in \partial D$ be closest to $ q^j$. To be more precise, $
\zeta^j = q^j + (0, \epsilon_j), \ \epsilon_j > 0 $. Choose points
$ p^j \in Q \big( q^j, C_j \ d(q^j, \partial D) \big ) $ such that
$ d_{D} ( p^j, q^j ) = R $. We now scale the domain $ D $ with
respect to the base point $ q^0 = (0,0) $ and the sequence $ \{
q^j \} $. Recall that
\[
\Delta_{\zeta^j}^{\epsilon_j} \circ \phi^{\zeta^j} \Big ( Q \big(
q^j, C_j \ d(q^j, \partial D) \big) \Big ) =
\]
\[
\Delta_{\zeta^j}^{\epsilon_j} \circ \phi^{\zeta^j} \circ
(\phi^{q^j})^{-1} \left( \Delta \Big(0, \tau \big( q^j, C_j \
d(q^j,\partial D) \big) \Big) \times \Delta \big(0, C_j \ d(q^j,
\partial D) \big) \right) =
\]
\[
\left \lbrace w: |w_1| < \frac{ \tau \big( q^j, C_j \ d(q^j,
\partial D ) \big) } { \tau( \zeta^j, \epsilon_j) },
\left| {\epsilon_j}^{-1} \Big( \displaystyle\sum_{l=1}^m
\alpha^{l,j} w_1^l \Big) + d^0(\zeta^j) w_2 + 1 \right| <
\frac{C_j d(q^j, \partial D ) d^0(q^j) } { \epsilon_j } \right
\rbrace
\]
where
\[
\alpha^{l,j} = \left (d^l(\zeta^j) - d^l(q^j) \right ){
\tau(\zeta^j, \epsilon_j) }^l
\]
Among other things, the following claims were proved in
\cite{Catlin-1989}.
\begin{itemize}

\item $ \tau \big( q^j, C_j \  d(q^j,
\partial D ) \big)  \lesssim C_j \tau( \zeta^j, \epsilon_j),$

\medskip

\item $ \big| d^l(q^j) \big| \lesssim \epsilon^j \big( \tau (
\zeta^j, \epsilon_j) \big)^{-l}, $

\medskip

\item $ \big| d^l(\zeta^j) \big| \lesssim \epsilon^j \big( \tau(
\zeta^j, \epsilon_j) \big)^{-l}, $

\medskip

\item $ d^0 (\zeta^j) \approx 1 $ and $ d^0 (q^j) \approx 1. $

\medskip

\end{itemize}
Moreover, $ d( q^j, \partial D) \approx \epsilon_j $. These
estimates together with the explicit description of the set
\[
\Delta_{\zeta^j}^{\epsilon_j} \circ \phi^{\zeta^j} \Big( Q \big(
q^j, C_j \ d(q^j, \partial D) \big) \Big)
\]
show that the sequence $ \left \lbrace
\Delta_{\zeta^j}^{\epsilon_j} \circ \phi^{\zeta^j} (p^j) \right
\rbrace  $ is bounded. Since $ C_j \rightarrow 0$, it follows that
\[
\Delta_{\zeta^j}^{\epsilon_j} \circ \phi^{\zeta^j}  (p^j)
\rightarrow (0,-1)
\]
as $ j \rightarrow \infty $. The scaled maps $
\Delta_{\zeta^j}^{\epsilon_j} \circ \phi^{\zeta^j} $ are
biholomorphisms and hence Kobayashi isometries from $ D $ onto the
dilated domains $ D^j := \big \{ z \in \mathbf{C}^2 : \rho \circ
\big( \phi^{\zeta^j} \big)^{-1} \circ \big(
\Delta_{\zeta^j}^{\epsilon_j} \big)^{-1}(z) < 0 \big \} $.
Therefore,
\begin{equation*}
d_{D^j} \left ( \Delta_{\zeta^j}^{\epsilon_j} \circ \phi^{\zeta^j}
(p^j), \Delta_{\zeta^j}^{\epsilon_j} \circ \phi^{\zeta^j} (q^j)
\right) = R  \label{R1}
\end{equation*}
which exactly means that
\[
d_{D^j} \left ( \Delta_{\zeta^j}^{\epsilon_j} \circ \phi^{\zeta^j}
(p^j), \Big(0,{-1}/{ d^0(\zeta^j)} \Big) \right) = R.
\]
Note that $ \left (0, {-1}/{d^0(\zeta^j)} \right) \rightarrow
(0,-1) $ as $ j \rightarrow \infty$. Applying lemma \ref{K3}, we
get after letting $ j \rightarrow \infty $ that
\[
d_{D_{\infty}} \big( (0,-1), (0,-1) \big) = R
\]
where $ D_{\infty}$ is the limiting domain obtained by scaling $D$
with respect to the base point $q^0$ and the sequence $ \{ q^j
\}$. This contradiction proves one part of the desired estimate.

\medskip

\noindent We now show that $ B_{D} (q, R) \subset Q \big( q, C_2 \
d(q, \partial D ) \big) $ for some constant $ C_2 $ independent of
$q$. Suppose this is not true. Then there exist a sequence of
points $ \{ q^j \} \subset D $ converging to some point $q^0 \in
\partial D $ and a sequence of positive real numbers $ C_j
\rightarrow \infty $ as $ j \rightarrow \infty $ such that
\[
Q \big( q^j, C_j \ d(q^j, \partial D ) \big) \nsupseteq B_{D}
(q^j, R).
\]
Assume that $ q^0 =(0,0) $ without loss of generality. Choose $
\zeta^j \in \partial D$ be closest to $ q^j$ and points $ p^j $ in
the complement of the closure of $ Q \big( q^j, C_j \ d(q^j,
\partial D ) \big) $ such that $ p^j \in B_{D} (q^j,
R) $. Then for $ \epsilon_j, D^j$ and $ D_{\infty}$ defined
analogously, we have
\[
d_{D^j} \left ( \Delta_{\zeta^j}^{\epsilon_j} \circ \phi^{\zeta^j}
(p^j), \Delta_{\zeta^j}^{\epsilon_j} \circ \phi^{\zeta^j} (q^j)
\right) < R.
\]
Said differently,
\begin{equation*}
d_{D^j} \left ( \Delta_{\zeta^j}^{\epsilon_j} \circ \phi^{\zeta^j}
(p^j), \Big(0, {-1}/ {d^0(\zeta^j)} \Big) \right) < R. \label{R2}
\end{equation*}
Fix $ \epsilon > 0 $ arbitrarily small and write $ \big (0,
-1/d^0(\zeta^j) \big) = z^j $ and $ (0,-1) = z^0 $ for brevity. It
now follows that
\begin{eqnarray*}
d_{D^j} \left (\Delta_{\zeta^j}^{\epsilon_j} \circ \phi^{\zeta^j}
(p^j), z^0 \right) & \leq & d_{D^j}
\left(\Delta_{\zeta^j}^{\epsilon_j} \circ \phi^{\zeta^j} (p^j),
z^j \right) + d_{D^j} \left ( z^j, z^0 \right) \nonumber \\
& \leq & R  + \epsilon
\end{eqnarray*}
for all $j$ large, where the second inequality uses the following
consequence of lemma \ref{K3}
\[
\displaystyle\lim_{j \rightarrow \infty} d_{D^j} \left ( z^j, z^0
\right) = d_{D_{\infty}} (z^0,z^0) = 0.
\]
It now follows from lemma \ref{E1} that
\[
 \Delta_{\zeta^j}^{\epsilon_j} \circ \phi^{\zeta^j}
(p^j) \in B_{D^j}(z^0, R + \epsilon) \subset B_{D_{\infty}} (z^0,
R + 2 \epsilon)
\]
for all $j$ sufficiently large. As a consequence, for all $j$
large
\begin{equation}
d_{D_{\infty}} \left(  \Delta_{\zeta^j}^{\epsilon_j} \circ
\phi^{\zeta^j} (p^j), z^0 \right)  < R + 2 \epsilon  \label{k14}
\end{equation}
On the other hand,
\begin{equation*}
\left | \left ( \Delta_{\zeta^j}^{\epsilon_j} \circ \phi^{\zeta^j}
(p^j) \right)_1 \right| \geq  \frac{ \tau \big( q^j, C_j \ d(q^j,
\partial D) \big) } { \tau ( \zeta^j, \epsilon_j) } \gtrsim C_j.
\end{equation*}
Since $ C_j \rightarrow \infty $ as $ j \rightarrow \infty $, we
conclude that
\[
\left| \left ( \Delta_{\zeta^j}^{\epsilon_j} \circ \phi^{\zeta^j}
(p^j) \right)_1 \right| \rightarrow + \infty
\]
as $ j \rightarrow \infty $. Consequently,
\[
\left|  \Delta_{\zeta^j}^{\epsilon_j} \circ \phi^{\zeta^j} (p^j) -
(0,-1) \right| \rightarrow + \infty
\]
as $ j \rightarrow \infty $. This is not possible in view of $
(\ref{k14}) $. Hence the result.
\end{proof}

\noindent It turns out that proceeding exactly as in the weakly
pseudoconvex case using the scaling methods described in section
$6$ one can prove the following statement:

\begin{prop} Let $ D$ be a smoothly bounded convex domain of
finite type in $ \mathbf{C}^n $. Then for all $ R > 0$, there
exists positive constants $ C_1$ and $C_2$ depending only on $ R$
and $ D $ such that for each $ q$ in $ D $ sufficiently close to $
\partial D$
\[
P \big( q, C_1 \ d(q, \partial D ) \big) \subset B_{D} (q, R)
\subset P \big( q, C_2 \ d(q, \partial D) \big)
\]
where $ P (q, r) $ is as in (\ref{t3}).
\end{prop}


\section{Isometries of the Kobayashi metric on strongly pseudoconvex
domains}

\noindent Let $ D$ be a bounded strongly pseudoconvex  domain in $
\mathbf{C}^n$ with a $ {C}^2$-smooth defining function $\rho^0$ 
defined on a neighbourhood $ U$ of the closure of $D$.
Choose $ \rho^k \in {C}^2(U; \mathbf{R}) $ such that $ \rho^k$
converges to $\rho^0$ in the $ {C}^2$-topology, i.e.,
\[
\| \rho^k - \rho^0  \| _{ C^2 (U)} = \displaystyle \sum _{j=1}^n
\displaystyle \sup _{ z \in U, | \alpha| \leq 2 } \Big| D^{\alpha}
\big( \rho^k_j(z) - \rho^0_j(z) \big) \Big| \rightarrow 0
\]
as $ k \rightarrow \infty $. Setting $ D^k = \big \{ z \in U :
\rho^k(z) < 0 \big \} $, observe that there exist uniform positive
constants $ C_1$ and $C_2$ such that for every $ z $ in $
N_{\epsilon} (\partial D)$, an $ \epsilon $-neighbourhood of $
\partial D $,
\[
C_1 d(z, \partial D) \leq d(z, \partial D^k) \leq C_2 d(z,
\partial D ).
\]
where $ d( \cdot, \partial D ) $ is the Euclidean distance to the
boundary.

\begin{prop} \label{F2} Let $ D $ and $ D^k$'s be as described above.
Let $ q^0$ and $ q^1$ be two distinct boundary points of $ D$.
Then for a suitable uniform constant $C$ and for all $k$ large,
\[
d_{D^k} (a,b) \geq  - (1/2) \log d(a, \partial D^k) - (1/2) \log
d(b, \partial D^k) - C
\]
whenever $ a, b \in D, a $ is near $q^0$ and $b$ is near $ q^1$.
\end{prop}

\noindent In case $ a, b $ are close to the same boundary point,
the upper estimate due to Forstneric and Rosay also remains stable
under small $ C^2 $ perturbations.

\begin{prop} \label{F3} Let $ D $ and $ D^k$'s be as described above
and let $ q^0 \in \partial D$. Then there exists a neighbourhood
$V = V (q^0)$ and a uniform constant $ C > 0$ such that
\begin{multline*}
d_{D^k}(a,b)  \leq  - (1/2) \log {d(a,\partial D^k)} +
 (1/2) \log \big( d(a,\partial D^k) + |a - b| \big) \\
 + (1/2) \log \big( d(b,\partial D^k) + |a - b| \big) - (1/2) \log {d(b,\partial D^k)} + C
\end{multline*}
for all $a, b \in   V \cap D$ and for all $ k$ large.
\end{prop}

\noindent The proof of proposition \ref{F3} will be given after
lemma \ref{F18} and does not use any of the arguments presented
till then. We proceed with the proof of proposition \ref{F2}
first. This will need several steps. To begin with, we need to
localise the Kobayashi metric. The following assertion about local
peak functions will be useful.

\begin{lem} \label{F4} There exist uniform positive constants $
C_1,C_2$ and $r$ such that for every $ A^k \in \partial D^k$, if $
\zeta \in B(A^k,r) \cap \partial D^k$ then there exists a local
peak function $ P_{\zeta}$ at $ \zeta, P_{\zeta} \in
\mathcal{O}(D^k_1) \cap {C}(\overline{D^k_1})$ where $ D^k_1 =
B(A^k ,r) \cap D^k$ is a neighbourhood of $ A^k$ of uniform size
with the property that
\[
C_1 | 1- P_{\zeta}(z)| \leq |z- \zeta| \leq C_2 \sqrt{|1-
P_{\zeta}(z)|}
\]
for all $z \in D^k_1$.
\end{lem}

\begin{proof} Note that the local peak function at $ A^k \in
\partial D^k$ is given by
\[
P_{A^k}(z)= \displaystyle \sum_{i=1}^n \frac{\partial
\rho^k}{\partial z_i} (A^k) ( z_i -A^k_i) + \frac{1}{2}
\displaystyle \sum_{i,j=1}^n  \frac{\partial^2 \rho^k}{\partial
z_i
\partial z_j} (A^k)( z_i -A^k_i)( z_j -A^k_j).
\]
Since $ \rho^k $ converges to $ \rho^0 $ in the $ C^2 $-topology, 
the neighbourhoods $ D^k_1$ can be chosen uniformly and
the computations in \cite{Forstneric&Rosay-1987} remain valid.
\end{proof}

\noindent The following simple consequence of the Schwarz lemma on
the unit disc will be needed:

\begin{lem}\label{F5} For each $ \lambda > 0$ there exists a positive
constant $ C_{\lambda} $ such that for every $ \epsilon > 0$ and
every holomorphic disc $ g: \Delta \rightarrow \Delta $ satisfying
$ | 1-g(0) | \leq \epsilon $ we have
\[
| 1-g(\zeta)| \leq \lambda
\]
whenever $ |\zeta| \leq 1-C_{\lambda} \epsilon $.
\end{lem}

\noindent Using this it is possible to control the behaviour of
analytic discs in $ D^k_1 $ near $ A^k$.

\begin{lem}\label{F6} Let $ D, D^k $ and $ D^k_1$ be
as described before. For $\eta > 0$, there exist uniform positive
constants $C$ and $\delta$ such that for every holomorphic disc $
h : \Delta \rightarrow D^k_1 $ satisfying $ d( h^k(0), \partial
D^k) < \delta $,
\[
\big| h^k(z) - h^k(0) \big| < \eta
\]
whenever $ | z | \leq 1- C d \big( h^k(0), \partial D^k \big) $.
\end{lem}

\begin{proof} For every holomorphic mapping $ h^k : \Delta \rightarrow
D^k_1$, choose points $ B^k$ on $ \partial D^k \cap \partial
D^k_1$ closest to $ h^k(0)$. It follows from lemma \ref{F4} there
exist uniform positive constants $ C_1 $ and $ C_2 $ such that the
local peak function at $B^k$ satisfies
\begin{eqnarray*}
\big| 1- P_{B^k} \big(h^k(0) \big) \big| & \leq & | h^k(0) - B^k |
/ { C_1}, \ \big| h^k(z) - B^k \big| \leq C_2 \sqrt{ \big| 1-
P_{B^k} \big(h^k(z) \big) \big|}
\end{eqnarray*}
for all $ z \in \Delta $. Now, applying lemma \ref{F5} to the
holomorphic discs $ P_{B^k} \circ h^k : \Delta \rightarrow \Delta
$ for $ \lambda = \big( { \eta}/{2 C_2} \big)^2$ yields
\[
\big| 1- P_{B^k} \big(h^k(z) \big) \big| \leq \lambda
\]
whenever $ |z| \leq 1 -  C_{\lambda}/ {C_1} { \big| h^k(0) - B^k
\big|} $. Thus, for all such $z$
\begin{eqnarray*}
\big | h^k(z) - h^k(0) \big| & \leq & \big| h^k(z) - B^k  \big| +
\big| B^k - h^k(0) \big| \\
& \leq & C_2 \sqrt{ \big| 1- P_{B^k} \big(h^k(z) \big) \big|} +
\big| h^k(0) - B^k \big| \\
& \leq & C_2 \sqrt{\lambda} + \big | h^k(0) - B^k \big| < \eta
\end{eqnarray*}
provided $ \big | h^k(0) - B^k \big| < {\eta}/{2} $ for all $k$.
Choosing $ \delta= {\eta}/{2} $ and $ C= C_{\lambda}/{C_1}$, we
get the required result.
\end{proof}

\begin{lem}\label{F7} For each fixed boundary point $A^k$ of $
D^k$, there exists a uniform positive constant $C$ and a
neighbourhood of $A^k$ of uniform size $ D^k_3 $ which is
compactly contained in $ D^k_1 $ such that
\[
F^K_{D^k} (z, X) \geq \big( 1- C d(z, \partial D^k) \big)
F^K_{D^k_1} (z,X)
\]
for every $ z$ in $ D^k_3$ and $ X$ a tangent vector at $z$.
\end{lem}

\begin{proof} Since $ D$ is bounded, we may assume that the diameter of
$D^k$ is at most one for all $k$. Fix $ A^k \in \partial D^k$.
Choose $ d \in (0,1) $ and a neighbourhood $ D^k_2 $ of $A^k$ of
uniform size which is relatively compact in $ D^k_1$ such that
\[
B(z,d) \cap D^k \subset D^k_1 \qquad
\]
whenever $ z \in D^k_2 $. Given $ \epsilon > 0 $ sufficiently
small, let $ \alpha^k = \alpha^k(\epsilon) $ be the largest number
in $[0,1]$ such that $ \big| h^k(z) - h^k(0) \big| \leq d$
whenever $ h^k : \Delta \rightarrow D^k$ is a holomorphic mapping
with $ h^k(0) \in D^k_2$, $ d \big( h^k(0), \partial D^k \big)
\leq \epsilon $ and $ |z| \leq \alpha^k $. Observe that the
Schwarz lemma implies that $ d \leq \alpha^k$. Further, note that
$ h^k \big( \Delta(0, \alpha^k) \big) \subset  B \big( h^k(0),d
\big) \cap D^k $ and $ B \big( h^k(0), d \big) \cap D^k  \subset
D^k_1 $ for all $k$. Now, applying lemma \ref{F6} to the mapping $
h^k|_{\Delta(0, \alpha^k)}$ yields uniform positive constants $
\tilde{C}$ and $\delta$ such that if $ d( h^k(0), \partial D^k) <
\delta$ then $ \big| h^k(z) - h^k(0) \big | \leq d/2 $ whenever $
|z| \leq \alpha^k \big( 1- \tilde{C} d(h^k(0), \partial D^k) \big)
$. In particular,
\[
\big| h^k(z) - h^k(0) \big| \leq d/2 
\]
for $|z| \leq \alpha^k - \tilde{C} \epsilon$. Let
\[
d^k = \sup \big \lbrace \big| h^k(z) - h^k(0) \big| : |z|=
\alpha^k, h^k \in \mathcal{O} \left( \Delta(0, \alpha^k);D^k_1
\right), h^k(0) \in D^k_2 \ \mbox{and} \ d \big(h^k(0),
\partial D^k \big) < \delta \big \rbrace.
\]
Observe that $ d^k \leq d$ for all $k$. Furthermore, by Hadamard's
three circle lemma, the function
\[ \log \displaystyle\sup_{|z| = r} \big| h^k(z) - h^k(0) \big|
\]
is a convex function of $\log r$. Therefore,
\begin{eqnarray*}
\frac{ \log d/2} { \log ( \alpha^k - \tilde{C} \epsilon ) } \leq
\frac{ \log d^k} { \log \alpha^k } \leq \frac{ \log d } { \log
\alpha^k }
\end{eqnarray*}
which implies that
\[
\left| \frac{\log \alpha^k}{ \log d} \right| \leq \left| \frac{
\log( \alpha^k - \tilde{C} \epsilon) }{ \log d/2 } \right|.
\]
Since
\[
\displaystyle \lim_{\epsilon \rightarrow 0} \big (1/ {\epsilon}
\big) \log \big( 1- \tilde{C} \epsilon / {\alpha^k}  \big)  = -
\tilde{C}/{ \alpha^k} \geq - \tilde{C}/d,
\]
it follows that
\[
1- \alpha^k \leq | \log \alpha^k | \leq { 2 \tilde{C} |\log d|
\epsilon} /(d \log 2).
\]
This shows that for all $k$
\[
\alpha^k \geq 1- { 2 \tilde{C} | \log d | \epsilon }/ (d \log 2).
\]
Set $ C = { 2 \tilde{C} |\log d| } /(d \log 2)$ and $ D^k_3 = B
\big( A^k, \delta \big) \cap D^k_2$. To summarise, we have proved
the following: There exists a uniform positive constant $C$ and
uniform neighbourhoods $ D^k_3 $ compactly contained in $ D^k_1$
such that for every holomorphic mapping $ h^k : \Delta \rightarrow
D^k$ satisfying $ h^k(0) \in D^k_3$,
\[
h^k \big( \Delta \left(0, 1- C \epsilon \right) \big) \subset
D^k_1.
\]
The required result then follows from the definition of the
infinitesimal Kobayashi metric.
\end{proof}

\begin{lem}\label{F8} Let $ D, D^k , D^k_1$ and $ D^k_3$ be
as described before. For each $ A^k \in \partial D^k$, there exist
neighbourhoods $ D^k_4 $ compactly contained in $ D^k_3$ such that
for a suitable uniform constant $C' > 0$,
\[
d_{D^k} \big( z, D^k \setminus D^k_1 \big) \geq - (1/2) \log d(z,
\partial D^k) - C'
\]
for all $ z$ in $ D^k_4$.
\end{lem}

\begin{proof} Fix $ A^k \in \partial D^k$. Choose a sufficiently
small neighbourhood $ D^k_4 $ of $A^k$ of uniform size, $ D^k_4 $
compactly contained in $ D^k_3$ such that for each point $ z \in
D^k_4 \cap D^k$, the point $ A^z \in \partial D^k_4 \cap
\partial D^k$ closest to $z$ satisfies
\[
\displaystyle \lim_{ z' \in D^k \setminus D^k_3 } \big| A^z - z'
\big| \geq \delta_0 > 0
\]
where $ \delta_0 $ can be chosen to be independent of the point $
z $ and the index $k$. This is possible since $ D^k_3$ and $
D^k_1$ are of uniform size. Pick $ z \in D^k_4 $ and $ w^k \in D^k
\setminus D^k_1$. Let $ \sigma^k$ be any arbitrary piecewise $
C^1$-smooth curve in $ D^k$ joining $ w^k$ and $z$. As we
travel along $ \sigma^k$ there is a first point $ \tilde{z}^k$ on
the curve with $ \tilde{z}^k \in \partial D^k_3 \cap D^k $. Let $
\gamma^k$ be the subcurve of $ \sigma^k$ joining $ \tilde{z}^k$
and $z$. Choose points $ B^k \in \partial D^k_4 \cap \partial D^k$
closest to the point $z$. It is already known that the local peak
function $P_{B^k}$ at $B^k$ is holomorphic on $ D^k_1$. Hence,
\begin{eqnarray*}
\lambda^k(t) = P_{B^k}(\gamma^k(t))
\end{eqnarray*}
is well-defined. Set $ \xi^k(t) = | \lambda^k(t)|$. The explicit
expression for $ P_{B^k} $ shows that it has no zeroes inside $
\overline{D^k_1}$. Hence, $ \xi^k$ is $ {C}^1$-smooth. It
follows from lemma \ref{F7} that
\begin{eqnarray*}
\int_0^1 F^K_{D^k} \big( \gamma^k(t), \dot{\gamma}^k(t) \big) dt
\geq \int_0^1 \left( 1- C d ( \gamma^k(t), \partial D^k) \right)
F^K_{D^k_1} \big( \gamma^k(t), \dot{\gamma}^k(t) \big) dt
\end{eqnarray*}
Since holomorphic maps decrease the Kobayashi metric, we have
\begin{eqnarray*}
\int_0^1 \left( 1- C d ( \gamma^k(t), \partial D^k) \right)
F^K_{D^k_1} \big( \gamma^k(t), \dot{\gamma}^k(t) \big) dt & \geq &
\int_0^1 \left( 1- C d ( \gamma^k(t), \partial D^k) \right)
F^K_{\Delta} \big( \lambda^k(t), \dot{\lambda}^k(t) \big) dt \\
& = & \int_0^1 \left( 1- C d ( \gamma^k(t), \partial D^k)
\right) \frac{ | \dot{\lambda}^k(t) |}{ 1- | \lambda^k(t) |^2 } dt \\
& = & \int_0^1 \left( 1- C d ( \gamma^k(t), \partial D^k)
\right) \frac{ \dot{\xi}^k(t) }{ 1- {\xi^k(t)} ^2} dt \\
& \geq & \int_0^1 \left( 1- C d ( \gamma^k(t), \partial D^k)
\right) \frac{ \dot{\xi}^k(t) }{ 2 ( 1- \xi^k(t)) } dt
\end{eqnarray*}
Also, by lemma \ref{F4}, for every $ t \in [0,1]$
\begin{eqnarray*}
1- C d \big( \gamma^k(t), \partial D^k \big) & \geq & 1 - C \big |
\gamma^k(t) - B^k \big | \\
& \geq & 1 - C C_2 \sqrt{ 1 - \big |P_{B^k}(\gamma^k(t)) \big| } \\
& = & 1- C C_2 \sqrt{1- \xi^k(t)}
\end{eqnarray*}
Therefore,
\begin{eqnarray*}
\int_0^1 F^K_{D^k} ( \gamma^k(t), \dot{\gamma}^k(t) ) dt  & \geq &
\int_0^1 \Big( 1- C C_2 \sqrt{1- \xi^k(t)} \Big)
\frac{ \dot{\xi}^k (t) }{ 2 \big( 1- \xi^k(t) \big) } dt \\
& \geq & - (1/2) \log \left( {1- \xi^k(1)} \right) +  (1/2) \log
\left( {1- \xi^k(0)} \right) + \big ( C C_2 /2 \big) \int_0^1
\frac{dx}{\sqrt{x}}
\end{eqnarray*}
Observe that
\begin{eqnarray*}
1- \xi^k(1) & = & 1- | P_{B^k} (\gamma^k(1))| \leq |\gamma^k(1)
- B^k| /{C_1} =  d \big(\gamma^k(1), \partial D^k \big)/ {C_1} \\
1- \xi^k(0) & = & 1- | P_{B^k} (\gamma^k(0))| \geq {| \gamma^k(0)
- B^k|}^2 \geq {\delta_0}^2
\end{eqnarray*}
This implies that
\begin{eqnarray*}
\int_0^1 F^K_{D^k} ( \gamma^k(t), \dot{\gamma}^k(t) ) dt & \geq &
\frac{1}{2}\log \left( \frac{C_1}{d(\gamma^k(1), \partial D^k)}
\right) - \frac{1}{2}\log \left( \frac{1}{{\delta_0}^2}
\right) + \frac{C C_2}{2} \int_0^1 \frac{dx}{\sqrt{x}} \\
& \geq & - (1/2) \log d(z, \partial D^k) - C'
\end{eqnarray*}
where
\[
C' = (1/2) \log C_1 + \log( \delta_0 ) + \big( C C_2 /2 \big)
\int_0^1 \frac{dx}{\sqrt{x}}.
\]
This shows that
\[
\int_0^1 F^K_{D^k} ( \sigma^k(t), \dot{\sigma}^k(t) ) dt \geq -
(1/2) \log d(z, \partial D^k) - C'.
\]
Taking the infimum over all admissible curves, we get the required
result.
\end{proof}

\noindent \textit{Proof of Proposition \ref{F2}:} For each $k \ge 1$ 
choose points $X^k$ and $Y^k$ on $ \partial D^k$ closest
to $a$ and $b$ respectively. Each path in $ D^k$ joining $a $ and
$b$ must exit from neighbourhoods of $ X^k $ and $ Y^k$. Hence the
result is immediate from lemma \ref{F8}. \qed

\medskip

\begin{lem} \label{F18} Let $ D$ and $ D^k$ be as described before. There exists a
uniform positive constant $ C$ such that
\[
d_{D^k}(x,x^k) \leq  \frac{1}{2} \log \left( \frac{ d(x^k,
\partial D^k)} { d(x, \partial D^k)} \right)+ C
\]
where $ x \in D$ is sufficiently close to $ \partial D$ and the
points $  x^k  \in D^k $ lie on the same normal to $ D^k $ as $x$.
\end{lem}

\begin{proof} This follows by integrating the infinitesimal metric along the straight line path
in $ D^k$ joining $ x $ and $ x^k $.
\end{proof}

\noindent \textit{Proof of Proposition \ref{F3}:} For each $ k$,
we denote by $ n^k (q)$ the unit outward normal to $ \partial D^k$
at $ q \in \partial D^k$. Fix $ C^k \in \partial D^k$ temporarily.
Since $ \rho^k$ converges to $ \rho^0$ in ${C}^2$-topology,
there exists a $0 < R \ll 1 $ such that

\begin{enumerate}

\item[(i)] $ \big| n^k(q) - n^k(C^k) \big| < 1/8 \ \mbox{for} \ q
\in \partial D^k \cap B( C^k, R), $

\medskip

\item[(ii)] $ z - \delta n^k(q) \in D^k$ and $ d \big( z - \delta
n^k(q), \partial D^k \big) > 3 \delta/4 $ \ whenever $ z \in D^k
\cap B( C^k, R), q \in \partial D^k \cap B ( C^k, 8 R)$ and $
\delta \leq 2 R $.

\end{enumerate}

\noindent Note that $R$ is independent of $C^k$ and $k$ for $k$
large. It is already known that there exist uniform positive
constants $ C_1$ and $C_2$ such that for every $ z $ in $
N_{\epsilon} (\partial D)$
\[
C_1 d(z, \partial D) \leq d(z, \partial D^k) \leq
C_2 d(z, \partial D )
\]
for all $k$ large. Now, choose $ r \in (0, R/{4 C_2} )$ and fix
two points $ a, b \in D \cap B(q^0,r)$. Let $ A^k, B^k \in
\partial D^k$ be the uniquely determined points closest to $a $ and
$b$ respectively. Next, find $ \tilde{C}^k \in \partial D^k$ such
that $ a, b \in B( \tilde{C}^k, R)$ for all $k$ large. Setting
\begin{eqnarray*}
a^k  =  a - | a-b| n^k(A^k)\ \mbox{and} \ b^k  =  b - | a-b|
n^k(B^k),
\end{eqnarray*}
note that
\[
\big| A^k - \tilde{C}^k \big| \leq \big|A^k - a \big| + \big| a -
\tilde{C}^k \big| < R/4 + R  < 2 R.
\]
Similarly, $ B^k \in B( \tilde{C}^k, 2 R)$. It follows from (ii)
that $ a^k, b^k \in D^k$, $ d(a^k, \partial D^k) > (3/4) |a-b | $
and $ d(b^k, \partial D^k) > (3/4) |a-b | $. The triangle
inequality gives the following upper estimate:
\begin{eqnarray*}
d_{D^k}(a,b) & \leq &  d_{D^k}(a,a^k) + d_{D^k}(a^k,b^k) +
d_{D^k}(b^k,b)
\end{eqnarray*}
By lemma \ref{F18}, we conclude that there exists a uniform
positive constant $C$ such that
\begin{eqnarray*}
d_{D^k}(a, a^k ) \leq (1/2) \log\left( \frac{d(a^k,
\partial D^k)}{ d(a, \partial D^k)} \right) + C \\
d_{D^k}(b, b^k ) \leq (1/2) \log\left( \frac{d(b^k,
\partial D^k)}{ d(b, \partial D^k)} \right) + C
\end{eqnarray*}
Also, by construction we have
\begin{eqnarray*}
d(a^k, \partial D^k) = d(a, \partial D^k) + | a-b | , \ d(b^k,
\partial D^k) = d(b, \partial D^k) + | a-b |.
\end{eqnarray*}
It remains to estimate the term $ d_{D^k}(a^k,b^k)$. First, it
follows from (i) that $ | a^k -b^k | \leq (5/4) |a-b |$. Now,
consider the analytic discs $ \phi^k : \mathbf{C} \rightarrow
\mathbf{C}^n$ defined by $ \phi^k( \lambda)= a^k + \lambda ( b^k -
a^k )$. If $ |\lambda| < 3/5 $ (respectively $ | \lambda - 1| <
3/5 )$, we obtain
\begin{eqnarray*}
\big| \phi^k(\lambda) - a^k \big| & < & (3/5)(5/4) |a-b|= (3/4)
|a-b| < d( a^k, \partial D^k) \\
\big( \mbox{respectively}\ \big| \phi^k(\lambda) - b^k \big| & < &
(3/5)(5/4) |a-b| = (3/4) |a-b| < d( b^k,\partial D^k) \big)
\end{eqnarray*}
i.e., if $\Omega = \Delta(0, 3/5) \cup \Delta(1, 3/5) $, then each
$ \phi^k $ is a holomorphic mapping of $ \Omega $ into $ D^k$.
Since holomorphic maps decrease the Kobayashi metric, it follows
that for all $k$
\[
 d_{D^k}(a^k,b^k) \leq d_{\Omega}(0,1).
\]
This completes the proof of the proposition. \qed

\medskip

\begin{thm} \label{G1}
Let $D_1, D_2$ be two bounded strongly pseudoconvex domains in
$\mbf C^n$ with $C^2$-smooth boundaries. Let $D_1^k, D_2^k $
for $ k \geq 1$ be two sequences of domains that converge to $D_1,
D_2$ respectively in the $C^2$ topology. Suppose that $f^k :
(D^k_1, d_{D^k_1}) \rightarrow (D^k_2, d_{D^k_2})$ is a $C^0$-smooth 
isometry for each $k \geq 1$ and that there is a point $p^1
\in D_1$ such that some subsequence $\{f^{k_j}(p^1)\}$ converges
to a point $p^2 \in D_2$. Then the sequence $ \{f^k\}$ admits a
subsequence that converges uniformly on compact sets of $ D_1$ to
a continuous mapping $ f : D_1 \rightarrow D_2$. Moreover, $f :
(D_1, d_{D_1}) \rightarrow (D_2, d_{D_2}) $ is a ${C}^0$-isometry. 
Further, assuming that each $f^k \in {C}^1 (D^k_1)$,
there is a uniform constant $C > 0$ with the property that:
\[
\vert f^{k_j}(p) - f^{k_j}(q) \vert \le C \vert p - q \vert^{1/2}
\]
for all $p, q \in D_1$.
\end{thm}

\noindent Our purpose is to prove that the family $ \{ f^{k_j} \}$
is `uniformly H\"{o}lderian' up to the boundary, i.e., every
mapping extends as a H\"{o}lder continuous one up to the boundary
with the H\"{o}lder constant \textit{independent} of $k$. In
particular, this implies the normality of this family on $
\overline{D}_1 $.

\begin{proof} The proof involves several steps.

\medskip

\noindent \textbf{Step I:} Let $ \{ K_{\nu} \}_{\nu =1}^{\infty}$
be an increasing sequence of relatively compact subsets of $ D_1$
that exhausts $ D_1$. Fix a pair $ K_{\nu_0} $ compactly contained
in $ K_{\nu_0 + 1}$ such that $ p^1 \in K_1 $ and write $ K_1 =
K_{\nu_0}$ and $ K_2 = K_{\nu_0 + 1}$ for brevity. Let $
\omega(K_1)$ be a neighbourhood of $ K_1$ such that $ \omega(K_1)
\subset K_2$. Since $ \{ D^k_1 \} $ converges to $ D_1$, it
follows that $ K_1 \subset \omega(K_1) \subset K_2 $ which in turn
is relatively compact in $ D^k_1 $ for all $k$ large. We show that
the sequence $ \{ f^k \}$ is equicontinuous at each point of $
\omega(K_1)$.

\medskip

\noindent For each $ x^1 \in \omega(K_1)$ fixed, there exists a $r
> 0$ such that $ B( x^1,r) $ is compactly contained in $ \omega(K_1)$.
The distance decreasing property of the Kobayashi metric together
with its explicit form on $ B(x^1, r) $ gives
\begin{eqnarray}
d_{D^k_2 } \big( f^k(x^1), f^k(y^1) \big) = d_{D^k_1}(x^1,y^1)
\leq d_{B( x^1,r)} (x^1,y^1) \leq  | x^1 -y^1 | /c \label{10.3}
\end{eqnarray}
for all $k$ large, $ y^1 \in B( x^1, r) $ and a uniform constant $
c> 0 $. Since $ D_2$ is bounded, $ D^k_2 $ are compactly contained
in $ B(0,R) $ for some $ R > 0 $ and for all $k$ large .
Consequently,
\begin{eqnarray*}
d_{B(0,R)} \big( f^k(x^1), f^k(y^1) \big) \leq d_{D^k_2 } \big(
f^k(x^1), f^k(y^1) \big)
\end{eqnarray*}
for all $k$ large. Again using the explicit form of the metric on
$ B(0, R) $ gives
\[
\big | f^k(x^1) - f^k(y^1) \big| \lesssim | x^1 - y^1 |
\]
for $ y^1 \in B(x^1, r) $. This shows that $ \{ f^k \} $ is
equicontinuous at each point if $ \omega(K_1) $. The diagonal
subsequence still denoted by the same symbols then converges
uniformly on compact subsets of $ D_ 1 $ to a limit mapping $ f :
D_1 \rightarrow \overline{D}_2 $ which is continuous.

\medskip

\noindent \textbf{Step II:} The proof of lemma \ref{K1} shows that
for $ i =1,2$,
\begin{equation*}
F^K_{D^k_i} ( \cdot, \cdot) \rightarrow F^K_{D_i} (\cdot, \cdot)
\end{equation*}
uniformly on compact sets of $ D_i \times \mathbf{C}^n$ as $ k
\rightarrow \infty $.

\medskip

\noindent \textbf{Step III:} Consequently,
\begin{equation}
d_{D^k_i} ( \cdot, \cdot) \rightarrow d_{D_i}( \cdot, \cdot)
\label{t4}
\end{equation}
as $ k \rightarrow \infty $.

\medskip

\noindent To show this, for $ p, q \in D_1 $, firstly note that
\begin{equation}
\displaystyle\limsup_{k \rightarrow \infty} d_{D^k_1} (p,q) \leq
d_{D_1} (p,q)  \label{11.1}
\end{equation}
which follows exactly as in lemma \ref{D1}. For the converse, 
pick $ R \gg d_{D_1}(p,q)$ and consider $ B_{D^k_1}(p, R)$. 
Note that $ q \in B_{D^k_1}(p, R)$
and $ B_{D^k_1}(p, R) $ is compactly contained in $ D^k_1 $ for
all $k$ large. We claim that $ B_{D^k_1}(p, R) \subset D_1$ for
all $k$ large. Let $ w^k \in B_{D^k_1}(p, R)$. Observe that the
only case to be investigated is when $w^k$ is very close to the
boundary of $ D^k_1$. In that case, it follows from lemma \ref{F8}
that
\begin{eqnarray*}
R \gg d_{D_1}(p,q) + \epsilon \geq d_{D^k_1} (p,q) \geq - (1/2)
\log d(w^k, \partial D^k_1) - C
\end{eqnarray*}
for some uniform positive constant $ C $ and $ \epsilon > 0 $
given. As a result,
\begin{equation}
d( w^k, \partial D^k_1) > 1/ {e^{ 2(C+R)}} > 0 \label{11.2}
\end{equation}
Let $ w^k $ converge to some point $ w \in N_{\epsilon} ( \partial
D_1 ) $. In fact, since $ D^k_1$ converge to $D_1$ in the $ C^2$-topology, 
$ w \in \overline{D}_1 $. Further, it follows from
(\ref{11.2}) that $ d(w,
\partial D_1) > 0 $ and consequently that $ w^k \in D_1$ for all
$k $ large. Since the points $ w ^k$ were arbitrarily chosen, this
completes the proof of the claim. Now, it follows from the
distance-decreasing property of the Kobayashi metric that
\begin{eqnarray*}
d_{D_1} (p,q) \leq d_{B_{D^k_1}(p, R)} (p,q)
\end{eqnarray*}
Applying lemma \ref{D2} to the domain $ D^k_1$ with $ B_{D^k_1}
(p, R)$ as the subdomain $ D'$ yields
\begin{equation*}
d_{B_{D^k_1} (p, R)} (p,q) \leq \frac{d_{D^k_1} (p,q)} { \tanh
\left( R/2 - d_{D^k_1} (p,q ) \right)}.
\end{equation*}
Using (\ref{11.1}), we have
\begin{equation*}
d_{D_1} (p,q) \leq \frac{d_{D^k_1} ( p, q)} { \tanh \big( R/2 -
d_{D_1}(p,q) -\epsilon \big) }.
\end{equation*}
Letting $ R \rightarrow \infty$ yields
\begin{equation}
d_{D_1}(p,q) \leq d_{D^k_1} ( p,q) + \epsilon.  \label{11.4}
\end{equation}
for all $k $ large. Putting together (\ref{11.1}) and
(\ref{11.4}), we get
\begin{equation*}
\displaystyle \lim_{k \rightarrow \infty}d_{ D^k_1 }(p,q) =
d_{D_1}(p,q).
\end{equation*}
The same proof works for the domain $ D_2 $ and is therefore
omitted.

\medskip

\noindent \textbf{Step IV:} We now show that $ d_{D_1} (x^1,y^1) =
d_{D_2} \big( f(x^1),f(y^1) \big)$ for all $ x^1, y^1 \in
\Omega_1$ where $ \Omega_1 = \big \{ q^1 \in D_1 : f(q^1) \in D_2
\big \}$. Note that $ \Omega_1$ is non-empty since $ p^1 \in D_1$.
Let $ x^1, y^1 \in \Omega_1$. It is known that
\begin{equation*}
d_{D^k_1} (x^1,y^1) = d_{D^k_2} \big( f^k(x^1),f^k(y^1) \big).
\end{equation*}
for all $k$ and thus by (\ref{t4}), it remains to show that the
right side above converges to $ d_{D_2} \big( f(x^1),f(y^1) \big)$
as $ k \rightarrow \infty $. For this note that
\[
\big | d_{D^k_2} \big(f^k(x^1),f^k(y^1) \big) - d_{D^k_2} \big
(f(x^1),f(y^1) \big) \big| \leq   d_{D^k_2} \big( f^k(x^1),f(x^1)
\big) + d_{D^k_2} \big( f(y^1),f^k(y^1) \big)
\]
by the triangle inequality. Since $ f^k(x^1) \rightarrow f(x^1) $
and the domains $ D^k_2 $ converge to $ D_2 $, it follows that
there is a small ball $ B \big( f(x^1), r \big) $ around $ f(x^1)
$ which contains $ f^k(x^1) $ for all large $k $ and which is
contained in $ D^k_2 $ for all large $k$, where $ r > 0 $ is
independent of $k$. Thus
\[
d_{D^k_2} \big( f^k(x^1),f(x^1) \big) \leq C \big| f^k(x^1) -
f(x^1) \big |
\]
for some uniform constant $ C $. The same argument works for
showing that $ d_{D^k_2} \big( f(y^1),f^k(y^1) \big) $ is small.
So to verify the claim, it is enough to prove that $ d_{D^k_2}
\big (f(x^1),f(y^1) \big) $ converges to $ d_{D_2} \big(
f(x^1),f(y^1) \big) $. But this is immediate from (\ref{t4}).

\medskip

\noindent \textbf{Step V:} The limit map $ f $ is a surjection
onto $ D_2 $. Firstly, we need to show that $ f(D_1) \subset D_2$.
Indeed, $ \Omega_1= D_1$. If $ q^0 \in \partial \Omega_1 \cap
D_1$, choose a sequence $ q^j \in \Omega_1$ that converges to $
q^0$. It follows from Step IV that
\begin{equation*}
d_{D_1}(q^j, p^1) = d_{D_2} \big( f(q^j), f(p^1) \big)
\end{equation*}
for all $j$. Since $ q^0 \in \partial \Omega_1$, the sequence $ \{
f(q^j) \} $ converges to a point on $ \partial D_2$ and as $D_2$
is complete in the Kobayashi distance, the right hand side above
becomes unbounded. However, the left hand side remains bounded
again because of completeness of $ D_1$. This contradiction shows
that $ \Omega_1= D_1$ which exactly means that $ f(D_1) \subset
D_2.$ The above observation coupled with Step IV forces that
\begin{equation*}
d_{D_1} (x^1,y^1) = d_{D_2} \big( f(x^1),f(y^1) \big)
\end{equation*}
for all $ x^1, y^1 \in D_1$. To establish the surjectivity of $f$,
consider any point $ q^0 \in \partial \big( f( D_1) \big) \cap
D_2$ and choose a sequence $ q^j \in f(D_1)$ that converges to
$q^0$. Let $ \{ p^j \} $ be sequence of points in $ D_1$ be such
that $ f(p^j)= q^j$. Then for all $j$ and for all $ x^1 \in D_1$,
\begin{equation}
d_{D_1}(x^1, p^j) = d_{D_2} \big( f(x^1), f(p^j) \big)
\label{11.6}
\end{equation}
There are two cases to be considered. After passing to a
subsequence, if needed,

\begin{enumerate}

\item[(i)] $  p^j \rightarrow x^0 \in \partial D_1$,

\item [(ii)] $ p^j \rightarrow  x^2 \in D_1$ as $ j \rightarrow
\infty $.

\end{enumerate}

\medskip

\noindent In case (i), observe that the right hand side remains
bounded because of the completeness of $ D_2$. Moreover, since $
D_1$ is complete in the Kobayashi metric, the left hand side in
(\ref{11.6}) becomes unbounded. This contradiction shows that $
f(D_1) = D_2$.

\medskip

\noindent For (ii), firstly, the continuity of the mapping $f$
implies that the sequence $ \{ f(p^j) \} $ converges to the point
$ f(x^2)$. Therefore, we must have $ f(x^2) = q^0$. Consider the
mappings $ (f^k)^{-1} : D^k_2 \rightarrow D^k_1$. Arguing as
before, we infer that the sequence $ \{ (f^k)^{-1} \}$ admits a
subsequence that converges uniformly on compact sets of $D_2$ to a
continuous mapping $ g: D_2 \rightarrow \overline{D}_1$. Then $ g
\circ f \equiv id_{D_1}$. Therefore,
\[
x^2 = g \circ f(x^2) = g(q^0) = \displaystyle \lim_{k \rightarrow
\infty} (f^k)^{-1} (q^0).
\]
Hence the sequence $ \{ (f^k)^{-1} (q^0) \} $ is compactly
contained in $ D_1$. Now, repeating the earlier argument for $
\{(f^k)^{-1}\}$, it follows that $g : D_2 \rightarrow D_1$ and $ f
\circ g \equiv id_{D_2} $. In particular, $f$ is surjective.

\medskip

\noindent This also shows that the limit mapping $f$ is a $ {C}^0$-isometry. 
We note that so far we did not need the isometries to
be $ {C}^1$-smooth. Now, assume that each $f^k \in {C}^1
(D^k_1)$.

\medskip

\noindent The following assertion gives a `uniform local
hyperbolicity' of the family $ \{ D^k_2 \}$ near $ \partial D_2$.

\medskip

\noindent \textbf{Step VI:} There exists an $ \epsilon
> 0$ and a uniform positive constant $C$ such that for any $k \ge 1$ and 
$q^2 \in D^k_2 \cap N_{\epsilon} (\partial D_2)$ and $v \in
\mathbf{C}^n$ we have
\[
F^K_{D^k_2} (q^2, v) \geq C \frac{ | v| }{\sqrt{d(q^2, \partial
D^k_2)}}.
\]
\noindent Fix a point $ q^2$ near $ \partial D_2 $ and let $v \in
\mathbf{C}^n \setminus \{0\}$. Let $ R^k $ be a sequence of
positive real numbers and $ \psi^k $ a sequence of holomorphic
mappings, $ \psi^k \in \mathcal{O} ( \Delta; D^k_2) \cap {C}^0(
\overline{\Delta}; \mathbf{C}^n ) $ satisfying $ \psi^k(0) = q^2$
and $ (\psi^k)'(0) = R_k v $. Choose $ \zeta^k \in \partial D^k_2
$ with $ | q^2 - \zeta^k | = d ( q^2,
\partial D^k_2) $ and let $ \rho^k_2$ be a ${C}^2$ -- smooth defining
function for $ D^k_2$ which is strictly plurisubharmonic in a
neighbourhood of $ \overline{ D^k_2}$. Then, for $ \eta > 0$
sufficiently small,
\[
\tilde{\rho}_2^k (z) = \rho_2^k (z) - \eta | z - \zeta^k | ^2
\]
is a strictly plurisubharmonic on a neighbourhood of closure of $
D^k_2$. Therefore,
\begin{eqnarray*}
\tilde{\rho}_2^k (q^2) = \tilde{\rho}_2^k \circ \psi^k(0) \leq
\frac{1}{2 \pi} \int_0^{2 \pi} \rho^k_2 \big( \psi^k( e ^{ \iota
t} ) \big) - \eta \left | \psi^k(e ^{\iota t}) - \zeta^k \right
|^2 dt \leq - \frac{\eta}{2 \pi} \int_0^{2 \pi} \left | \psi^k(e
^{\iota t}) - \zeta^k \right |^2 dt.
\end{eqnarray*}
Applying the Cauchy integral formula it follows that
\begin{eqnarray*}
\big | (\psi^k)'(0) \big |^2 & = &  \displaystyle\sum_{j=1}^n
\left |\frac{1}{2 \pi \iota} \int_{0}^{2 \pi}
\frac{\psi^k_j(e^{\iota t}) - \zeta^k_j}{e^{2 \iota t}} dt \right
|^2   \leq \frac{1}{2 \pi} \int_0^{2 \pi} \left | \psi^k(e^{\iota
t}) - \zeta^k \right |^2 \; dt.
\end{eqnarray*}
As a consequence,
\[
\big| (\psi^k)'(0) \big| ^2 \leq | q - \zeta^k |^2 - \rho^k_2
(q^2) / {\eta}.
\]
Now, since $ \rho^k_2 \rightarrow \rho^0_2$ in the $ {C}^2$-topology, 
it follows that there exists a uniform constant $ M > 0$ such that
\[
| \rho^k_2(z) | \leq M d(z, \partial D^k_2)
\]
for all $k$ large and for all $ z$ sufficiently close to $
\partial D^k_2$. Then
\[
\big|(\psi^k)'(0) \big|^2 \leq \big( 1 + M/{\eta} \big) d(q^2,
\partial D^k_2).
\]
The required result then follows from the definition of the
infinitesimal metric.

\medskip

\noindent \textbf{Step VII:} There exist uniform positive constants
$A$ and $B$ such that
\[
A \ d(x, \partial D^k_1) \leq d( f^k(x), \partial D^k_2) \leq B \
d(x, \partial D^k_1)
\]
for all $x$ sufficiently close to $ \partial D_1$.

\medskip

\noindent To see this, fix $ y \in D_1 $ sufficiently close to $
\partial D_1 $ and $ x^0 \in \partial D_1 $. Let $ \{x^j \}$ be a sequence of points in $ D_1$
converging to $ x^0$. Then for each $j $ large there is an index $
k_0 $ such that $ x^j \in D^k_1 $ for $ k \geq k_0 $. For such a
fixed large $j$, apply lemma \ref{F8} and lemma \ref{F18} to get
uniform positive constants $ C_1, C_2 $ such that
\begin{equation} \label{P}
-(1/2) \log  {d(x^j, \partial D^k_1)}  - C_2 \leq d_{D^k_1}(x^j,y)
\leq -(1/2) \log  {d(x^j, \partial D^k_1)} + C_1
\end{equation}
for all $ k \geq k_0 $. Note that $ \{ f^k(y) \} $ is compactly
contained in $ D_2$ from Step V. From \cite{Seshadri&Verma-2006}
it follows that each $f^k$ is continuous up to $ \overline{
D^k_1}$ and $ f^k( \partial D^k_1) \subset \partial D^k_2$. Since
$ D^k_2$ converges to $D_2$, it follows that $ \{ x^j \} \subset
N_{\epsilon} (\partial D^k_1)$ for all $k$ large. Again using
lemma \ref{F8} and lemma \ref{F18} yields
\begin{eqnarray} \label{Q}
-(1/2) \log {d \big(f^k(x^j), \partial D^k_2 \big)} - C_3 & \leq &
d_{D^k_2} \big(f^k(y),f^k(x^j) \big) \\ & \leq & -(1/2) log \Big(
{d \big(f^k(x^j), \partial D^k_2 \big)} d \big( f^k(y),
\partial D^k_2 \big) \Big) + C_4  \nonumber
\end{eqnarray}
for $ k \geq k_0 $ and $ C_3, C_4 > 0 $ uniform in $k$. Since $
d_{D^k_1}(y, x^j)= d_{D^k_2} \big(f^k(y),f^k(x^j) \big) $, it
follows from (\ref{P}) and (\ref{Q}) that
\[
A \ d(x, \partial D^k_1) \leq d( f^k(x), \partial D^k_2) \leq B \
d(x, \partial D^k_1)
\]
for some uniform $ A, B > 0 $ and $x$ sufficiently close to $
\partial D_1$.

\medskip

\noindent \textbf{Step VIII:} For $ \zeta^0 \in \partial D_1$
fixed, there exists a neighbourhood $V$ of $ \zeta^0$ such that
for all $ p, q \in V \cap D_1$ and for all $k$ large
\[
\big| f^k(p) - f^k(q) \big| \leq C | p-q | ^{1/2}
\]
where $C$ is a constant independent of $p, q \in V \cap D_1$.

\medskip

\noindent Firstly, observe that given $\eta > 0$ there exist  $
\eta', \eta'' > 0$ such that
\begin{itemize}

\item $ f \big( \{ z \in D_1 : d(z, \partial D_1) < \eta \} \big)
\subset \{ w \in D_2 : d(w, \partial D_2) < \eta' \} $

\medskip

\item $ f^{-1} \big(\{ w \in D_2 : d(w, \partial D_2) < \eta\}
\big) \subset \{ z \in D_1 : d(z,\partial D_1) < \eta'' \}$

\medskip

\item $ \displaystyle \lim_{ \eta \rightarrow 0} \eta'= 0 $ and $
\displaystyle \lim_{ \eta \rightarrow 0} \eta''= 0 $.

\end{itemize}

\noindent These are consequences of the completeness of the
Kobayashi distance on strongly pseudoconvex domains. Now, from
Step VI there exists a uniform positive constant $C$ such that
\begin{eqnarray*}
F^K_{D^k_2} \Big( f^k(x), df^k(x)v \Big) \geq C \frac{ \big
|df^k(x)v \big|}{\sqrt{d \big(f^k(x), \partial D^k_2 \big)}}
\end{eqnarray*}
for all $k$ large and for all $x$ sufficiently close to $
\partial D_1$. Using the fact that each $f^k$ is an isometry
we get
\begin{eqnarray*}
C \frac{ \big|df^k(x)v \big|}{\sqrt{d \big(f^k(x), \partial D^k_2
\big)}} \leq F^K_{D^k_2} \big( f^k(x), df^k(x)v \big) =
F^K_{D^k_1}(x,v) \leq \frac{|v|}{d(x, \partial D^k_1)}
\end{eqnarray*}
for all tangent vectors $v$ at $x$. This implies that
\begin{eqnarray*}
\big| df^k(x)v \big| \leq C \frac{|v| \sqrt{d \big(f^k(x),
\partial D^k_2 \big)}}{d(x, \partial D^k_1)}
\end{eqnarray*}
By Step VII, we have
\begin{eqnarray}
\big|df^k(x)v \big| \leq C \frac{|v|}{\sqrt{d(x, \partial D^k_1)}}
\label{12.3}
\end{eqnarray}
Next, observe that since $ \partial D_1$ is $ {C}^2$-smooth,
there exists a $0 < R \ll 1$ such that
\[
z - \delta n(\zeta) \in D_1 \ \mbox{and} \ d ( z - \delta
n(\zeta), \partial D_1) > 3 \delta/4
\]
for all $ z \in D_1 \cap B( \zeta^0, R), \zeta \in
\partial D_1 \cap B( \zeta^0,8R)$ and $ \delta \leq 2R $.
Here, $n(\zeta)$ denotes the unit outward normal to $
\partial D_1$ at $ \zeta$.

\medskip

\noindent Now, fix two points $ p, q \in D_1 \cap B( \zeta^0, R)$.
Choose $ p^0, q^0 \in \partial D_1$ closest to $p $ and $q$
respectively. Set $ p' = p - |p-q| n(p^0)$ and $ q'= q - |p-q|
n(q^0)$. Let $ \gamma $ be the union of three segments: the first
one being the straight line path joining $p$ and $p'$ along the
inward normal to $p^0$, the second one being a straight line path
joining $p'$ and $q'$ and finally the third path is taken to be
the straight line path joining $ q'$ and $q^0$ along the inward
normal to the point $q^0$. Integrating (\ref{12.3}) along this
polygonal path, we get for all $k$ large
\[
\big| f^k(p) - f^k(q) \big| \leq C | p-q | ^{1/2}
\]
uniformly for all $p, q \in D_1 \cap B(\zeta^0, R)$. This
completes Step VIII.

\medskip

\noindent Since Step VIII together with the $ C^1 $-smoothness
assumption on the isometries immediately gives the result of
theorem \ref{G1}, we may conclude.
\end{proof}

\noindent As a corollary of theorem \ref{G1}, we obtain:

\begin{cor} Let $ D_1$ and $D_2$ be two bounded strongly pseudoconvex
domains in $ \mathbf{C}^n$ with $ {C}^2$-smooth boundaries. Let
$ f^k : (D_1, d_{D_1}) \rightarrow (D_2, d_{D_2}) $ be a sequence
of $ C^0$-isometries between $ D_1$ and $ D_2$. Suppose that
there exists a point $ p^1 \in D_1$ such that some subsequence of
$ \{ f^k(p^1) \}$ converges to a point of $ D_2$, then the
sequence $ \{f^k \}$ admits a subsequence $ \{ f^{k_j} \}$ that
converges uniformly on compact sets of $ D_1 $ to a continuous
mapping $ f : D_1 \rightarrow D_2$. Moreover, $f : (D_1, d_{D_1})
\rightarrow (D_2, d_{D_2}) $ is a ${C}^0$-isometry. Further,
assuming that each $f^k \in {C}^1 (D^k_1)$, there is a uniform
constant $C > 0$ with the property that:
\[
\vert f^{k_j}(p) - f^{k_j}(q) \vert \le C \vert p - q \vert^{1/2}
\]
for all $p, q \in D_1$.
\end{cor}



\begin{thebibliography}{99}

\bibitem{Aladro-1989} G. Aladro, \textit{The comparability of the Kobayashi approach
region and the admissible approach region}, Illinois J. Math.,
\textbf{33}(1989), 42--63

\bibitem{Alexander&Wermer} H. Alexander and J. Wermer,
\textit{Several complex variables and Banach algebras}, Third
edition, Graduate Texts in Mathematics, 35, Springer-Verlag, New
York, 1998.

\bibitem{Berteloot-1994} F. Berteloot, \textit{Characterisation of
models in $ \mathbf{C}^2$ by their auotmorphism groups}, Internat.
J. Math., \textbf{5}(1994), 619-634

\bibitem{Berteloot&Coeure-1991} F. Berteloot and G.
Coeur\'{e}, \textit{Domaines de $\mathbf{B}^2$, pseudoconvexes et
de type fini ayant un groupe non compact d'automorphismes}, Ann.
Inst. Fourier (Grenoble), \textbf{41}(1991), 77-86

\bibitem{Catlin-1989} D. Catlin, \textit{Estimates of invariant
metrics on pseudoconvex domains of dimension two}, Math. Z.,
\textbf{200}(1989), 429-466

\bibitem{Coupet-1992} B. Coupet, \textit{Uniform extendibility of
automorphisms}, Contemp. Math., \textbf{137}(1992), 177-183

\bibitem{CoupetGaussier&Sukhov-1999} B. Coupet, H. Gaussier and A. Sukhov, \textit{ Regularity
of CR maps between convex hypersurfaces of finite type}, Proc.
Amer. Math. Soc., \textbf{127}(1999), 3191--3200.

\bibitem{Coupet&Pinchuk-2001} B. Coupet and S. I. Pinchuk,
\textit{Holomorphic equivalence problem for weighted homogeneous
rigid domains in $\mathbf{C}^{n+1}$}, Complex analysis in modern
mathematics (Russian), FAZIS, (2001), 57--70

\bibitem{Coupet&Sukhov-1997} B. Coupet and A. Sukhov, \textit{On
CR mappings between pseudoconvex hypersurfaces of finite type in $
\mathbf{C}^2$}, Duke Math. J., \textbf{88}(1997), 381-304

\bibitem{CoupetPinchuk&Sukhov-1996} B. Coupet, S. Pinchuk and A.
Sukhov, \textit{On boundary rigidity and regularity of holomorphic
mappings}, Internat. J. Math., \textbf{7}(1996), 617-643

\bibitem{Coupet&Sukhov-1996} B. Coupet and A. Sukhov, \textit{On
the uniform extendibility of proper holomorphic mappings}, Complex
Variables Theory Appl., \textbf{28}(1996), 243-248

\bibitem {Fornaess&Sibony-1981} J. E. Fornaess and N. Sibony, \textit {Increasing
sequence of complex manifolds}, Math. Ann., \textbf{255}(1981),
351-360

\bibitem{Forstneric&Rosay-1987} F. Forstneric and J.P. Rosay,
\textit{Loocalisation of the Kobayashi metric and the boundary
continuity of proper holomorphic mappings}, Math. Ann.,
\textbf{279}(9187), 239-252

\bibitem{Fridman-1983} B. L. Fridman, \textit{Biholomorphic
invariants of a hyperbolic manifold and some applications}, Trans.
Amer. Math. Soc., \textbf{276}(1983), 685-698

\bibitem{Gaussier-1997} H. Gaussier, \textit{Characterization of
convex domains with noncompact automorphism group}, Michigan Math.
J., \textbf{44}(1997), 375--388

\bibitem {Graham-1975} Ian Graham, \textit{Boundary behaviour of
the Carath\'{e}odory and Kobayashi metrics on strongly
pseudonconvex domains in $ \mathbf{C}^n $ with smooth boundary},
Trans. Amer. Math. Soc., \textbf{207}(1975), 219-240

\bibitem{Greene&Krantz-1982} R. E. Greene and S. G. Krantz, \textit{Deformation of complex
structures, estimates for the $\bar \partial $ equation and
stability of the Bergman kernel}, Adv. in Math. \textbf{43}(1982),
no. 1, 1--86.

\bibitem{Greene&Krantz-1984} R. E. Greene and S. G. Krantz, \textit{Stability of the Carathéodory and Kobayashi metrics and
applications to biholomorphic mappings}, In: Complex analysis of
several variables (Madison, Wis., 1982), 77--93, Proc. Sympos.
Pure Math., \textbf{41}, Amer. Math. Soc., Providence, RI, 1984.

\bibitem{Herbort-2005} G. Herbort, \textit{Estimation on
invariant distances on the pseudoconvex domains of finite type in
dimension two}, Math. Z., \textbf{251}(2005), 673-703

\bibitem{Hoffman-1960} K. Hoffman, \textit{Minimal boundaries for
analytic polyhedra},  Rend. Circ. Mat. Palermo, \textbf{9}(1960),
147--160

\bibitem{Jarnicki&Pflug} M. Jarnicki and P. Pflug,
\textit{Invariant distances and metrics in Complex Analysis},
Walter de Gruyter and Co.

\bibitem{Kim&Krantz-2008} K.T. Kim and S. G. Krantz, \textit{A
Kobayashi metric version of Bun Wong's Theorem}, Complex Var.
Elliptic Equ., \textbf{54}(2009), 355--369

\bibitem{Kim&Ma-2003} K.T. Kim and D. Ma, \textit{Characterisation
of the Hilbert ball by its automorphisms}, J. Korean Math. Soc.,
\textbf{40}(2003), 503-516

\bibitem{Kim&Pagano-2001} K.T. Kim and A. Pagano, \textit{Normal
analytic polyhedra in $\mathbf{C}^2$ with a noncompact
automorphism group}, J. Geom. Anal., \textbf{11}(2001), 283-293

\bibitem{Krantz-1991} S. G. Krantz, \textit{Invariant metrics and the boundary behavior
of holomorphic functions on domains in $ \mathbf{C}^n$}, J. Geom.
Anal., \textbf{1}(1991), 71--97

\bibitem{Lempert-1981} L. Lempert, \textit{La m\'{e}trique
de Kobayashi et la repr\'{e}sentation des domaines sur la boule},
Bull. Soc. Math. France, \textbf{109}(1981), 427-474

\bibitem{Mcneal-1992} J. D. McNeal, \textit{Convex domains of finite type},
J. Funct. Anal., \textbf{108}(1992), 361--373

\bibitem{Mcneal-1994} J. D. McNeal, \textit{Estimates on the Bergman
kernels of convex domains}, Adv. Math., \textbf{109}(1994),
108--139

\bibitem{NagelStein&Wainger-1985} A. Nagel, E. M. Stein and S. B. Wainger, \textit{Balls and
metrics defined by vector fields I : Basic properties}, Acta Math.,
\textbf{155}(1985), 103--147.

\bibitem {Pinchuk-1980} S. I. Pinchuk, \textit{Holomorphic
inequivalence of certain classes of domains in $\mathbf{C}^n$},
Math. USSR Sb. (N.S.), \textbf{39}(1980), 61-86

\bibitem{Pinchuk-1982} S. I. Pinchuk, \textit{Homogeneous domains with piecewise smooth
boundaries}, Mat. Zametki, \textbf{32}(1982), 729-735

\bibitem{Rosay-1979} J-P. Rosay, \textit{Sur une
caract\`{e}risation de la boule parmi les domaines de
$\mathbf{C}^n$ par son groupe d'automorphismes}, Ann. Inst.
Fourier (Grenoble), \textbf{29}(1979), 91-97

\bibitem {Royden-1971} H. L. Royden, \textit {Remarks on the
Kobayashi metric}, Lecture Notes in math., \textbf{185}(1971),
125-137

\bibitem {Rudin} Walter Rudin, \textit{Function theory in the unit ball of $ \mathbf{C}^n
$}, Springer-Verlag.

\bibitem {Seshadri&Verma-2006} H. Seshadri and K. Verma, \textit{On
isometries of the Carath\'{e}odory and Kobayashi metrics on
strongly pseudoconvex domains}, Ann. Scuola Norm. Sup. Pisa Cl.
Sci.(5), \textbf{V}(2006), 393-417

\bibitem {Seshadri&Verma-2009} H. Seshadri and K. Verma, \textit{On
the compactness of isometry groups in complex analysis}, Complex
Var. Elliptic Equ., \textbf{54}(2009), 387--399

\bibitem {Venturini-1989} S. Venturini, \textit{Comparison between
the Kobayashi and Carath\'{e}odory distances on strongly
pseudoconvex bounded domains in $\mathbf{C}^n$}, Proc. Amer. Math.
Soc., \textbf{107}(1989), 725-730

\bibitem{Webster-1977} S. M. Webster, \textit{On the mapping problem for algebraic real
hypersurfaces}, Invent. Math, \textbf{43}(1977), 53-68

\bibitem{Wong-1977} B. Wong, \textit{Characterisation of the unit
ball in $ \mathbf{C}^n$ by its automorphism group}, Invent. Math.,
\textbf{41}(1977), 253-257

\bibitem{Yu-1992} J. Yu, \textit{Multitypes of convex domains},
Indiana Univ. Math. J., \textbf{41}(1992), 837-849

\bibitem{Yu-1995} J. Yu, \textit{Weighted boundary limits of the generalized Kobayashi-Royden
metrics on weakly pseudoconvex domains},  Trans. Amer. Math. Soc.,
\textbf{347}(1995), 587--614

\end{thebibliography}
\end{document}